%% file: arXiv.tex
\documentclass[aos,preprint]{imsart}
\usepackage[square,sort,comma,numbers]{natbib}
\usepackage{xr}

\RequirePackage[OT1]{fontenc}
\RequirePackage{amsthm,amsmath}
\RequirePackage[numbers]{natbib}

\usepackage{upgreek}

\usepackage{url}
\usepackage[colorlinks=True,citecolor=blue,urlcolor=blue,pagebackref=true,backref=true]
{hyperref}
\renewcommand*{\backrefalt}[4]{%
    \ifcase #1 \footnotesize{(Not cited.)}%
    \or        \footnotesize{(Cited on page~#2.)}%
    \else      \footnotesize{(Cited on pages~#2.)}%
    \fi}

\usepackage{xr}

\makeatletter
\newcommand*{\addFileDependency}[1]{
  \typeout{(#1)}
  \@addtofilelist{#1}
  \IfFileExists{#1}{}{\typeout{No file #1.}}
}
\makeatother

\usepackage{graphics}
\usepackage{graphicx}
\usepackage{psfrag}

\usepackage{color}

\usepackage{amssymb,bbm}

\usepackage[colorlinks]{hyperref}
\usepackage{nicefrac}

 \usepackage{tabularx}%

\usepackage{enumitem}
\usepackage{booktabs}
\usepackage{caption,subcaption}

\usepackage{mathtools}

\usepackage{etoc}

\input{final_macros}

\makeatletter
\renewcommand{\tableofcontents}{%
  \@starttoc{toc}%
}
\makeatother

\begin{document}


\begin{frontmatter}

\title{Singularity, Misspecification, and the Convergence Rate of EM}
\runtitle{Singular Models and Slow convergence of EM}


\begin{aug}
\author{\fnms{Raaz} \snm{Dwivedi}\thanksref{t1,m1}\ead[label=e1]{raaz.rsk@berkeley.edu}},
\author{\fnms{Nhat} \snm{Ho}\thanksref{t1,m1}\ead[label=e2]{minhnhat@berkeley.edu}},
\author{\fnms{Koulik} \snm{Khamaru}\thanksref{t1,m1}\ead[label=e3]{koulik@berkeley.edu}},\\
\author{\fnms{Martin J.} \snm{Wainwright}\thanksref{t2,m1,m2}\ead[label=e4]
{wainwrig@berkeley.edu}},
\author{\fnms{Michael I.} \snm{Jordan}\thanksref{t3,m1}\ead[label=e5]
{jordan@cs.berkeley.edu}}
\and
\author{\fnms{Bin} \snm{Yu}\thanksref{t4,m1}
\ead[label=e6]{binyu@berkeley.edu}
\ead[label=u1,url]{http://www.foo.com}}

\thankstext{t1}{Raaz Dwivedi, Nhat Ho,
  and Koulik Khamaru contributed equally to this work.}
\thankstext{t2}{Supported by Office of Naval Research grant
DOD ONR-N00014-18-1-2640 and National Science Foundation grant
NSF-DMS-1612948}
\thankstext{t3}{Supported by Army Research Office grant
W911NF-17-1-0304}
\thankstext{t4}{Supported by National Science Foundation grant  
NSF-DMS-1613002}
\runauthor{Dwivedi, Ho, Khamaru, Wainwright, Jordan and Yu}

\affiliation{University of California, Berkeley\thanksmark{m1} and
  Voleon Group\thanksmark{m2} }

\address{Raaz Dwivedi\\
Department of Electrical Engineering \\
and Computer Sciences\\
\printead{e1}
}

\address{Nhat Ho\\
Department of Electrical Engineering \\
and Computer Sciences\\
\printead{e2}}

\address{Koulik Khamaru\\
Department of Statistics\\
\printead{e3}
}

\address{Martin J. Wainwright\\
Department of Statistics\\
Department of Electrical Engineering \\
and Computer Sciences\\
\printead{e4}
}

\address{Michael I. Jordan\\
Department of Statistics\\
Department of Electrical Engineering \\
and Computer Sciences\\
\printead{e5}}

\address{Bin Yu\\
Department of Statistics\\
Department of Electrical Engineering \\
and Computer Sciences\\
\printead{e6}}
\end{aug}

\begin{abstract}
A line of recent work has analyzed the behavior of the
Expectation-Maximization (EM) algorithm in the well-specified setting,
in which the population likelihood is locally strongly concave around
its maximizing argument.  Examples include suitably separated Gaussian
mixture models and mixtures of linear regressions.  We consider
over-specified settings in which the number of fitted components is
larger than the number of components in the true distribution. Such
mis-specified settings can lead to singularity in the Fisher
information matrix, and moreover, the maximum likelihood estimator
based on $n$ i.i.d. samples in $d$ dimensions can have a non-standard
$\mathcal{O}((d/n)^{\frac{1}{4}})$ rate of convergence.  Focusing on
the simple setting of two-component mixtures fit to a $d$-dimensional
Gaussian distribution, we study the behavior of the EM algorithm both
when the mixture weights are different (unbalanced case), and are
equal (balanced case). Our analysis reveals a sharp distinction
between these two cases: in the former, the EM algorithm converges
geometrically to a point at Euclidean distance of
$\mathcal{O}((d/n)^{\frac{1}{2}})$ from the true parameter, whereas in
the latter case, the convergence rate is exponentially slower, and the
fixed point has a much lower $\mathcal{O}((d/n)^{\frac{1}{4}})$
accuracy.  Analysis of this singular case requires the introduction of
some novel techniques: in particular, we make use of a careful form of
localization in the associated empirical process, and develop a
recursive argument to progressively sharpen the statistical rate.
\end{abstract}



\end{frontmatter}

\setcounter{secnumdepth}{3}
\setcounter{tocdepth}{1}
\newpage
\begin{center}
\textbf{Organization}
\end{center}
\tableofcontents


\section{Introduction} 
\label{sec:introduction}

The growth in the size and scope of modern data sets has presented the
field of statistics with a number of challenges, one of them being how
to deal with various forms of heterogeneity. Mixture models provide a
principled approach to modeling heterogeneous collections of data
(that are usually assumed i.i.d.).  In practice, it is frequently the
case that the number of mixture components in the fitted model does
not match the number of mixture components in the data-generating
mechanism. It is known that such mismatch can lead to substantially
slower convergence rates for the maximum likelihood estimate (MLE) for
the underlying parameters.  In contrast, relatively less attention has
been paid to the computational implications of this mismatch.  In
particular, the algorithm of choice for fitting finite mixture models
is the Expectation-Maximization (EM) algorithm, a general framework
that encompasses various types of divide-and-conquer computational
strategies. The goal of this paper is to gain a fundamental
understanding of the behavior of EM when used to fit over-specified mixture
models.


\paragraph{Statistical issues with over-specification} 
\label{par:statistical_issues_with_over_specification_}

While density estimation in finite mixture models is relatively well
understood~\cite{Vandegeer-2000, Ghosal-2001}, characterizing the
behavior of maximum likelhood for parameter estimation has remained
challenging. The main difficulty for analyzing the MLE in such
settings arises from label switching between the
mixtures~\cite{Green_JRSSB-97, Stephens-2002}, and lack of strong
concavity in the likelihood.  Such issues do not interfere with
density estimation, since the standard divergence measures like the
Kullback-Leibler and Hellinger distances remain invariant under
permutations of labels, and strong concavity is not required. An
important contribution to the understanding of parameter estimation in
finite mixture models was made by Chen~\cite{Chen1992}. He considered
a class of over-specified finite mixture models; here the term
``over-specified'' means that the model to be fit has more mixture
components than the distribution generating the data.  In an
interesting contrast to the usual $\obs^{-\frac{1}{2}}$ convergence
rate for the MLE based on $\obs$ samples, Chen showed that for
estimating scalar location parameters in a certain class of
over-specified finite mixture models, the corresponding rate slows
down to $\obs^{-\frac{1}{4}}$.  This theoretical result has practical
significance, since methods that over-specify the number of mixtures
are often more feasible than methods that first attempt to estimate
the number of components, and then estimate the parameters using the
estimated number of components~\cite{Rousseau-2011}. In subsequent
work, Nguyen~\cite{Nguyen-13} and Heinrich et al.~\cite{Jonas-2016}
have characterized the (minimax) convergence rates of parameter
estimation rates for mixture models in both exactly-fitted or
over-specified settings in terms of the Wasserstein distance.


\paragraph{Computational concerns with mixture models} 
\label{par:computational_concerns_with_mixture_models_}

While the papers discussed above address the statistical behavior of a
global maximum of the log-likelihood, they do not consider the
associated computational issues of obtaining such a maximum. In
general settings, non-convexity of the log-likelihood makes it
impossible to guarantee that the iterative algorithms used in practice
converge to the global optimum, or equivalently the MLE.  Perhaps the
most widely used algorithm for computing the MLE is the
expectation-maximization (EM) algorithm~\cite{Rubin-1977}. Early work
on the EM algorithm~\cite{Jeff_Wu-1983} showed that its iterates
converge asymptotically to a local maximum of the log-likelihood
function for a broad class of incomplete data models; this general
class includes the fitting of mixture models as a special case.  The
EM algorithm has also been studied in the specific setting of Gaussian
mixture models; here we find results both for the population EM
algorithm, which is the idealized version of EM based on an infinite
sample size, as well as the usual sample-based EM algorithm that is
used in practice.  For Gaussian mixture models, the population EM
algorithm is known to exhibit various convergence rates, ranging from
linear to super-linear (quasi-Newton like) convergence if the overlap
between the mixture components tends to zero~\cite{Jordan-1996,
  Jordan-2000}.  It has also been noted in several
papers~\cite{redner1984mixture,Jordan-2000} that the convergence of EM
can be prohibitively slow when the mixtures are not well separated.

\paragraph{Prior work on EM} 
\label{par:prior_work_on_em_}

Balakrishnan et al.~\cite{Siva_2017} laid out a general theoretical
framework for analysis of the EM algorithm, and in particular how to
prove non-asymptotic bounds on the Euclidean distance between
sample-based EM iterates and the true parameter.  When applied to the
special case of two-component Gaussian location mixtures, assumed to
be well-specified and suitably separated, their theory guarantees that
(1) population EM updates enjoy a geometric rate of convergence to the
true parameter when initialized in a sufficiently small neighborhood
around the truth, and (2) sample-based EM updates converge to an
estimate at Euclidean distance of order $(\dims/\obs)^{\frac{1}{2}}$,
based on $\obs$ i.i.d.\ draws from a finite mixture model in
$\realdim$.  Further work in this vein has characterized the behavior
of EM in a variety of settings for two Gaussian mixtures, including
convergence analysis with additional sparsity
constraints~\cite{HanLiu_nips2015,Caramanis-nips2015,Cheng_2018},
global convergence of population EM~\cite{Hsu-nips2016}, guarantees of
geometric convergence under less restrictive conditions on the two
mixture components~\cite{klusowski_2017, Daskalakis_colt2017},
analysis of EM with unknown mixture weights, means and covariances for
two mixtures~\cite{Cai_2018}, and the analysis of EM to more than two
Gaussian components~\cite{Sarkar_nips2017,Cheng_2018}. Other related
work has provided optimization-theoretic guarantees for EM by viewing
it in a generalized surrogate function
framework~\cite{kumar_nips2017}, and analyzed the statistical
properties of confidence intervals based on an EM
estimator~\cite{ychen_2018}.

An assumption common to all of this previous work is that there is no
misspecification in the fitting of the Gaussian mixtures; in
particular, it is assumed that the data is generated from a mixture
model with the same number of components as the fitted model.  A
portion of our recent work~\cite{Raaz-misspecified} has shown that EM
retains its fast convergence behavior---albeit to a biased
estimate---in \emph{under-specified} settings where the number of
components in the fitted model are less than that in the true model.
However, as noted above, in practice, it is most common to use
over-specified mixture models.  For these reasons, it is desirable to
understand how the EM algorithm behaves in the over-specified
settings.

\paragraph{Our contributions} 
\label{par:our_contributions}

The goal of this paper is to shed some light on the non-asymptotic
performance of the EM algorithm for over-specified mixtures.  We
provide a comprehensive study of over-specified mixture models when
fit to a particularly simple (non-mixture) data-generating mechanism;
a multivariate normal distribution $\Ncal(0, \sd^2I_\dims)$ in $\dims$
dimensions with known scale parameter $\sd > 0$.  This setting,
despite its simplicity, suffices to reveal some rather interesting
properties of EM in the over-specified context.  In particular, we
obtain the following results.
\begin{itemize}
  \item {\bf Two-mixture unbalanced fit:} For our first model class,
    we study a mixture of two location-Gaussian distributions with
    unknown location, known variance and known unequal weights for the two
    components. For this case, we establish that the population EM
    updates converge at a geometric rate to the true parameter; as an
    immediate consequence, the sample-based EM algorithm converges in
    $\order{\log (n/d)}$ steps to a ball of radius $(d/n)^{\frac{1}{2}}$. The
    fast convergence rate of EM under the unbalanced setting provides
    an antidote to the pessimistic belief that statistical estimators
    generically exhibit slow convergence for over-specified mixtures.
  \item {\bf Two-mixture balanced fit:} In the balanced version of the
    problem in which the mixture weights are equal to $\frac{1}{2}$
    for both components, we find that the EM algorithm behaves very
    differently.  Beginning with the population version of the EM
    algorithm, we show that it converges to the true parameter from an
    arbitrary initialization.  However, the rate of convergence varies
    as a function of the distance of the current iterate from the true
    parameter value, becoming exponentially slower as the iterates
    approach the true parameter.  This behavior is in sharp contrast
    to well-specified
    settings~\cite{Siva_2017,Daskalakis_colt2017,Sarkar_nips2017},
    where the population updates converge at a geometric rate.  We
    also show that our rates for population EM are tight.  By
    combining the slow convergence of population EM with a novel
    localization argument, one involving the empirical process
    restricted to an annulus, we show that the sample-based EM
    iterates converge to a ball of radius $(d/n)^{\frac{1}{4}}$ around
    the true parameter after $\mathcal{O}((n/d)^{\frac{1}{2}})$ steps.
    The $n^{-\frac{1}{4}}$ component of the Euclidean error matches
    known guarantees for the global maximum of the
    MLE~\cite{Chen1992}.  The localization argument in our analysis is
    of independent interest, because such techniques are not required
    in analyzing the EM algorithm in well-specified settings when the
    population updates are globally contractive.  We note that
    ball-based localization methods are known to be essential in deriving sharp
    statistical rates for M-estimators
    (e.g.,~\cite{Vandegeer-2000,Bar05,Kolt06}); to the best of our
    knowledge, the use of an annulus-based localization argument in analyzing
    an algorithm is novel.
\end{itemize}
Moreover, we show via extensive numerical experiments that the fast convergence
of EM for the unbalanced fit is a special case; and that the slow
behavior of EM proven for the balanced fit (in particular the rate of order
$\obs^{-\frac{1}{4}}$) arises in several general (including more than two components) over-specified Gaussian
mixtures with known variance, known or unknown weights, and unknown location parameters.

\paragraph{Organization} 
\label{par:organization}
The remainder of the paper is organized as follows. In
Section~\ref{sec:set_up} we provide illustrative
simulations of EM in different settings in order to motivate the settings analyzed later in the
paper. We then provide a thorough analysis of the convergence rates of EM
when over-fitting Gaussian data with two components in Section~\ref{sec:main_results}
and the key ideas of the novel proof techniques in Section~\ref{sec:new_techniques_for_analyzing_sample_em}.
We provide a thorough discussion of our results 
in Section~\ref{sec:discussion}, exploring their general applicability, and presenting further simulations that substantiate the value of our theoretical framework. Detailed proofs of our results and discussion
of certain additional technical aspects of our results are provided in the
appendix.
\paragraph{Notation} For any two sequences $a_{n}$ and $b_{n}$, the notation
\mbox{$a_{n} \precsim b_{n}$} or $a_n = \order{b_n}$ means that $a_{n} \leq
C b_{n}$ for some universal constant $C$. Similarly, the notation $a_{n} \asymp
b_{n}$ or $a_n= \Theta(b_n)$ denotes that both the conditions, $a_{n} \precsim
b_{n}$ and $b_{n} \precsim a_{n}$, hold.
Throughout this paper, $\weight$ denotes a variable and $\pi$ denotes the
mathematical constant ``pi''.

\paragraph{Experimental settings} 
\label{par:experimental_settings_}

We summarize a few common aspects of the numerical experiments
presented in the paper. Population-level computations were done using
numerical integration on a sufficiently fine grid.  With finite
samples, the stopping criteria for the convergence of EM were: (1) the
change in the iterates was small enough, or (2) the number of
iterations was too large (greater than $100,000$).  Experiments were
averaged over several repetitions (ranging from $25$ to $400$).  In
majority of the runs, for each case, criteria~(1) led to convergence.
In our plots for sample EM, we report $\widehat m_e + 2\widehat s_e$
on the y-axis, where $\widehat m_e, \widehat s_e$ respectively denote
the mean and standard deviation across the experiments for the metric
under consideration, e.g., the parameter estimation
error. Furthermore, whenever a slope is provided, it is the slope for
the least-squares fit on the log-log scale for the quantity on
$y$-axis when fitted with the quantity reported on the $x$-axis.  For
instance, in Figure~\ref{fig:snr_effect}(b), we plot
$\vert\widehat\theta_n-\thetastar \vert$ on the $y$-axis value versus
the sample size $n$ on the $x$-axis, averaged over $400$ experiments,
accounting for the deviation across these experiments.  Furthermore,
the green dotted line with legend $\weight=0.3$ and the corresponding
slope $-0.48$ denote the least-squares fit and the respective slope
for $\log \vert\widehat\theta_n -\thetastar \vert$ (green solid dots)
with $\log n$ for the experiments corresponding to the setting
$\weight = 0.3$.




\section{Motivating simulations and problem set-up}
\label{sec:set_up}

In this section, we explore a wide range of behavior demonstrated by
EM for certain settings of over-specified location Gaussian
mixtures.
We begin with several simulations that
illustrate fast and slow convergence of EM for various settings, and
serve as a motivation for the theoretical results derived later in the
paper. We provide basic background on EM in Section~\ref{sub:setup},
and describe the problems to be tackled.


\subsection{Problem set-up} 
\label{sub:problem_set_up}

Let $\normDensity(\cdot \:; \mu, \Sigma)$ denote the density of a
Gaussian random vector with mean $\mu$ and covariance $\Sigma$.
Consider the two component Gaussian mixture model with density
\begin{align}
  \label{EqnSpecial}
  f(x; \thetastar, \sd, \weight)
  & \defn \weight \normDensity(x;
\thetastar,\sd^2I_d) + (1-\weight) \normDensity(x ; - \thetastar,
\sd^2 I_d).
\end{align}
Given $n$ samples from the distribution~\eqref{EqnSpecial}, suppose
that we use the EM algorithm to fit a two-component location Gaussian mixture
with fixed weights and variance\footnote{Refer to
  Section~\ref{sec:discussion} for a discussion for the case of
  unknown weights and variances.}  and special structure on the
location parameters---more precisely, we fit the model with density
\begin{align}
  \label{EqModelFit}
  f(x; \theta, \sd, \weight) & \defn \weight \normDensity(x ; \theta,
  \sigma^2I_\dims) + (1 - \weight) \normDensity(x; -\theta,
  \sigma^2I_\dims)
\end{align}
using the EM algorithm, and take the solution\footnote{Strictly speaking,
 different initialization of EM may converge to different estimates. 
 For the settings analyzed theoretically in this work, the EM always 
 converges towards the same estimate in the limit of infinite steps, 
 and we use a stopping criterion to determine the final estimate. 
 See the discussion on experimental settings in 
 Section~\ref{sec:introduction} for more details.} as an estimate of
$\thetastar$. An important aspect of the problem at hand is the
signal strength, which is measured as the separation between the means
of mixture components relative to the spread in the components. For
the model~\eqref{EqnSpecial}, the signal strength is given by the
ratio $\enorm{\thetastar}/\sd$.  When this ratio is large, we refer to
it as the \emph{strong signal} case; otherwise, it corresponds to the \emph{weak
signal} case.  Of particular interest to us is the behavior of EM in
the limit of weak signal when there is no separation; i.e.,
$\enorm{\thetastar} = 0$.  For such cases, we call the
fit~\eqref{EqModelFit} an \emph{unbalanced} fit when $\weight \neq
\frac{1} {2}$ and a \emph{balanced} fit when $\weight=\frac{1}{2}$.
Note that the setting of $\thetastar=0$ corresponds to the simplest case of over-specified
fit, since the true model has just one component (standard normal
distribution irrespective of the parameter $\weight$) but the fitted
model has two (one extra) component (unless the EM estimate is also $0$). We now present the empirical
behavior of EM for these models and defer the derivation of EM updates
to Section~\ref{sub:setup}.


\subsection{Numerical Experiments: Fast to slow convergence of EM}
\label{sub:motivation_fast_to_slow_convergence_of_em}

We begin with a numerical study of the effect of separation among the
mixtures on the statistical behavior of the estimates returned by EM.
Our main observation is that weak or no separation leads to relatively
low accuracy estimates.  Additional simulations for more general
mixtures, including more than two components, are provided in
Section~\ref{sub:slow_rates_for_general_mixtures}.  Next, via
numerical integration on a grid with sufficiently small discretization
width, we simulate the behavior of the population EM algorithm
width---an idealized version of EM in the limit of infinite
samples---in order to understand the effect of signal strength on EM's
algorithmic rate of convergence, i.e., the number of steps needed for
population EM to converge to a desired accuracy. We observe a slow
down of EM on the algorithmic front when the signal strength
approaches zero.


\subsubsection{Effect of signal strength on sample EM} 
\label{ssub:effect_of_signal_strength_on_em}

In Figure~\ref{fig:snr_effect}, we show simulation results for data generated
from the model~\eqref{EqnSpecial} in dimension $d = 1$ and noise
variance $\sigma^2 = 1$, and for three different values of the weight
\mbox{$\weight \in \braces{0.1, 0.3, 0.5}$.}  In all cases, we fit a
two-location Gaussian mixture with fixed weights and variance as
specified by equation~\eqref{EqModelFit}.  The two panels show the
estimation error of the EM solution as a function of $n$ for two
distinct cases of the data-generating mechanism: (a) in the strong
signal case, we set $\thetastar=5$ so that the data has two
well-separated mixture components, and (b) to obtain the limiting case
of no signal, we set $\thetastar = 0$, so that the two mixture
components in the data-generating distribution collapse to one, and we
are simply fitting the data from a standard normal distribution.

In the strong signal case, it is well
known~\cite{Siva_2017,Daskalakis_colt2017} that EM solutions have
an estimation error (measured by the Euclidean distance between the EM estimate and the true parameter $\thetastar$)  that achieves the classical (parametric) rate
$n^{-\frac{1}{2}}$; the empirical results in Figure~\ref{fig:snr_effect}(a)
are simply a confirmation of these theoretical predictions.  More
interesting is the case of no signal (which is the limiting case with weak
signal), where the simulation results shown in panel (b) 
of Figure~\ref{fig:snr_effect} reveal a different
story. In this case, whereas the EM solution (with random standard normal
initialization)
has an error that decays as
$n^{-\frac{1}{2}}$ when $\weight \neq 1/2$, its error decays at the
considerably slower rate $n^{-\frac{1}{4}}$ when $\weight = 1/2$.  We
return to these cases in further detail in
Section~\ref{sec:main_results}.
\begin{figure}[h]
  \begin{center}
    \begin{tabular}{cc}
    \widgraph{0.45\textwidth}{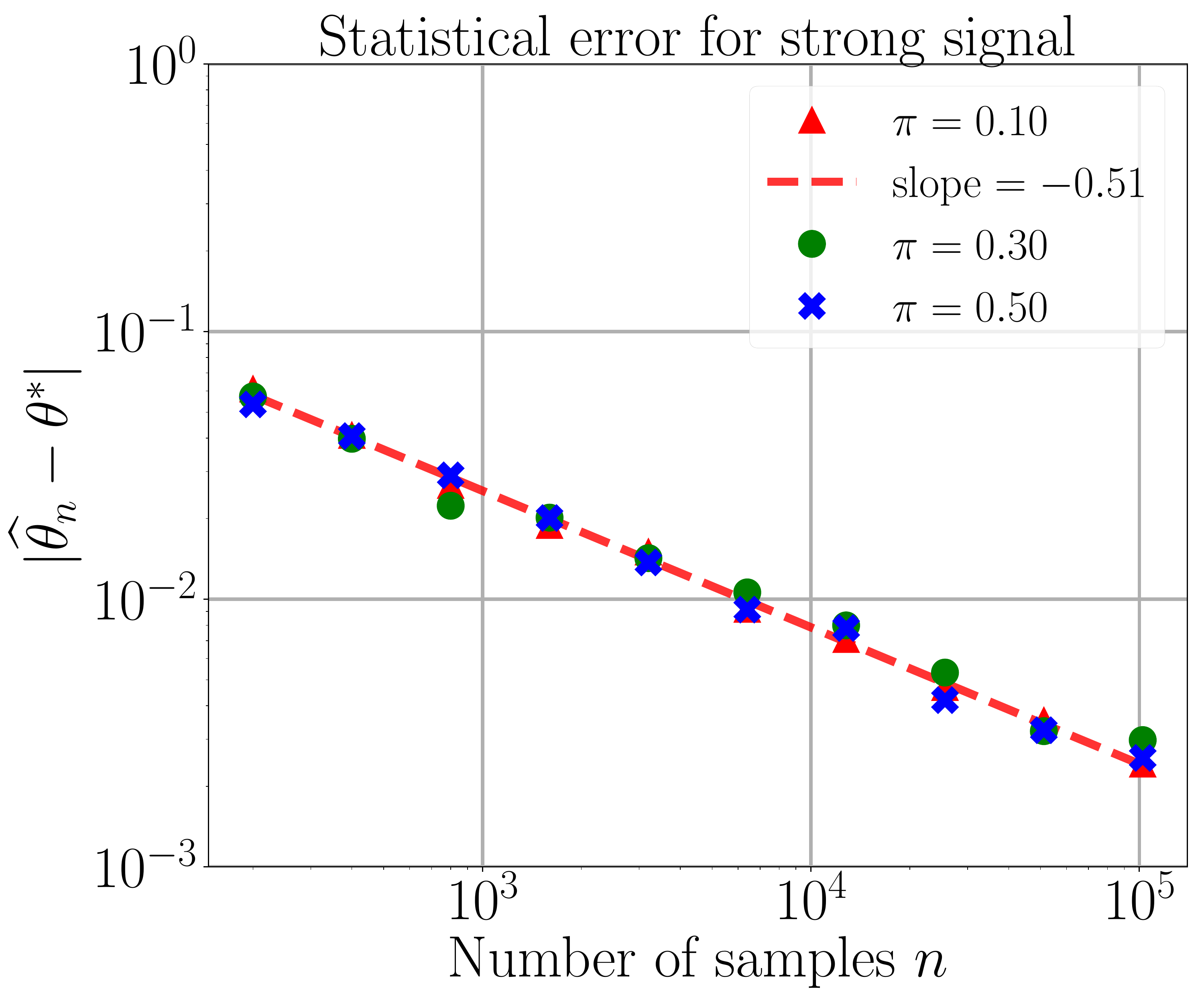} &
    \widgraph{0.45\textwidth}{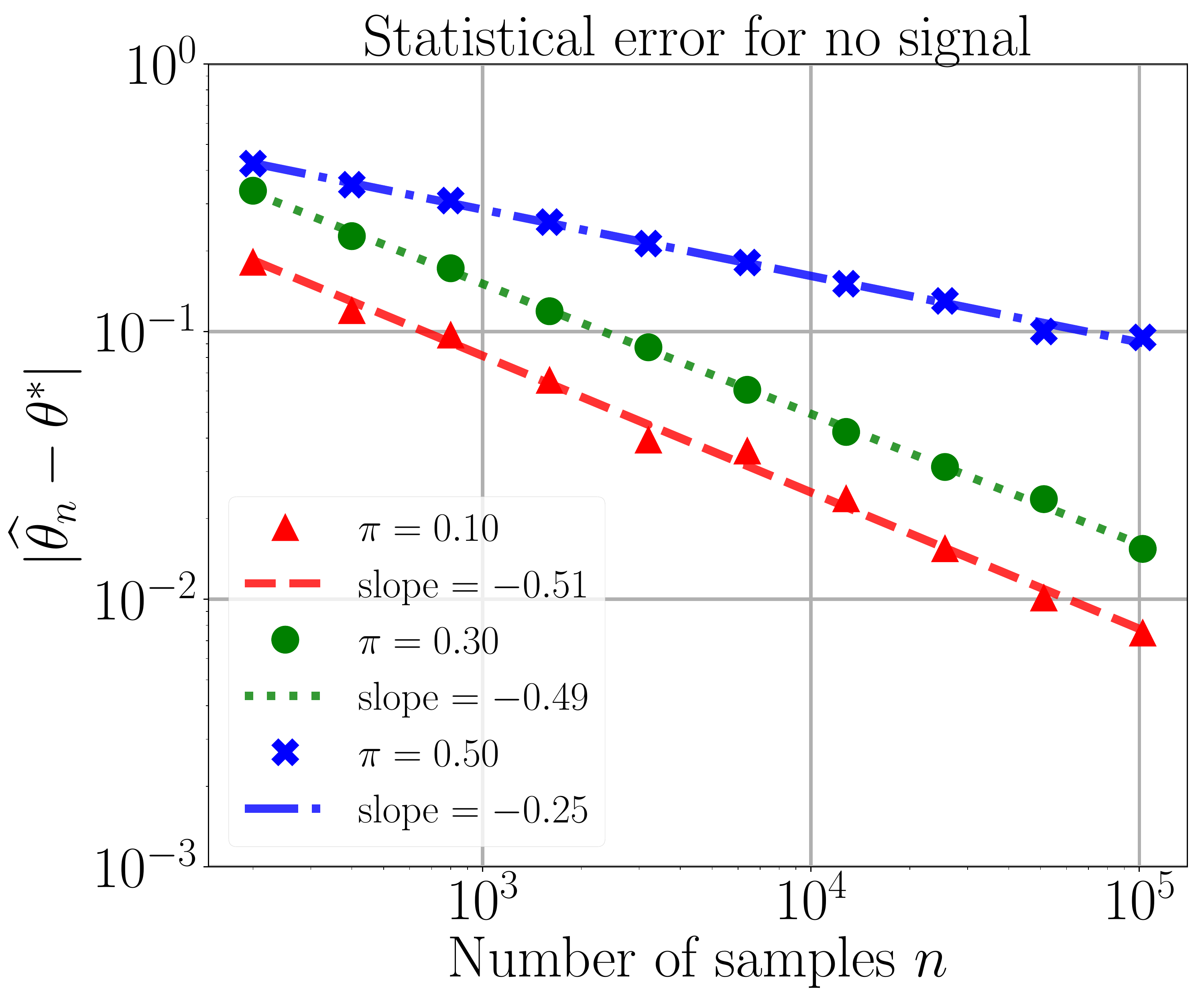} \\
    (a) $\thetastar = 5$ & (b) $\thetastar = 0$
    \end{tabular}
    \caption{Plots of the error $| \widehat \theta_n - \thetastar|$ in
      the EM solution versus the sample size $n$, focusing on the
      effect of signal strength on EM solution accuracy.  The true
      data distribution is given by $\weight\NORMAL(\thetastar, 1)+
      (1-\weight)\NORMAL(-\thetastar, 1)$ and we use EM to fit the
      model $\weight \NORMAL(\theta, 1)+
      (1-\weight)\NORMAL(-\theta,1)$, generating the EM estimate
      $\widehat\theta_n$ based on $n$ samples.  (a) When the signal is
      strong, the estimation rate decays at the parametric rate
      $n^{-\frac{1}{2}}$, as revealed by the $-1/2$ slope in a least-square
      fit of the log error based on the log sample size $\log n$. (b)
      When there is no signal ($\thetastar = 0$), then depending on
      the choice of weight $\weight$ in the fitted model, we observe
      two distinct scalings for the error: $n^{-\frac{1}{2}}$ when $\weight
      \neq 0.5$, and, $n^{-\frac{1}{4}}$ when $\weight=0.5$, again as
      revealed by least-squares fits of the log error using the log
      sample size $\log n$.  }
\label{fig:snr_effect}
  \end{center}
\end{figure}

\subsubsection{Interesting behavior of population EM}

The intriguing behavior of the sample EM algorithm in the ``no
signal'' case motivated us to examine the behavior of population EM
for this case.  To be clear, while sample EM is the practical
algorithm that can actually be applied, it can be insightful for
theoretical purposes to first analyze the convergence of the
population EM updates, and then leverage these findings to understand
the behavior of sample EM~\cite{Siva_2017}. Our analysis follows a
similar road-map.  Interestingly, for the case with $\thetastar=0$,
the population EM algorithm behaves significantly differently for the
unbalanced fit ($\weight \neq \frac{1}{2}$) as compared to the
balanced fit ($\weight= \frac{1}{2}$) (equation~\eqref{EqModelFit}).
In Figure~\ref{fig:population_em}, we plot the distance of the
population EM iterate $\theta^t$ to the true parameter value,
$\theta^* = 0$, on the vertical axis, versus the iteration number $t$
on the horizontal axis.  With the vertical axis on a log scale, a
geometric convergence rate of the algorithm shows up as a negatively
sloped line (disregarding transient effects in the first few
iterations).

\begin{figure}[h]
  \begin{center}
    \begin{tabular}{cc}
    \widgraph{0.45\textwidth}{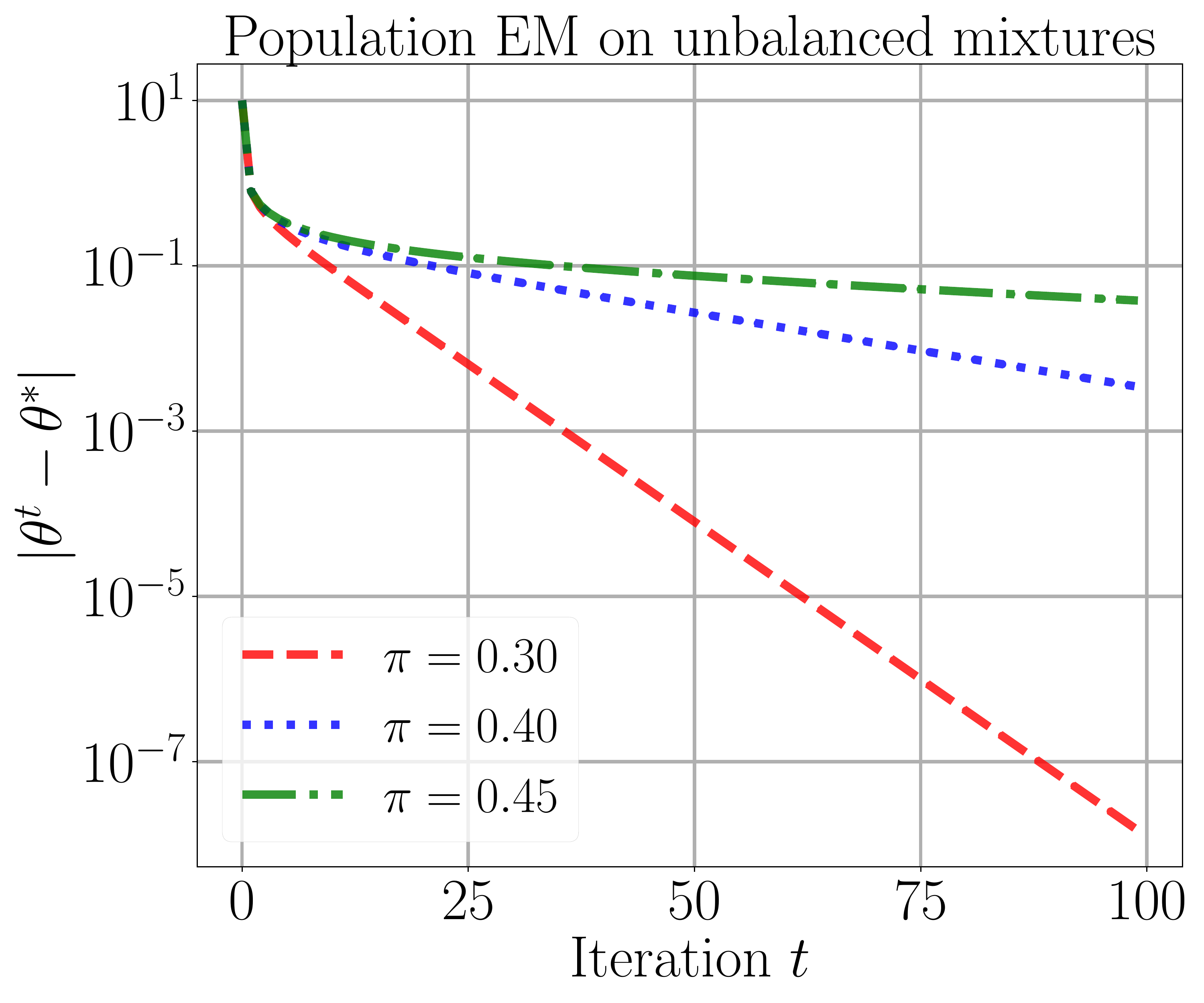} 
      & \widgraph{0.45\textwidth}{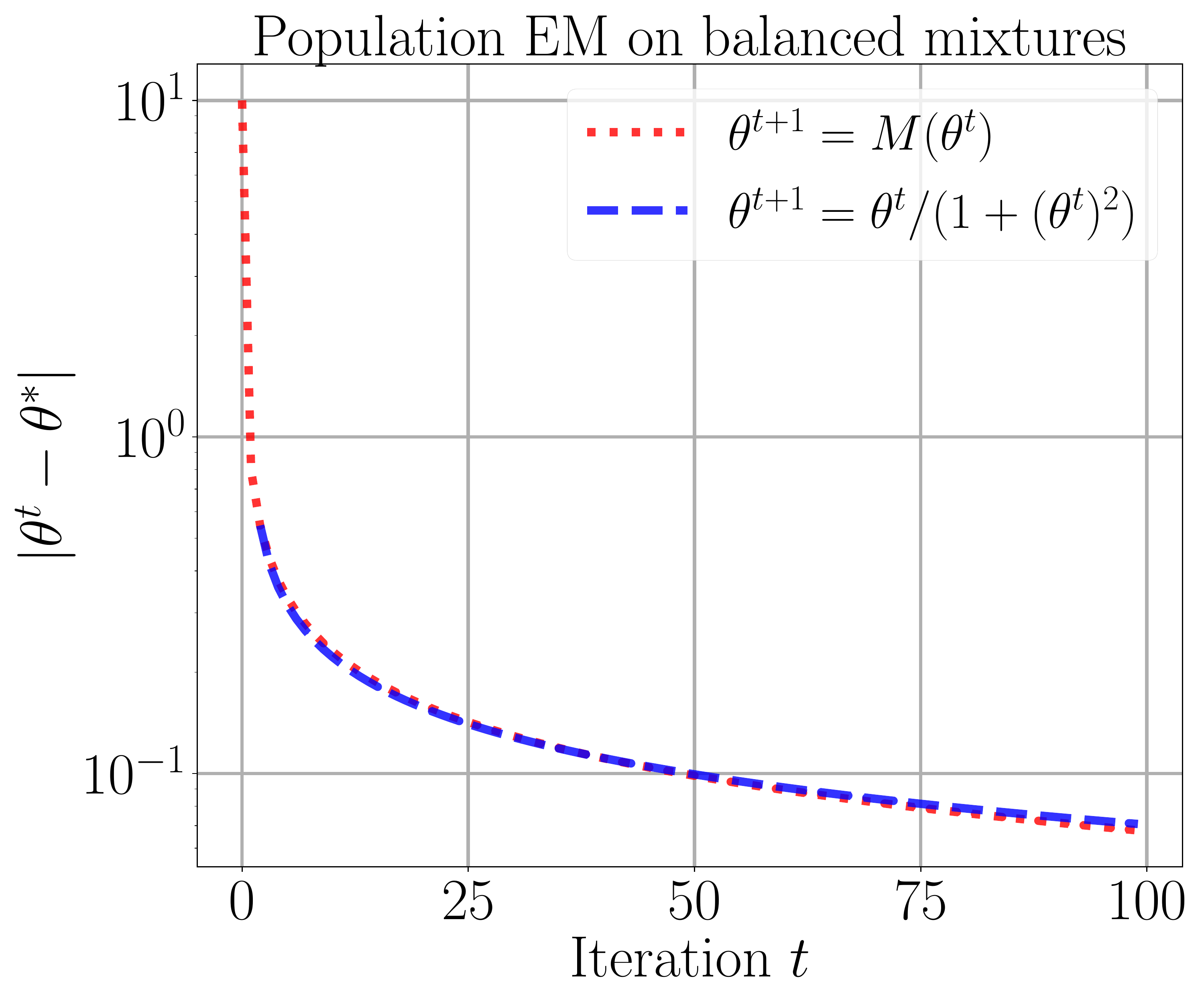} \\
    (a) & (b)
    \end{tabular}
\caption{ Behavior of the (numerically computed) population EM
  updates~\eqref{EqnPopMupdate} when the underlying data distribution
  is $\NORMAL(0, 1)$.  (a) Unbalanced mixture fits~\eqref{EqModelFit}
  with weights~$(\weight, 1-\weight)$: We observe geometric
  convergence towards $\thetastar=0$ for all $\weight \neq 0.5$
  although the rate of convergence gets slower as $\weight\to 0.5$.
  (b) Balanced mixture fits~\eqref{EqModelFit} with weights~$(0.5,
  0.5)$: We observe two phases of convergence. First, EM quickly
  converges to ball of constant radius and then it exhibits slow
  convergence towards $\thetastar=0$.  Indeed, we see that during the
  slow convergence, the population EM updates track the curve given by
  $\theta^{t+1} = \theta^t/(1+(\theta^t)^2)$ very closely, as
  predicted by our theory.}
\label{fig:population_em}
  \end{center}
\end{figure}

For the unbalanced mixtures in panel (a), we see that EM converges
geometrically quickly, although the rate of convergence (corresponding
to the slope of the line) tends towards zero as the mixture weight
$\weight$ tends towards $1/2$ from below.  For $\weight = 1/2$, we obtain a
balanced mixture, and, as shown in the plot in panel (b), the
convergence rate is now sub-geometric.  In fact, the behavior of the
iterates is extremely well characterized by the recursion $\theta
\mapsto \frac{\theta}{1 + \theta^2}$.

The theory to follow provides a precise characterization of the
behavior seen in Figures~\ref{fig:snr_effect}(b)
and~\ref{fig:population_em}. Furthermore, in
Section~\ref{sec:discussion}, we provide further support for
relevannce of our theoretical results in explaining the behavior of EM
for other classes of over-specified models, including Gaussian mixture
models with unknown weights as well as mixtures of linear regressions.

\subsection{EM updates for the model fit~\eqref{EqModelFit}}
\label{sub:setup}

In this section, we provide a quick introduction to the EM
updates. Readers familiar with the literature can skip directly to the
main results in Section~\ref{sec:main_results}.  Recall that the two-component model fit is based on the density
\begin{align}
\label{EqModelFitNew}
    \weight \normDensity(x  ;\theta, \sd^2 I_\dims) + (1-\weight)
    \normDensity(x ; -\theta, \sd^2I_\dims).
\end{align}
From now on we assume that the data is drawn from the zero-mean
Gaussian distribution $\NORMAL(0, \sd^2I_d)$.  Note that the model fit
described above contains the true model with $\thetastar=0$ and it is
referred to as an over-specified fit since for any non-zero $\theta$,
the fitted model has two components.

The maximum likelihood estimate is obtained by solving the following
optimization problem
\begin{align}
    \mle \in \arg \max_{\theta \in \Theta} \frac{1}{n} \sum_{i=1}^n
    \left \{ \log(\weight \normDensity(x_i;\theta,\sd^2I_{\dims}) +
    (1-\weight)\normDensity(x_i;-\theta,\sd^2I_{\dims}) ) \right \}.
\label{eq:sample_likelihood}
\end{align}
In general, there is no closed-form expression for $\mle$.
The EM algorithm circumvents this problem via a
minorization-maximization scheme. Indeed, population EM is a surrogate
method to compute the maximizer of the population log-likelihood
\begin{align}
\label{eq:pop_likelihood}
    \mathcal{L}(\theta) := \Exs_{X} \brackets{ \log (\weight
      \normDensity(X;\theta,\sd^2I_{\dims}) +
      (1-\weight)\normDensity(X;-\theta,\sd^2I_{\dims})},
\end{align}
where the expectation is taken over the true distribution.
On the other hand, sample EM attempts to estimate $\mle$.
We now describe the expressions for both the sample and population EM updates
for the model-fit~\eqref{EqModelFitNew}.

Given any point $\theta$, the EM algorithm proceeds in two steps: (1)
compute a surrogate function $Q(\cdot; \theta)$ such that $Q(\theta';
\theta) \leq \mathcal{L}(\theta')$ and $Q(\theta; \theta) =
\mathcal{L}(\theta)$; and (2) compute the maximizer of $Q(\theta';
\theta)$ with respect to $\theta'$.  These steps are referred to as
the E-step and the M-step, respectively.  In the case of two-component
location Gaussian mixtures, it is useful to describe a hidden variable
representation of the mixture model.  Consider a binary indicator
variable $\hidden \in \left\{0,1\right\}$ with the marginal
distribution \mbox{$\Prob(Z = 1) = \weight$} and \mbox{$\Prob(Z = 0) =
  1 - \weight$,} and define the conditional distributions
\begin{align*}
  \parenth{X \mid Z= 0} \sim \Ncal(-\theta, \sd^{2}I_{d}), 
  \quad\text{and} \quad 
  \parenth{X \mid Z=1} \sim \Ncal(\theta,\sd^{2}I_{d}).
\end{align*}
These marginal and conditional distributions define a joint distribution
over the pair $(X, Z)$, and by construction, the induced marginal distribution
over $X$ is a Gaussian mixture of the form~\eqref{EqModelFitNew}.
For EM, we first compute the conditional probability of $Z = 1$ given $X
= x$: \begin{align}
  \label{EqnWeightFun}
  \weightFun_{\theta}(x) = 
  \weightFun^\weight_{\theta}(x) \mydefn
  \frac{\weight \exp
    \parenth{-\frac{\enorm{\theta - x}^2}{2\sd^2}}}{\weight \exp
    \parenth{-\frac{\enorm{\theta - x}^2}{2\sd^2}} + (1-\weight) \exp
    \parenth{-\frac{\enorm{\theta + x}^2}{2\sd^2}}}.
\end{align}
Then, given a vector $\theta$, the E-step in the population EM
algorithm involves computing the minorization function $\theta'
\mapsto Q(\theta', \theta)$. Doing so is equivalent to computing
  the expectation
\begin{align}
\label{EqnEmMinor}
\funQ(\theta'; \theta) & = -\frac{1}{2} \Exs \brackets
     {\weightFun_{\theta}(X)\enorm{X - \theta'}^2 + (1 -
       \weightFun_{\theta}(X))\enorm{X + \theta'}^2 },
\end{align}
where the expectation is taken over the true distribution (here
$\NORMAL (0, \sd^2I_\dims)$.  In the M-step, we maximize the function
\mbox{$\theta' \mapsto \funQ(\theta'; \theta)$.}  Doing so defines a
mapping $\updateM: \Rspace^{d} \to \Rspace^{d}$, known as the
\emph{population EM operator}, given by
\begin{align}
\label{EqnPopMupdate}
\updateM(\theta) &= \arg \max_{\theta' \in \Rspace^\usedim}
\funQ(\theta', \theta) = \Exs \Big[ (2\weightFun_{\theta}(X)-1) X
  \Big].
\end{align}
In this definition, the second equality follows by computing the
gradient $\nabla_{\theta'} \funQ$, and setting it to zero.  In
summary, for the two-component location mixtures considered in this
paper, the population EM algorithm is defined by the sequence
\mbox{$\theta^{t+1} = M(\theta^t)$}, where the operator $M$ is defined
in equation~\eqref{EqnPopMupdate}.

We obtain the \emph{sample EM update} by simply replacing
the expectation $\Exs$ in equations~\eqref{EqnEmMinor}
and~\eqref{EqnPopMupdate} by the empirical average based on an
observed set of samples.  In particular, given a set of i.i.d.\ samples
$\braces{X_i}_{i=1}^\obs$, the sample EM operator $\updateM_\obs
:\realdim \mapsto \realdim$ takes the form
\begin{align}
\label{eq:sample_em_operator}
\updateM_\obs(\theta) & \defn \frac{1}{\obs} \sum_{i=1}^n (2
\weightFun_\theta(X_i) - 1) X_i.
\end{align}
Overall, the sample EM algorithm generates the sequence of iterates
given by $\theta^{t+1} = M_n(\theta^t)$.

In the sequel, we study the convergence of EM both for the population EM
algorithm in which the updates are given by 
\mbox{$\theta^{t+1} = M(\theta^t)$}, and the sample-based EM sequence given
by \mbox{$\theta^{t+1} = M_n(\theta^t)$}. 
With this notation in place, we now turn to the main results of this paper.


\section{Main results}
\label{sec:main_results}

In this section, we state our main results for the convergence rates
of the EM updates under the unbalanced and balanced mixture fit. We start
with the easier case of unbalanced mixture fit in
Section~\ref{sub:unbalanced_mixtures} followed by the more delicate
(and interesting) case of the balanced fit in
Section~\ref{ssub:population_em_for_balanced_mixtures}.


\subsection{Behavior of EM for unbalanced mixtures}
\label{sub:unbalanced_mixtures}
We begin with a characterization of both the population and
sample-based EM updates in the setting of unbalanced mixtures.
In particular, we assume that the fitted two-components
mixture model~\eqref{EqModelFitNew} has known weights $\weight$ and
$1-\weight$, where $\weight \in (0, 1/2)$.
The following result characterizes the behavior 
of the EM updates for this set-up.
\begin{theorem} 
\label{ThmUnbalanced}
Suppose that we fit an unbalanced instance (i.e., $\weight \neq
\frac{1}{2}$) of the mixture model~\eqref{EqModelFitNew} to
$\NORMAL(0, \sd^2 I_d)$ data.  Then:
\begin{enumerate}
  \item[(a)] The population EM operator~\eqref{EqnPopMupdate} is
    globally strictly contractive, meaning that
\begin{subequations}
\begin{align}
\label{EqnUnbalancedPop}  
  \enorm{\updateM(\theta)} & \leq \parenth{1 - \contracasym^2/2}
  \enorm{\theta} \qquad \mbox{for all $\theta \in \real^\usedim$,}
\end{align}
where $\contracasym \defn |1 - 2 \weight| \in (0, 1)$.
\item[(b)] There are universal constants $c, c'$ such that given any
  $\delta \in (0, 1)$ and a sample size $\obs \geq \unicon
  \frac{\sd^2} {\contracasym^4} \, (\dims + \log(1/\delta)) $, the
  sample EM sequence \mbox{$\theta^{t+1} = \updateM_\obs(\theta^t)$}
  generated by the update~\eqref{eq:sample_em_operator} satisfies the
  upper bound
\begin{align}
  \label{EqnUnbalancedSample}
\enorm{\theta^t} \leq   \enorm{\theta^0}\parenth{ 1 -
    \frac{\contracasym^2}{2}}^t + \frac{
    \uniconnew(\enorm{\theta^0}\sd^2 + \contracasym \sigma) }{\contracasym^2} \sqrt{
    \frac{\dims
      + \log(1/\delta)}{\obs}}\;,
\end{align}
\end{subequations}
with probability at least $1-\delta$.
\end{enumerate}
\end{theorem}
\noindent See Appendix~\ref{sub:proof_of_theorem_thm:ThmUnbalanced} for the proof of this theorem.  \\


\paragraph{Fast convergence of population EM}

The bulk of the effort in proving Theorem~\ref{ThmUnbalanced} lies in
establishing the guarantee~\eqref{EqnUnbalancedPop} for the population
EM iterates. Such a contraction bound immediately yields the exponential
fast convergence of the population EM updates
$\theta^{\iter+1}=\updateM(\theta^{\iter})$ to $\thetastar=0$:
\begin{align}
  \label{eq:unbalanced_pop}
  \enorm{\theta^{T}} \leq \smallradius \quad\text{for }\quad T \geq
  \frac{1}{\log\frac{1}{(1-\contracasym^2/2)}} \cdot {\log\parenth{
      \frac{\enorm{\theta^0}}{\smallradius}}}.
\end{align}
Since the mixture weights $(\weight, 1 - \weight)$ are bounded away
from $1/2$, we have that $\contracasym=|1 - 2 \weight|$ is bounded
away from zero, and thus population EM iterates converge in
$\order{\log(1/\smallradius)}$ steps to an $\smallradius$-ball around
$\thetastar=0$.  This result is equivalent to showing that in the
unbalanced instance $(\weight \neq 1/2)$, the log-likelihood is
strongly concave around the true parameter.


\paragraph{Statistical rate of sample EM}

Once the bounds~\eqref{EqnUnbalancedPop}
and~\eqref{eq:unbalanced_pop}) have been established, the proof of the
statistical rate~\eqref{EqnUnbalancedSample} for sample EM utilizes
the scheme laid out by Balakrishnan et al.~\cite{Siva_2017}. In
particular, we prove a non-asymptotic uniform law of large numbers
(Lemma~\ref{lemma:bound_EM_operators} stated in
Section~\ref{SecSuboptimal}) that allows for the translation from
population to sample EM iterates.  Roughly speaking,
Lemma~\ref{lemma:bound_EM_operators} guarantees that for any radius $r
> 0$, tolerance $\delta \in (0,1)$, and sufficiently large $n$, we
have
\begin{align}
  \label{eq:radem_sneak_peek}  
  \Prob \left [ \sup \limits_{\|\theta\|_2 \leq r}
    \enorm{M_{n}(\theta) - M(\theta)} \leq \unicontwo \sigma( \sigma r
    + \contracasym ) \sqrt{\frac{d+\log(1/\delta)}{n}} \; \right] \geq
  1-\delta.
\end{align}  
This bound, when combined with the contractive
behavior~\eqref{EqnUnbalancedPop} or equivalently the exponentially
fast convergence~\eqref{eq:unbalanced_pop} of the population EM
iterates allows us to establish the stated
bound~\eqref{EqnUnbalancedSample}. (See, e.g., Theorem~2 in the
paper~\cite{Siva_2017}.)

Putting the pieces together, we conclude that the \mbox{sample EM}
updates converge to an estimate of $\thetastar$---that has Euclidean
error of the order $(\dims/\obs)^{\frac{1}{2}}$---after a relatively
small number of steps that are of the order $\log(n/d)$.  Note that
this theoretical prediction is verified by the simulation study in
Figure~\ref{fig:snr_effect}(b) for the univariate setting ($d=1$) of
the unbalanced mixture-fit.  In Figure~\ref{fig:unbalanced_rates}, we
present the scaling of the radius of the final EM
iterate\footnote{Refer to the discussion before
  Section~\ref{sec:set_up} for details on the stopping rule for EM.}
with respect to the sample size $n$ and the dimension $d$, averaged
over $400$ runs of sample EM for various settings of $(n, d)$.  Linear
fits on the log-log scale in these simulations suggest a rate close to
$(d/n)^{\frac{1}{2}}$ as claimed in Theorem~\ref{ThmUnbalanced}.
\begin{figure}[h]
  \begin{center}
    \begin{tabular}{cc}
      \widgraph{0.45\textwidth}
               {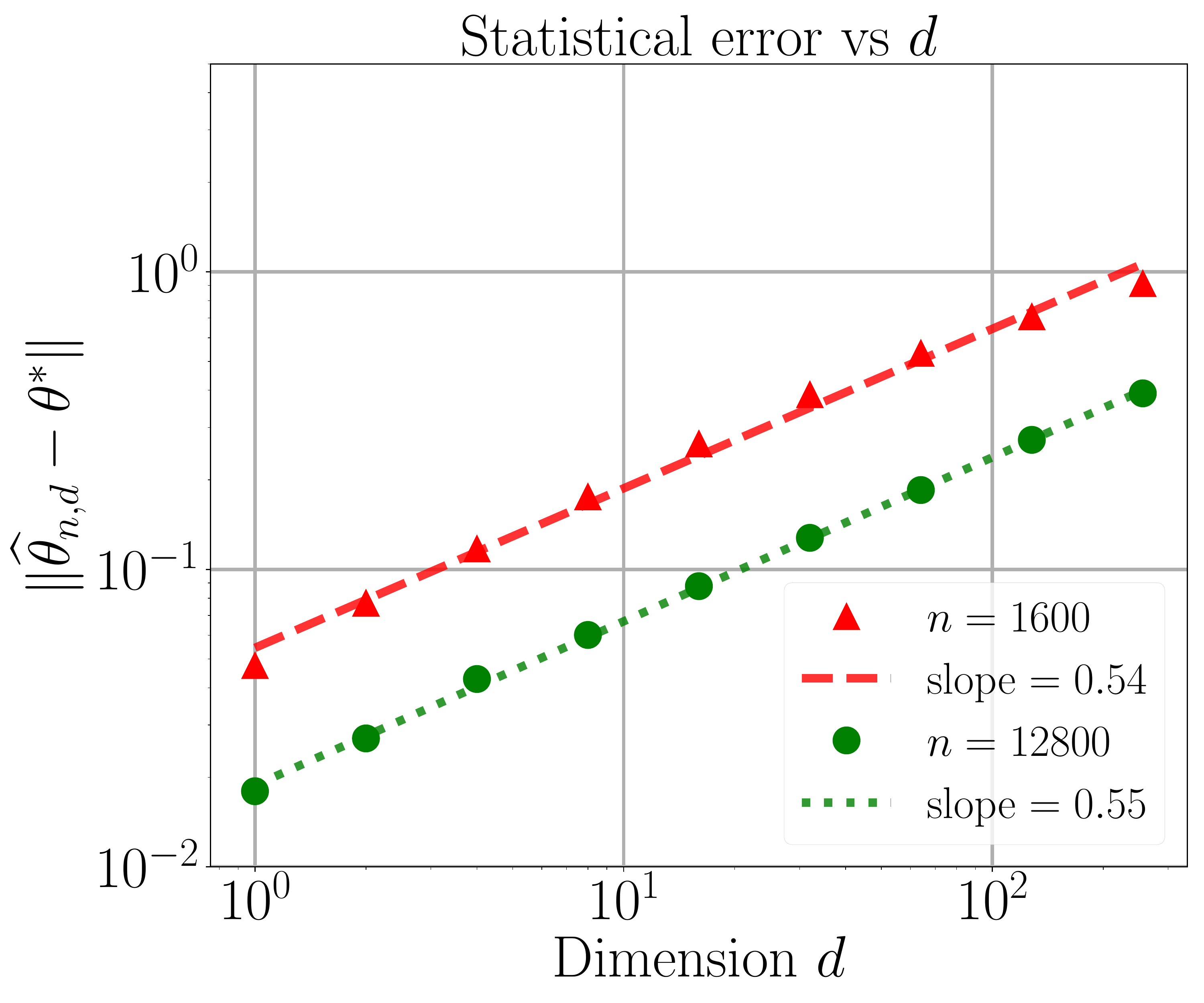}
               & \widgraph{0.45\textwidth}
               {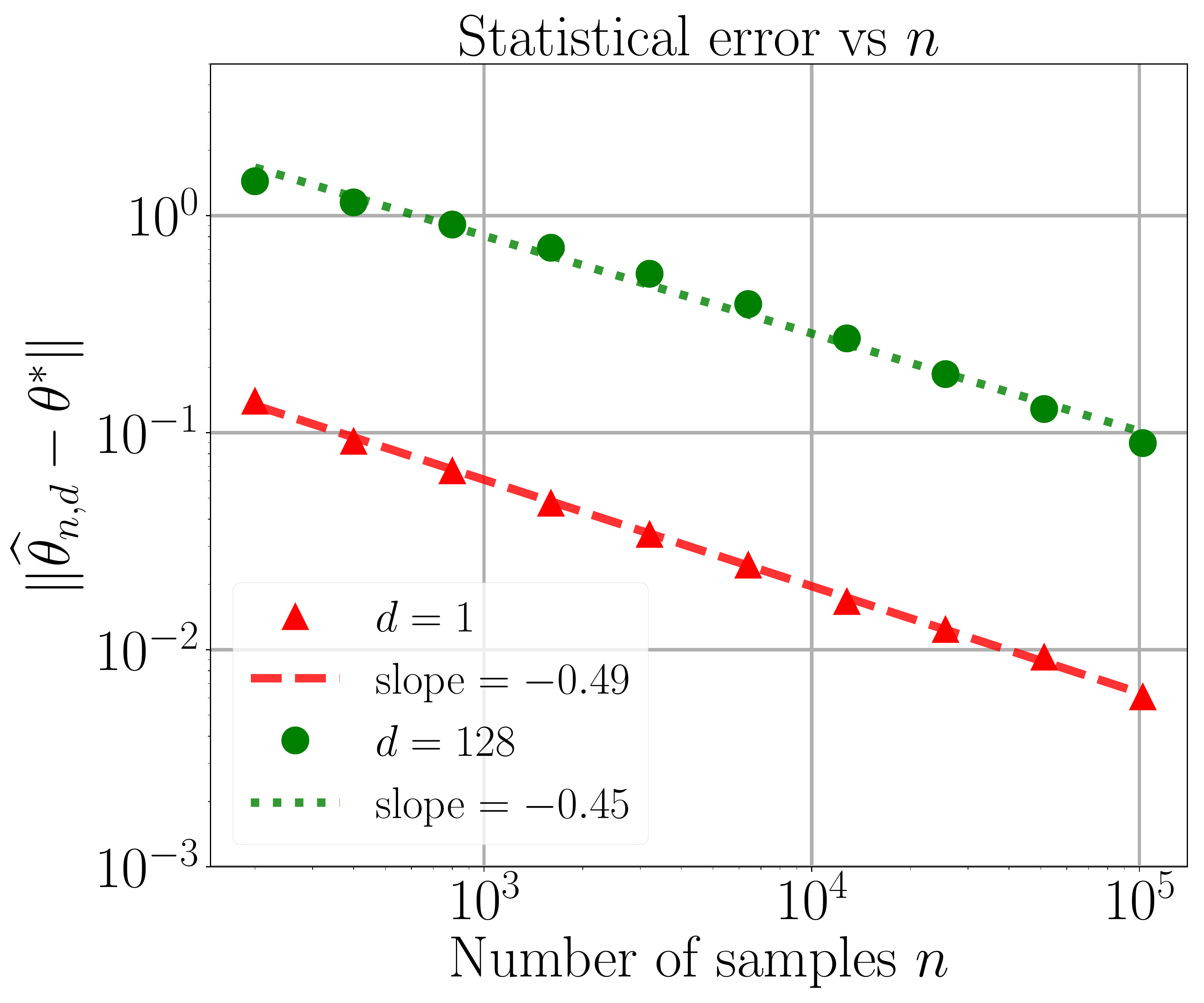}
               \\
(a) & (b)
    \end{tabular}
    \caption{Scaling of the Euclidean error $\|\widehat \theta_{n,d} -
      \thetastar\|_2$ for EM estimates $\widehat \theta_{n,d}$
      computed using the unbalanced ($\weight\neq \frac{1}{2}$) 
      mixture-fit~\eqref{EqModelFitNew}.
      Here the true data distribution is $\NORMAL(0, I_d)$, i.e., 
      $\thetastar=0$, and
      $\widehat\theta_{n, d}$ denotes the EM iterate upon convergence
      when we fit a two-mixture model with mixture weights $(0.3,
      0.7)$ using $n$ samples in $d$ dimensions. 
      (a)~Scaling with respect to $\dims$ for $\obs \in
      \braces{1600, 12800}$.  (b)~Scaling with respect to $\obs$ for
      $\dims \in \braces{1, 128}$. We ran experiments for several
      other pairs of $(n, d)$ and the conclusions were the same.  The
      empirical results here show that that our theoretical upper
      bound of the order $(\dims/\obs)^{\tfrac{1}{2}}$ on the EM solution is
      sharp in terms of \mbox{$n$ and $d$}.}
  \label{fig:unbalanced_rates}
  \end{center}
\end{figure}

\paragraph{Remark} 
We make two comments in passing. First, the value of $\enorm{\theta^0}$
in the convergence rate of sample EM updates in
Theorem~\ref{ThmUnbalanced} can be assumed to be of constant
order; this assumption stems from the fact the population EM operator
maps any $\theta^0$ to a vector with norm smaller than $\sqrt{2/ \pi}$
(cf. Lemma~\ref{lemma:bound_initialization} in
Appendix~\ref{sec:ini_unbalanced}). Second, when the weight parameter~$\weight$
is assumed to be unknown in the model fit~\eqref{EqModelFitNew}, the EM
algorithm exhibits fast convergence when $\weight$ is initialized
sufficiently away from $\frac{1}{2}$; see 
Section~\ref{sub:when_the_weights_are_unknown} for more details.


\paragraph{From unbalanced to balanced fit} 
\label{par:from_unbalanced_to_balanced_fit}
The bound~\eqref{eq:unbalanced_pop} shows that the extent of unbalancedness
in the mixture weights plays a crucial role in the
geometric rate of convergence for the population EM.  
When the mixtures become more balanced, that
is, weight $\weight$ approaches $1/2$ or equivalently $\contracasym$
approaches zero, the number of steps $T$
required to achieve $\smallradius$-accuracy scales as
$\order{{\log(\enorm{\theta^0}/\smallradius)}/{\contracasym^2}}$
and in the limit $\contracasym \rightarrow 0$, this bound
degenerates to $\infty$ for any finite $\smallradius$.  Indeed,
the bound~\eqref{EqnUnbalancedPop} from Theorem~\ref{ThmUnbalanced}
simply states that the population EM operator is non-expansive for balanced
mixtures~($\contracasym=0$), 
and does not provide any particular rate of convergence for this case.
It turns out that the EM algorithm is worse in the balanced case, both in terms
of the optimization speed and in terms of the statistical rate.  
This slower statistical rate is in accord with existing results for
the MLE in over-specified mixture models~\cite{Chen1992}; the novel
contribution here is the rigorous analysis of the analogous behavior
for the EM algorithm.


\subsection{Behavior of EM for balanced mixtures} 
\label{ssub:population_em_for_balanced_mixtures}

In this section, we first provide a sharp characterization of the
algorithmic rate of convergence of the population EM update for the balanced
fit (see Section~\ref{ssub:slow_convergence_of_population_em}).  We
then provide sharp bound for the statistical rate for
the sample EM updates (cf. Section~\ref{SecSubBoundsSampleEM}).


\subsubsection{Slow convergence of population EM}
\label{ssub:slow_convergence_of_population_em}
We now analyze the behavior of the population EM operator
for the balanced fit.  We show that it is globally convergent, albeit
with a contraction parameter that depends on $\theta$, and degrades
towards $1$ as $\|\theta\|_2 \rightarrow 0$.  Our statement involves
the constant $p \defn \Prob(\abss{X} \leq 1) + \frac{1}{2}
\Prob(\abss{X}>1)$, where $X \sim \Ncal(0,1)$ denotes a standard
normal variate. (Note that $p<1$.)

\begin{theorem} 
\label{thm:pop_over}
Suppose that we fit a balanced instance ($\weight=\frac{1}{2}$)
of the mixture model~\eqref{EqModelFitNew} to $\NORMAL(0,
\sd^2 I_d)$ data. Then the population EM operator~\eqref{EqnPopMupdate}  
${\theta \mapsto \updateM(\theta)}$
has the following properties:
\begin{enumerate}[label=(\alph*)]
  \item For all non-zero
$\theta$, we have
\begin{subequations}
  \begin{align}
    \label{EqnPopEMUpper}
    \frac{\enorm{\updateM(\theta)}}{\enorm{\theta}} & \leq \gamup(\theta)
    \; \defn 1 - p + \frac{p}{1 + \frac{\enorm{\theta}^2}{2\sd^2}} \; < \;
    1.
\end{align}
  \item For all non-zero $\theta$ such that $\enorm{\theta}^2
  \leq
\frac{5\sd^2}{8}$, we have
\begin{align}
  \label{EqnPopEMLower}
\frac{\enorm{\updateM(\theta)}}{\enorm{\theta}} & \geq \gamlow(\theta)
\; \defn \; \frac{1}{1 + \frac{2 \enorm{\theta}^2}{\sd^2}}.
\end{align}
\end{subequations}
\end{enumerate}
\end{theorem}

\noindent See Appendix~\ref{sub:proof_of_theorem_thm:pop_over} for the
proof of Theorem~\ref{thm:pop_over}.  \\

The salient feature of Theorem~\ref{thm:pop_over} is that the
contraction coefficient $\gamup(\theta)$ is not globally bounded away
from $1$ and in fact satisfies $\lim_{\theta\rightarrow 0}
\gamup(\theta) = 1$.  In conjunction with the lower
bound~\eqref{EqnPopEMLower}, we see that
\begin{align}
  \label{eq:pop_approx}
\frac{\enorm{\updateM(\theta)}}{\enorm{\theta}} & \asymp \parenth{1 -
  \frac{\enorm{\theta}^2}{\sd^2}} \qquad \mbox{for small
  $\enorm{\theta}$.}
\end{align}
This precise contraction behavior of the population EM operator is in 
accord with that of the simulation study in Figure~\ref{fig:population_em}(b). 

The preceding results show that the population EM updates should
exhibit two phases of behavior.  In the first phase, up to a
relatively coarse accuracy of the order $\sigma$, the iterates exhibit
geometric convergence.  Concretely, we are guaranteed to have
$\enorm{\theta^{\Tone}} \leq \sqrt{2} \sigma$ after running the
algorithm for $\Tone \defn \frac{\log(\enorm{\theta^0}^2 /
  (2\sd^2))}{\log(2/(2-p))}$ steps.  In the second phase, as the error
decreases from $\sqrt{2} \sigma$ to a given $\smallradius \in \big
(0, \sqrt{2} \sigma \big)$, the convergence rate becomes
sub-geometric; concretely, we have
\begin{align}
\label{eq:balanced_pop_rate}
  \enorm{\theta^{\Tone + t}} & \le \,
  \smallradius\quad \text{for} \quad t \geq
  \frac{c\sd^2}{\smallradius^2} \log(\sd/\smallradius).
\end{align}
Note that the conclusion~\eqref{eq:balanced_pop_rate} shows that for small
enough $\smallradius$, the population EM takes
$\Theta(\log(1/\smallradius)/\smallradius^2)$ steps to find
$\smallradius$-accurate estimate of $\thetastar=0$.  This rate is
extremely slow compared to the geometric rate
$\mathcal{O}(\log(1/\smallradius))$ derived for the unbalanced
mixtures in Theorem~\ref{ThmUnbalanced}. Hence, the slow rate establishes
a qualitative difference in the behavior of the EM
algorithm between the balanced and unbalanced setting.

Moreover, the sub-geometric rate of EM in the balanced case is also in stark
contrast with the favorable behavior of EM for the exact-fitted
settings analyzed in past work.  \mbox{Balakrishnan et
  al.~\cite{Siva_2017}} showed that when the EM algorithm is used to
fit a two-component Gaussian mixture with sufficiently large
value of $\frac{\enorm{\theta^\star}}{\sigma}$ 
(known as the high signal-to-noise ratio, or high SNR for short), the
population EM operator is contractive, and hence geometrically
convergent, within a neighborhood of the true parameter $\thetastar$.
In a later work on the two-component balanced mixture fit model,
Daskalakis~\etal~\cite{Daskalakis_colt2017} showed that the
convergence is in fact geometric for any non-zero value of the SNR.
The model considered in Theorem~\ref{thm:pop_over} can be seen as
the limiting case of weak signal for a two mixture model---which degenerates
to the Gaussian distribution when the SNR becomes exactly zero. For such
a limit, we observe that the fast convergence of population EM sequence
no longer holds.

\subsubsection{Upper and lower bounds on sample EM}
\label{SecSubBoundsSampleEM}
We now turn to the statements of upper and lower bounds on the rate of
the sample EM iterates for the balanced fit on Gaussian data. 
We begin with an upper bound, which involves the previously defined function
$\gamup(\theta)\; \defn 1 - p + p/ \big({1 + \frac{\enorm{\theta}^2}{2\sd^2}}\big)$.
\begin{theorem} 
	\label{theorem:convergence_rate_sample_EM}
  Consider the sample EM updates $\theta^t = \updateM_\obs(\theta^{t-1})$
  for the balanced instance ($\weight=\frac{1}{2}$) of 
  the mixture model~\eqref{EqModelFitNew}
  based on $n$ i.i.d. $\NORMAL(0, \sd^2 I_d)$ samples.
  Then, there exist universal constants $\braces{\unicon'_k}_{k=1}^4$
	such that for any scalars $\powerr \in (0, \frac{1}{4})$ and $\delta
	\in (0, 1)$, any sample size \mbox{$\obs \geq \unicon'_1 (\dims
		+\log(\log(1/\powerr)/\delta)) \;$} and any iterate number \mbox{$t \geq  \unicon'_2
		\log\frac{\Vert{\theta^0}\Vert^2\obs}{\sd^2\dims}+ \unicon'_3
		\left(\frac{\obs}{\dims}\right)^{\frac{1}{2}-2\powerr}
		\log(\frac{\obs}{\dims}) \log(\frac{1}{\powerr})$,} 
	we have
	\begin{align}
	\label{eqn:theorem_sample_em_balanced_expression}
	\|\theta^t \|_{2} \leq \left [\enorm{\theta^0} \cdot \prod_{j=0}^{t-1}
	\gamup(\theta^j) \right] + \unicon'_4 \sd \parenth{\frac{\sd^2
			(\dims + \log \frac{\log(4 /
            \smallradius)}{\delta})}{\obs} }^{\frac{1}{4} - \powerr},
	\end{align}
	with probability at least $1-\delta$.
\end{theorem}
\noindent See Section~\ref{sec:new_techniques_for_analyzing_sample_em} for
a discussion of the techniques employed to prove this theorem.
The detailed proof is provided in Appendix~\ref{sub:proof_of_theorem_thmsampleoverfit}, where we also provide some
more details on the definitions of these constants.

As we show in our proofs, once the iteration number $t$ satisfies the
lower bound stated in the theorem, the second term on the right-hand
side of the bound~\eqref{eqn:theorem_sample_em_balanced_expression}
dominates the first term; therefore, from this point onwards, the the
sample EM iterates have Euclidean norm of the order $(d/n)^{\frac{1}{4} -
	\powerr}$. Note that $\powerr \in (0, \frac{1}{4})$ can be chosen
arbitrarily close to zero, so at the expense of increasing the lower
bound on the number of iterations $t$ by a logarithmic factor 
$\log (1/\powerr)$,
we can obtain
rates arbitrarily close to $(d/n)^{\frac{1}{4}}$.

We note that earlier studies of parameter estimation for
over-specified mixtures, in both the frequentist~\cite{Chen1992} and
Bayesian settings~\cite{Ishwaran-2001,Nguyen-13}, have derived a rate
of $n^{-\frac{1}{4}}$ for the global maximum of the log likelihood.
To the best of our knowledge,
Theorem~\ref{theorem:convergence_rate_sample_EM} is the first
non-asymptotic algorithmic result that shows that such rates apply to
the fixed points and dynamics of the EM algorithm, which need not
converge to the global optima.

The preceding discussion was devoted to an upper bound on sample EM for
the balanced fit. Let us now match this upper bound, at least in the univariate
case $\dims = 1$, by showing that any non-zero fixed point of the sample
EM updates has Euclidean norm of the order $\obs^{-\frac{1}{4}}$.  In 
particular, we prove the following lower bound.
\begin{theorem}
	\label{theorem:sample_lower_bound}
	There are universal positive constants $\unicon, \unicon'$ such that for any
	non-zero solution $\fixedpointtheta$ to the sample EM fixed-point
	equation $\theta = \mupdate_\obs(\theta)$ for the balanced mixture
  fit, we have
	\begin{align}
	\Prob \brackets{\vert{\fixedpointtheta} \vert \geq \unicon \,
		\obs^{-\frac{1}{4}}} \geq \unicon'.
	\end{align}
\end{theorem}
\noindent 
See Appendix~\ref{sub:proof_of_theorem_theorem:lower_bound_fixed_point}
for the proof of this theorem.

Since the iterative EM scheme converges only to one of its fixed
points, the theorem shows that one cannot obtain a high-probability
bound for any radius smaller than $n^{-\frac{1}{4}}$. As a
consequence, with constant probability, the radius of convergence
$n^{-\frac{1}{4}}$ for sample EM convergence in
Theorem~\ref{theorem:convergence_rate_sample_EM} for the univariate
setting is tight.


\section{New techniques for sharp analysis of sample EM} 
\label{sec:new_techniques_for_analyzing_sample_em}

In this section, we highlight the new proof techniques introduced in
this work that are required to obtain the sharp characterization of
the sample EM updates in the balanced case
(Theorem~\ref{theorem:convergence_rate_sample_EM}).  We begin in
Section~\ref{SecSuboptimal} by elaborating that a direct application
of the previous frameworks leads to sub-optimal statistical rates for
sample EM in the balanced fit.  This sub-optimality motivates the
development of new methods for analyzing the behavior of the sample EM
iterates, based on an \emph{annulus-based localization argument} over a 
sequence of epochs, which we sketch out in \mbox{Sections~\ref{SecLocalize}
and \ref{SecSubKeyRecursion}.}  We remark that our novel techniques,
introduced here for analyzing EM with the balanced fit, are likely to
be of independent interest. We believe that they can potentially be
extended to derive sharp statistical rates in other settings when the
algorithm under consideration does not exhibit an geometrically fast convergence.


\subsection{A sub-optimal guarantee}
\label{SecSuboptimal}
Let us recall the set-up for the procedure suggested by Balakrishnan
et al.~\cite{Siva_2017}, specializing to the case where the true
parameter $\thetastar = 0$, as in our specific set-up.  Using the
triangle inequality, the norm of the sample EM iterates $\theta^{t + 1}
= \updateM_\obs(\theta^t)$ can be upper
bounded by a sum of two terms as follows:
\begin{align}
\label{EqnBreak}
  \enorm{\theta^{t + 1}} = \enorm{\updateM_\obs(\theta^{t} )} \leq
  \enorm{\updateM_\obs(\theta^t) - \updateM(\theta^t)} +
  \enorm{\updateM(\theta^t)}
\end{align}
for all $t \geq 0$. The first term on the right-hand side corresponds to the deviations
between the sample and population EM operators, and can be controlled
via empirical process theory.  The second term corresponds to
the behavior of the (deterministic) population EM operator, as applied
to the sample EM iterate $\theta^t$, and needs to be controlled via a
result on population EM.

Theorem~2 from Balakrishnan et al.~\cite{Siva_2017} is based on
imposing generic conditions on each of these two terms, and then using
them to derive a generic bound on the sample EM iterates.  In the
current context, their theorem can be summarized as follows.  For
given tolerances $\delta \in (0, 1)$, $\smallradius > 0$ and starting
radius $r > 0$, suppose that there exists a function \mbox{$\varepsilon(n,
\delta) > 0$}, decreasing in terms of the sample size $n$,
and a contraction coefficient \mbox{$\kappa \in (0, 1)$} such that
\begin{subequations}
\begin{align}
  \label{EqnSivaConditions}
  \sup_{\enorm{\theta} \geq \smallradius}
  \frac{\enorm{M(\theta)}}{\enorm{\theta}}\!\leq\!\kappa\text{ and }
  \Prob\brackets{\sup_{\enorm{\theta}\leq r}
    \enorm{M_n(\theta)\!-\!M(\theta)} \!\leq\! \varepsilon(n, \delta)} \!\geq\!
  1\!-\!\delta.
\end{align}
Then for a sample size $\obs$ sufficiently large and $\smallradius$
sufficiently small to ensure that
\begin{align}
  \label{EqnSivaBounds}
\smallradius \; \stackrel{(i)}{\leq} \; \frac{\varepsilon(n,
  \delta)}{1 - \kappa} \; \stackrel{(ii)}{\leq} \; r,
\end{align}
\end{subequations}
the sample EM iterates are guaranteed to converge to a ball of radius
${\varepsilon(n, \delta)}/({1-\kappa})$ around the true parameter
$\thetastar = 0$.

In order to apply this theorem to the current setting, we need to
specify a choice of $\varepsilon(n, \delta)$ for which the bound on
the empirical process holds.  The following auxiliary result provides
such control for us:
\begin{lemma}
\label{lemma:bound_EM_operators} 
There exists universal positive constants $c_1$ and $c_2$ such
that for any positive radius $r$, any threshold $\delta \in (0, 1)$,
and any sample size \mbox{$\obs \geq c_2\dims\log(1/\delta)$}, we have
\begin{align}
\label{EqnLemmaOneBound}
\Prob \brackets{\sup \limits_{\|\theta\|_2 \leq r}
  \enorm{M_{n}(\theta) - M(\theta)} \leq c_1 \sigma(\sigma r+\contracasym
  )
  \sqrt{\frac{d + \log(1/\delta)}{n}}} \geq 1- \delta,
\end{align}
where $\contracasym =\abss{1-2\weight}$ denotes the imbalance in the
mixture fit~\eqref{EqModelFitNew}.
\end{lemma}
\noindent The proof of this lemma is based on Rademacher complexity
arguments; see Appendix~B.1 for
the details. 

With the choice $r = \enorm{\theta^0}$,
Lemma~\ref{lemma:bound_EM_operators} guarantees that the second
inequality in line~\eqref{EqnSivaConditions} holds with
$\varepsilon(n, \delta) \lesssim \sd^2 \enorm{\theta^0}\sqrt{d/n}$.
On the other hand, Theorem~\ref{thm:pop_over} implies that for any
$\theta$ such that $\enorm{\theta} \geq \smallradius$, we have that
population EM is contractive with parameter bounded above by
$\kappa(\smallradius) \asymp 1 - \smallradius^2$.  In order to
satisfy inequality (i) in equation~\eqref{EqnSivaBounds}, we solve the
equation ${\varepsilon(n, \delta)}/({1-\kappa(\smallradius)}) =
\smallradius$.  Tracking only the dependency on $d$ and $n$, we
obtain\footnote{Moreover, with this choice of $\smallradius$,
  inequality (ii) in equation~\eqref{EqnSivaBounds} is satisfied with
  a constant $r$, as long as $n$ is sufficiently large relative to
  $d$.}
\begin{align}
\label{eq:vanilla_rates}
  \frac{\sqrt{d/n}}{\smallradius^2} = \smallradius \quad
  \Longrightarrow \quad \smallradius = \order{(d/n)^{\frac{1}{6}}},
\end{align}
which shows that the Euclidean norm of the sample EM iterate is
bounded by a term of order $(d/n)^{\frac{1}{6}}$.

While this rate is much slower than the classical $(d/n)^{\frac{1}{2}}$ rate
that we established in the unbalanced case, it does not coincide with
the $n^{-\frac{1}{4}}$ rate that we obtained in
Figure~\ref{fig:snr_effect}(b) for balanced setting with $d = 1$.
Thus, the proof technique based on the framework of Balakrishnan et
al.~\cite{Siva_2017} appears to be non-optimal.  The sub-optimality of
this approach necessitates the development of a more refined
technique. Before sketching this technique, we now quantify
empirically the convergence rate of sample EM in terms of both
dimension $\dims$ and sample size $\obs$ for the balanced mixture fit.
In Figure~\ref{fig:sample_balanced_rates}, we summarize the results of
these experiments. The two panels in the figure exhibit that the error
in the sample EM estimate scales as $(d/n)^{\frac{1}{4}}$, thereby
providing further numerical evidence that the preceding approach
indeed led to a sub-optimal result.
\begin{figure}[h]
  \begin{center}
    \begin{tabular}{cc}
\widgraph{0.45\textwidth}
         {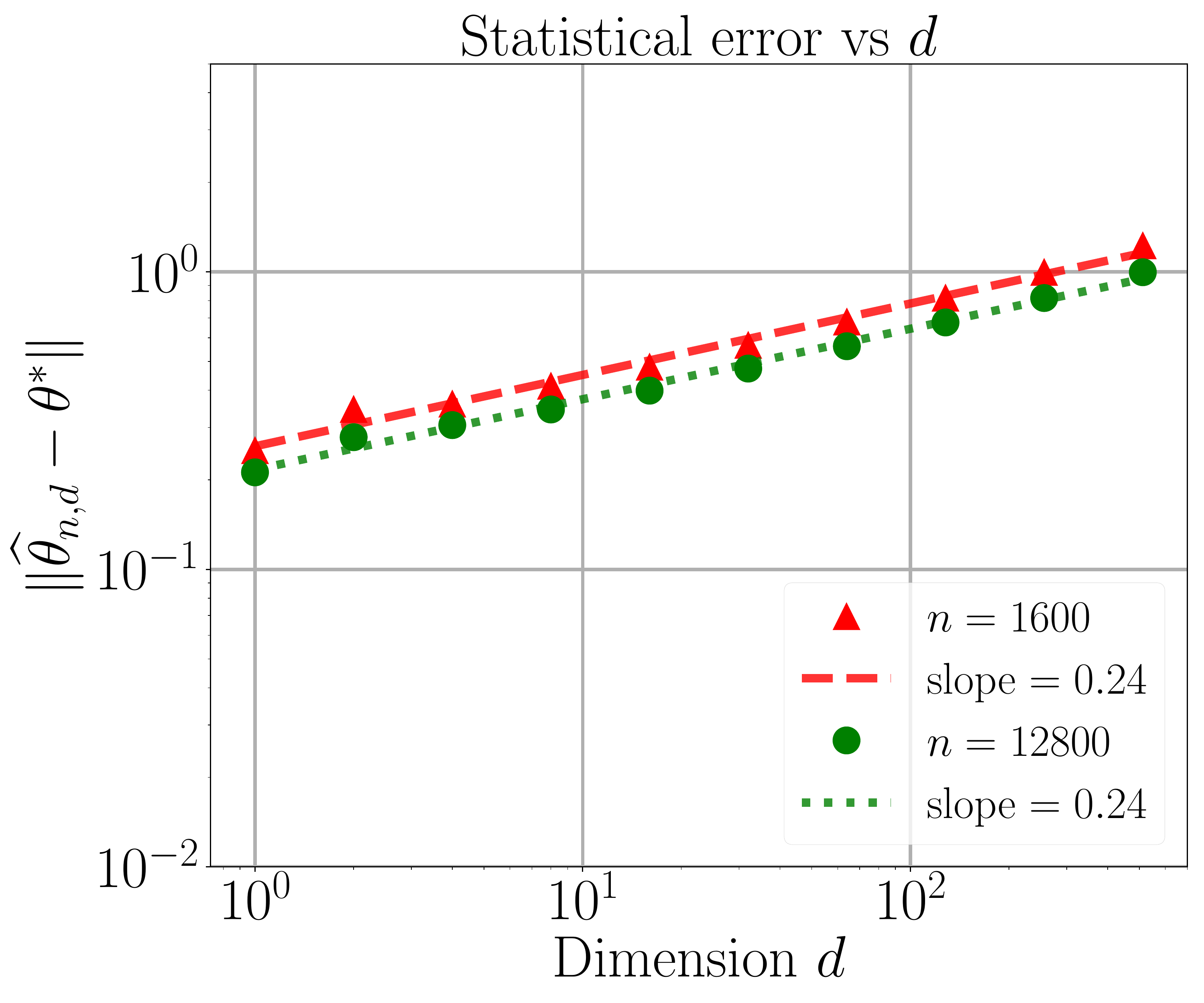}
         & \widgraph{0.45\textwidth}
         {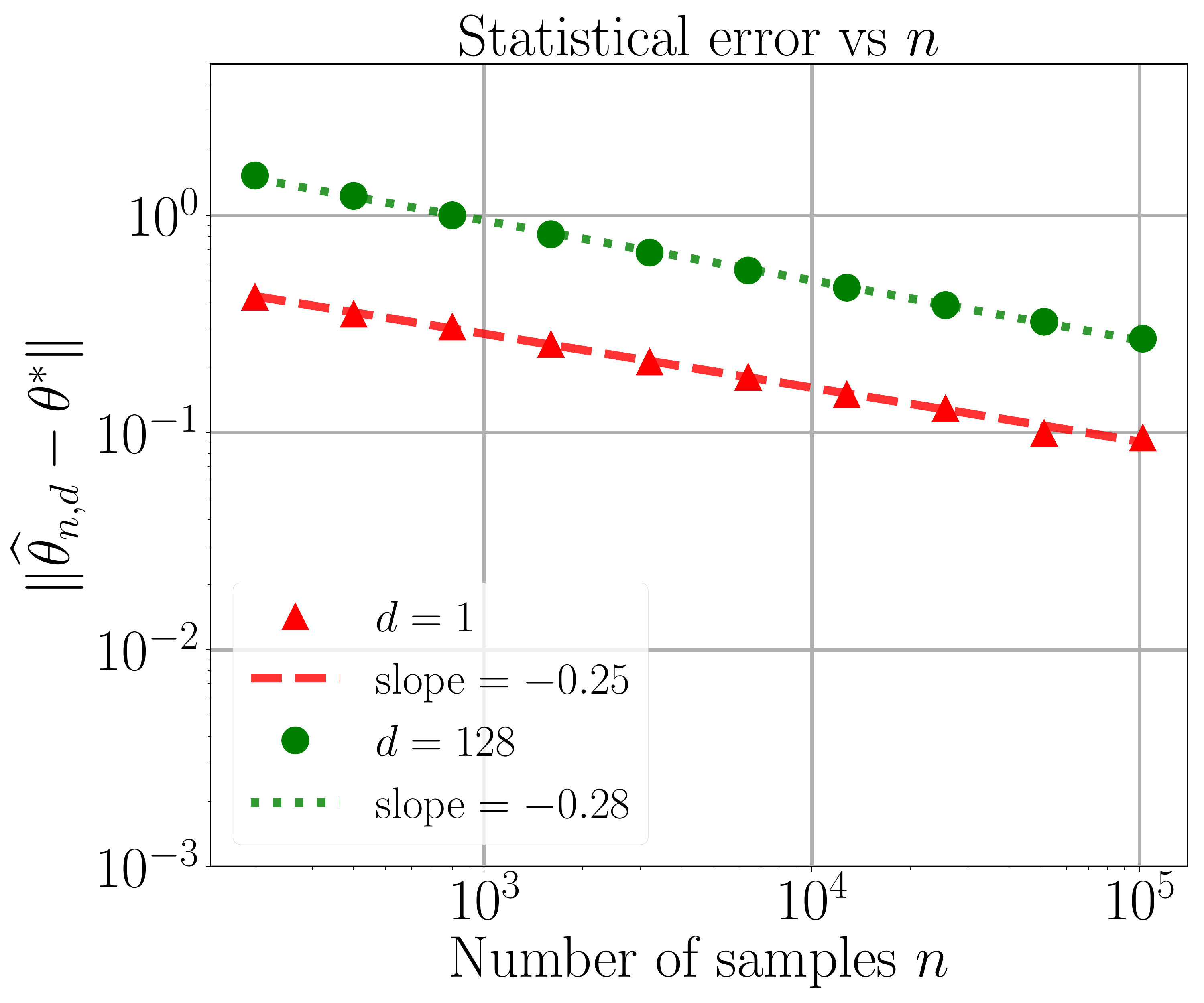}
         \\
         (a) & (b)
    \end{tabular}
 \caption{
Scaling of the Euclidean error $\|\widehat \theta_{n,d} -
      \thetastar\|_2$ for EM estimates $\widehat \theta_{n,d}$
      computed using the balanced ($\weight=\frac{1}{2}$) mixture-fit~\eqref{EqModelFit}.
      Here the true data distribution is $\NORMAL(0, I_d)$, 
      i.e., $\thetastar=0$, and
      $\widehat\theta_{n, d}$ denotes the EM iterate upon convergence
      when we fit a balanced mixture with $n$ samples in $d$ dimensions.
      (a)~Scaling with respect to $\dims$ for $\obs \in
      \braces{1600, 12800}$.  (b)~Scaling with respect to $\obs$ for
      $\dims \in \braces{1, 128}$. We ran experiments for several
      other pairs of $(n, d)$ and the conclusions were the same.  Clearly, the
   empirical results suggest a scaling of order $(\dims/\obs)^{\frac{1}{4}}$
   for the final iterate of sample-based EM. }
 \label{fig:sample_balanced_rates}
  \end{center}
\end{figure}

\begin{figure}[t]
	\begin{center}
		\begin{tabular}{c|c}
			\widgraph{0.48\textwidth}
			{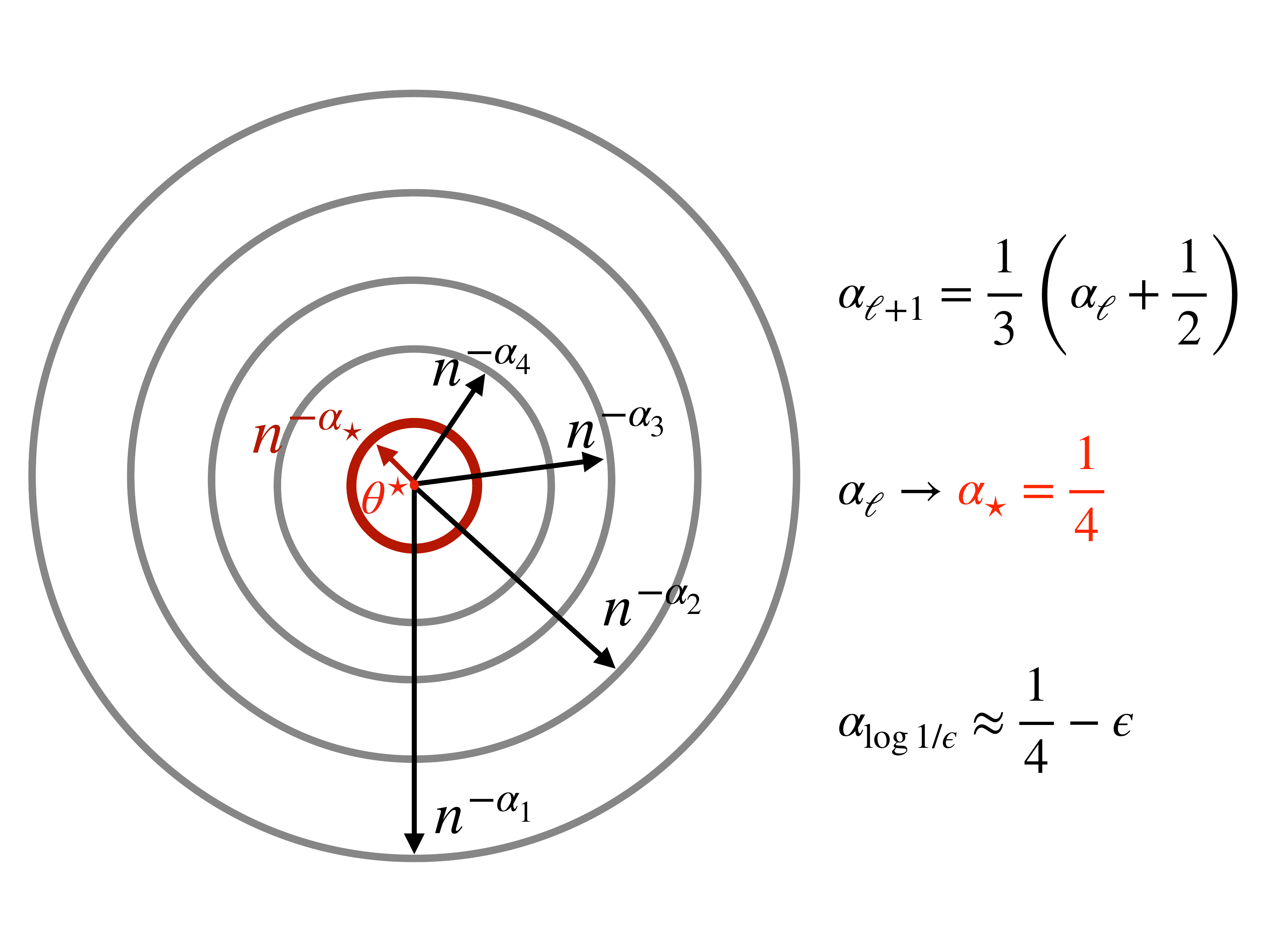}
			&
			\widgraph{0.48\textwidth}
			{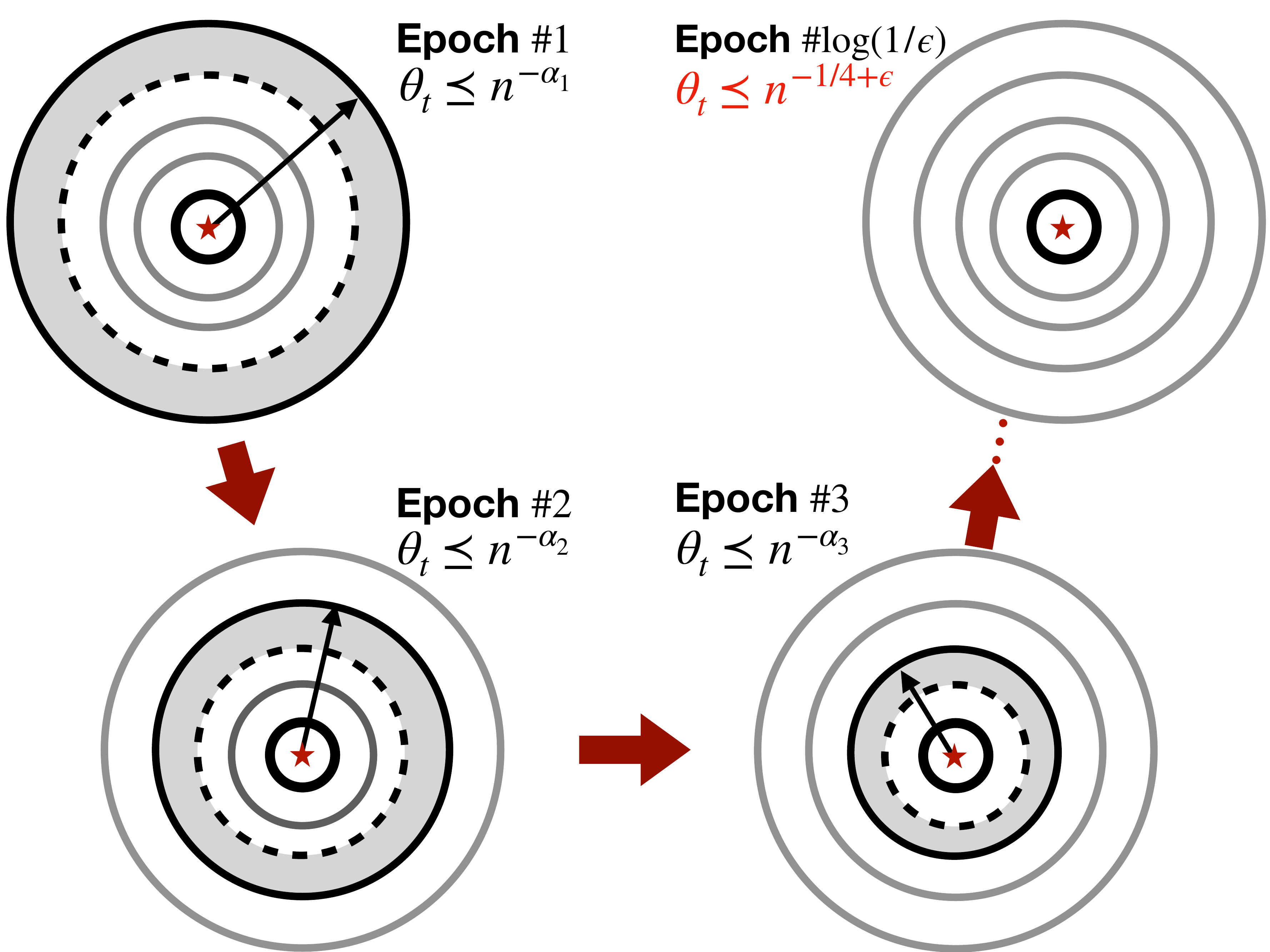}\\
			(a) & (b)
		\end{tabular}
		\caption{Illustration of the annulus-based-localization argument part
    (I):
    Defining the epochs or equivalently the annuli.
			(a) Outer radius for the $\ind$-th epoch is given by $n^{-\seqalpha{\ind}}$
			(tracking dependency only on $\obs$).
			(b) For any given epoch $\ind$, we analyze the behavior of the EM 
			sequence  $\theta^{t+1}=M_n(\theta^t)$, when $\theta^t$ lies in the 
			annulus around $\thetastar$ with inner and outer radii given by 
			$n^{-\seqalpha{\ind+1}}$, and $n^{-\seqalpha{\ind}}$, respectively.
			We prove that EM iterates move from one epoch to the next epoch (e.g. 
			epoch $\ell$ to epoch $\ell + 1$) after at most $\sqrt{n}$ iterations.
			Given the definition of $\seqalpha{\ind}$, we see that the 
			inner and outer radii of the aforementioned annulus converges linearly to
			$n^{-\frac{1}{4}}$. Consequently, after at most
			$\log(1/\powerr)$ epochs (or $\sqrt{n}\log(1/\powerr)$ 
			iterations), the EM iterate lies in a ball of radius
			$n^{-{1}/{4}+\powerr}$ around $\thetastar$.
			We illustrate the one-step dynamics in any given annulus
			in Figure~\ref{fig:illus_local_arg_one_epoch}.
		}
		\label{fig:illus_local_arg}
	\end{center}
\end{figure}
\begin{figure}[h]
	\begin{center}
		\begin{tabular}{c}
			\widgraph{.60\textwidth}{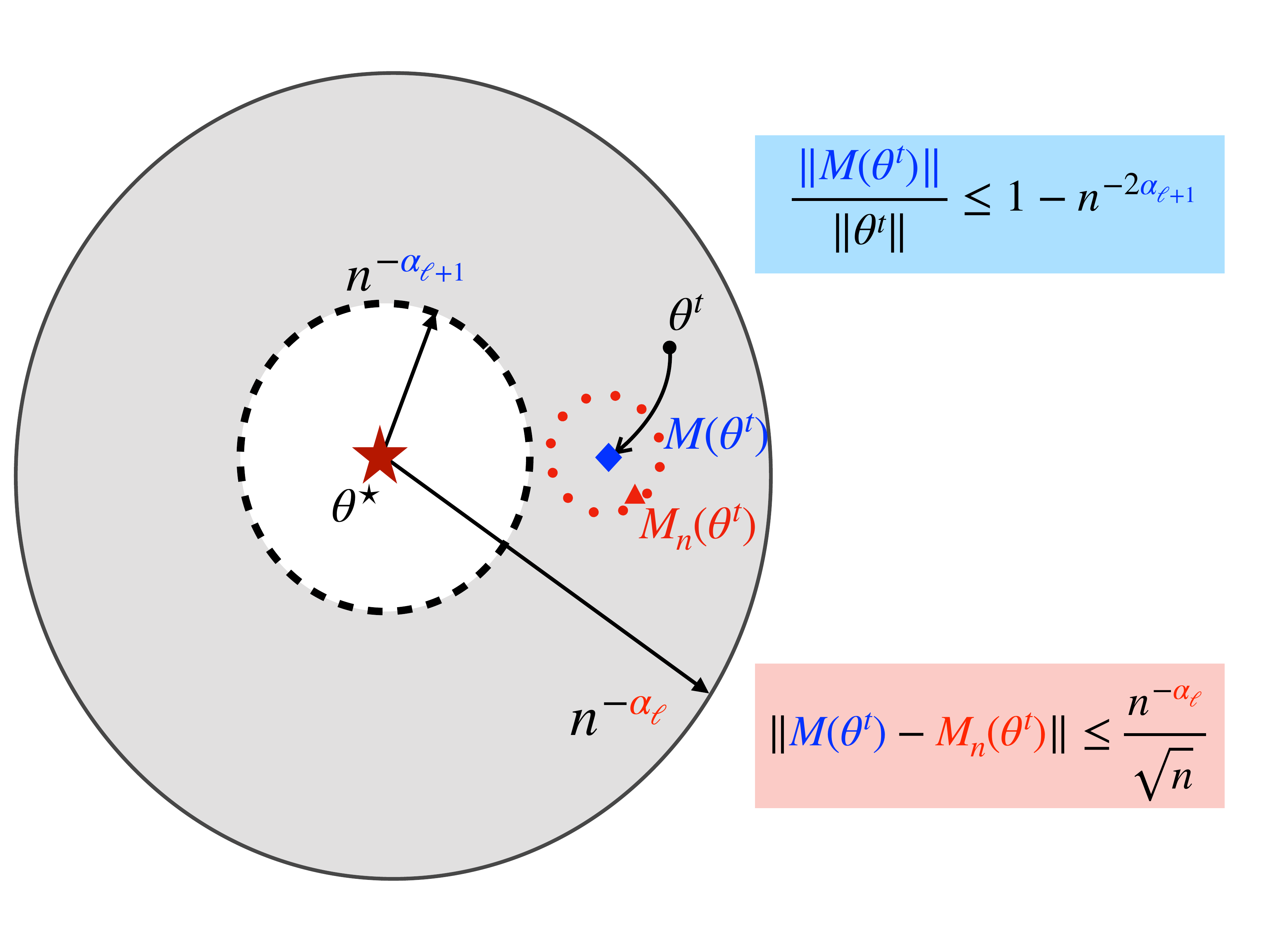}
		\end{tabular}
		\caption{Illustration of the annulus-based-localization argument part
    (II):
      Dynamics of EM in the $\ind$-th epoch or equivalently the annulus
      $\obs^{-\alpha_{\ell+1}}\leq \|\theta^t-\thetastar\|_2 \leq \obs^
      {-\alpha_\ell}.$
			For a given epoch $\ind$, we analyze the behavior of the EM 
			sequence  $\theta^{t+1}=M_n(\theta^t)$, when $\theta^t$ lies in the 
			annulus with inner and outer radii given by $n^{-\seqalpha{\ind+1}}$,
      and $n^{-\seqalpha{\ind}}$, respectively.
			In this epoch, the population EM operator $\updateM(\theta^t)$
			contracts with a contraction coefficient that depends on 
			$n^{-\alpha_{\ell + 1}}$, which is the inner
			radius of the disc, while the perturbation error
			${\|\updateM_\obs(\theta^t) - \updateM(\theta)\|_2}$
			between the sample and population EM operators
			depends on $n^{-\alpha_\ell}$, which is the outer radius of the disc.
			Overall, we prove that $M_n$ is non-expansive and after at most
			$\sqrt{n}$ steps, the sample EM updates move from epoch $\ind$ to
			epoch $\ind+1$.
		}
		\label{fig:illus_local_arg_one_epoch}
	\end{center}
\end{figure}


\subsection{Annulus-based localization over epochs}
\label{SecLocalize}

Let us try to understand why the preceding argument led to a
sub-optimal bound. In brief, its ``one-shot'' nature contains two
major deficiencies. First, the tolerance parameter $\smallradius$ is
used both (a) for measuring the contractivity of the updates, as in
the first inequality in equation~\eqref{EqnSivaConditions}, \emph{and}
(b) for determining the final accuracy that we achieve.  At earlier
phases of the iteration, the algorithm will converge \emph{more
  quickly} than suggested by the worst-case analysis based on the
final accuracy.  A second deficiency is that the argument uses the
radius $r$ only once, setting it to a constant to reflect the
initialization $\theta^0$ at the start of the algorithm.  This means
that we failed to ``localize'' our bound on the empirical process in
Lemma~\ref{lemma:bound_EM_operators}.  At later iterations of the
algorithm, the norm $\enorm{\theta^t}$ will be smaller, meaning that
the empirical process can be more tightly controlled.  We note that
ideas of localizing the radius $r$ for an empirical process plays a
crucial role in obtaining sharp bounds on the error of $M$-estimation
procedures~\cite{Vandegeer-2000, Bar05, Kolt06, Wai19}.

A novel aspect of the localization argument in our setting is the use
of an annulus instead of a ball.  In particular, we analyze the
iterates from the EM algorithm assuming that they lie within a
pre-specicied annulus, defined by an inner and an outer radius. On one
hand, the outer radius of the annulus helps to provide a sharp control
on the perturbation bounds between the population and sample
operators.  On the other hand, the inner radius of the annulus is used
to tightly control the algorithmic rate of convergence.

We now summarize our key arguments. The entire sequence of sample EM
iterations is broken up into a sequence of different epochs.  During
each epoch, we localize the EM iterates to an annulus.  In more
detail:
\begin{itemize}
\item We index epochs by the integer $\ell = 0, 1, 2, \ldots$, and
  associate them with a sequence $\{\alpha_\ell\}_{\ell \geq 0}$ of
  scalars in the interval $[0, \frac{1}{4}]$.  The input to epoch
  $\ell$ is the scalar $\alpha_\ell$, and the output from epoch $\ell$
  is the scalar $\alpha_{\ell + 1}$.
\item The $\ell$-th epoch is defined to be the set of all iterations
  $t$ of the sample EM algorithm such that the sample EM iterate
  $\theta^t$ lies in the following annulus:
  \begin{align}
    \left(\frac{\dims}{\obs}\right)^{\alpha_{\ell+1}}\leq
  \|\theta^t-\thetastar\|_2 \leq \left(\frac{\dims}{\obs} \right)^{\alpha_\ell}.
  \end{align}
  We establish that the sample-EM operator is non-expansive so that
  each epoch is well-defined (and that subsequent iterations can only
  correspond to subsequent epochs).
  \item Upon completion of epoch $\ell$ at iteration $T_\ell$, the EM
    algorithm returns an estimate $\theta^{T_\ell}$ such that
    $\|\theta^{T_\ell}\|_2 \precsim (d/n)^{\alpha_{\ell+1}}$, where
    \begin{align}
\label{EqnKeyRecursion}
\alpha_{\ell+1} = \frac{1}{3} \alpha_{\ell} + \frac{1}{6}.
    \end{align}
Note that the new scalar $\alpha_{\ell + 1}$ serves as the input to
epoch $\ell + 1$.
\end{itemize}
The recursion~\eqref{EqnKeyRecursion} is crucial in our analysis: it
tracks the evolution of the exponent acting upon the ratio $d/n$, and
the rate $(d/n)^{\alpha_{\ell+1}}$ is the bound on the Euclidean norm
of the sample EM iterates achieved at the end of \mbox{epoch $\ell$.}

A few properties of the recursion~\eqref{EqnKeyRecursion} are worth
noting.  First, given our initialization $\alpha_0 = 0$, we see that
$\alpha_1 = \frac{1}{6}$, which agrees with the outcome of our
one-step analysis from above.  Second, as the recursion is iterated,
it converges from below to the fixed point $\alpha^* = \frac{1}{4}$.
Thus, our argument will allow us to prove a bound arbitrarily close to
$(d/n)^{\frac{1}{4}}$, as stated formally in
Theorem~\ref{theorem:convergence_rate_sample_EM} to follow.  Refer to
Figures~\ref{fig:illus_local_arg} and
\ref{fig:illus_local_arg_one_epoch} for an illustration of the
definition of these annuli, epochs and the associated conclusions.


\subsection{How does the key recursion~\eqref{EqnKeyRecursion} arise?}
\label{SecSubKeyRecursion}

Let us now sketch out how the key recursion~\eqref{EqnKeyRecursion}
arises.  Consider epoch $\ind$ specified by input \mbox{$\alpha_{\ind} <
\frac{1}{4}$}, and consider an iterate $\theta^t$ in the following annulus:
$ \enorm{\theta^t} \in [(\dims/\obs)^{\alpha_{\ind+1}}, (\dims/\obs)^{\alpha_
{\ind}}]$.
We begin by proving that this initial condition ensures that
$\enorm{\theta^t}$ is less than level $(\dims/\obs)^{\alpha_{\ind}}$
for all future iterations; for details, see Lemma~\ref{LemNonExpansive} stated in the
Appendix.  Given this guarantee, our second step is to make use of the
inner radius of the considered annulus to apply
Theorem~\ref{thm:pop_over} for the population EM operator, for all
iterations $t$ such that $\enorm{\theta^t} \geq
(\dims/\obs)^{\alpha_{\ind+1}}$.  Consequently, for these iterations,
we have
\begin{subequations}
  \begin{align}
    \label{EqnTwinA}
\enorm{\updateM(\theta^t)} & \leq \Big( 1 - p + \frac{p}{1 +
  \frac{\enorm{\theta}^2}{2\sd^2}}\Big) \enorm{\theta^t} \nonumber \\
   & \precsim (1
-(\dims/\obs)^{2 \alpha_{\ind + 1}}) (\dims/\obs)^{\alpha_{\ind}} \leq
\gammamin \; \left( \frac{\dims}{\obs} \right)^{\alpha_{\ind}},
\end{align}
where $\gammamin \defn e^{- (\dims/\obs)^{2\alpha_{\ind + 1}}}$.  On the other
hand, using the outer radii of the annulus and applying Lemma~\ref{lemma:bound_EM_operators}
for this epoch, we obtain that
\begin{align}
  \label{EqnTwinB}
  \enorm{\updateM_\obs(\theta^t)-\updateM(\theta^t)} \precsim
  \parenth{\frac{\dims}{\obs}}^{\alpha_{\ind}} \sqrt{\frac{\dims}{\obs}} =
  \parenth{\frac{\dims}{\obs}}^{\alpha_{\ind} + 1/2},
\end{align}
\end{subequations}
for all $t$ in the epoch.  Unfolding the basic triangle
inequality~\eqref{EqnBreak} for $T$ steps, we find that
\begin{align*}
\enorm{\theta^{t + T}} & \leq \enorm{\updateM_\obs(\theta^t) -
    \updateM(\theta^t)} (1 + \gammamin + \ldots + \gammamin^{T -
    1}) + \gammamin^T \enorm{\theta_t} \\
& \leq \frac{1}{1 - \gammamin} \enorm{\updateM_\obs(\theta^t) -
  \updateM(\theta^t)} + e^{-T(\dims/\obs)^{2\alpha_{\ind + 1}}}
(\dims/\obs)^{\alpha_{\ind}}.
\end{align*}
The second term decays exponentially in $T$, and our analysis shows
that it is dominated by the first term in the relevant regime of
analysis.  Examining the first term, we find that $\theta^{t+T}$ has
Euclidean norm of the order
\begin{align}
\label{EqnRDefn}
\enorm{\theta^{t+T}} & \precsim \frac{1}{1 - \gammamin}
\enorm{\updateM_\obs(\theta^t) - \updateM(\theta^t)} \approx
\underbrace{ \parenth{\frac{\dims}{\obs}}^{- 2\alpha_{\ind + 1}}
  \parenth{\frac{\dims}{\obs}}^{\alpha_{\ind} + 1/2}}_{ = \,: \; r}.
\end{align}
The epoch is said to be complete once $\enorm{\theta^{t+T}} \precsim
\left(\frac{\dims}{\obs} \right)^{\alpha_{\ind + 1}}$.  Disregarding constants,
this condition is satisfied when $r = \left(\frac{\dims}{\obs}
\right)^{\alpha_{\ind + 1}}$, or equivalently when
\begin{align*}
\parenth{\frac{\dims}{\obs}}^{- 2\alpha_{\ind + 1}}
\parenth{\frac{\dims}{\obs}}^{\alpha_{\ind} + 1/2} & =
\left(\frac{\dims}{\obs} \right)^{\alpha_{\ind + 1}}.
\end{align*}
Viewing this equation as a function of the pair $(\alpha_{\ind + 1},
\alpha_{\ind})$ and solving for $\alpha_{\ind + 1}$ in terms of
$\alpha_{\ind}$ yields the recursion~\eqref{EqnKeyRecursion}. Refer
to Figure~\ref{fig:illus_local_arg_one_epoch}
for a visual illustration of the localization argument summarized above
for a given epoch.

Of course, the preceding discussion is informal, and there remain many
details to be addressed in order to obtain a formal proof.  We refer
the reader to Appendix~\ref{sub:proof_of_theorem_thmsampleoverfit} for
the complete argument.


\section{Generality of results and future work} 
\label{sec:discussion}

Thus far, we have characterized the behavior of the EM algorithm for
different settings of over-specified location Gaussian mixtures. We
established rigorous statistical guarantees of EM under two particular
but representative settings of over-specified location Gaussian
mixtures: the balanced and unbalanced mixture-fit. The log-likelihood
for the unbalanced fit remains strongly log-concave\footnote{Moreover,
  in
  Appendix~\ref{sec:unbalanced_vs_balanced_fits_closer_look_at_log_likelihood}
  we differentiate the unbalanced and balanced fit based on the
  log-likelihood and the Fisher matrix and provide a heuristic
  justification for the different rates between the two cases.} (due
to the fixed weights and location parameters being sign flips) and
hence the Euclidean error of the final iterate of EM decays at the
usual rate $(d/n)^{\frac{1}{2}}$ with $n$ samples in $d$ dimensions. However,
in the balanced case, the log-likelihood is no longer strongly
log-concave and the error decays at the slower rate
$(d/n)^{\frac{1}{4}}$. We view our results as the first step in
understanding and possibly improving the EM algorithm in non-regular
settings. We now provide a detailed discussion that sheds light on the
general applicability of our results. In particular, we discuss the behavior
of EM under the following settings:
(i) over-specified mixture models with unknown weight parameters 
(Section~\ref{sub:when_the_weights_are_unknown}),
(ii) over-specified mixture of linear regression (Section~\ref{sub:slow_rates_for_mixture_of_regressions}),
and  (iii) more general settings with over-specified mixture models (Section~
\ref{sub:slow_rates_for_general_mixtures}).
We conclude the paper with a discussion of  several future
directions that arise from the previous settings in Section~\ref{sub:future_directions}.


\subsection{When the weights are unknown} 
\label{sub:when_the_weights_are_unknown}

Our theoretical analysis so far assumed that the weights were fixed,
an assumption common to a number of previous papers in the
area~\cite{Siva_2017,Daskalakis_colt2017,kumar_nips2017}.  In
Appendix~\ref{sec:sunknown_mixtures}, we consider the case of unknown
weights for the model fit~\eqref{EqModelFitNew}.  In this context, our
main contribution is to show that if the weights are initialized far
away from $\frac{1} {2}$---meaning that the initial mixture is highly
unbalanced---then the EM algorithm converges quickly, and the results
from Theorem~\ref{ThmUnbalanced} are valid.  (See
Lemma~\ref{lemma:contraction_pop} in
Appendix~\ref{sec:sunknown_mixtures} for the details.)  
On the other hand, if the initial mixture is not heavily imbalanced, we
observe the slow convergence of EM consistent with Theorems~\ref{thm:pop_over}
and \ref{theorem:convergence_rate_sample_EM}.

\subsection{Slow rates for mixture of regressions} 
\label{sub:slow_rates_for_mixture_of_regressions}

Thus far, we have considered the behavior of the EM algorithm in
application to parameter estimation in mixture models.  
Our findings turn out to hold somewhat
more generally, with Theorems~\ref{thm:pop_over}
and~\ref{theorem:convergence_rate_sample_EM} having analogues when the
EM algorithm is used to fit a mixture of linear regressions in
over-specified settings.  Concretely, suppose that $(Y_{1}, X_{1}),
\ldots, (Y_{n}, X_ {n}) \in \Rspace \times \Rspace^{d}$ are
i.i.d. samples generated from the model
\begin{align}
\label{eq:truemodel_regression}
Y_{i} = X_{i}^{\top} \thetastar + \sigma \noise_{i}, \qquad
\mbox{for $i = 1, \ldots, n$,}
\end{align} 
where $\{\noise_{i}\}_{i=1}^{n}$ are i.i.d. standard Gaussian
variates, and the covariate vectors $X_{i} \in \real^d$ are also
i.i.d. samples from the standard multivariate Gaussian $\NORMAL(0,
I_{d})$.  Of interest is to estimate the parameter $\thetastar$ using
these samples and EM is a popular method for doing so.  When
$\thetastar$ has sufficiently large Euclidean norm, a setting referred
to as the strong signal case, Balakrishnan~\etal~\cite{Siva_2017}
showed that the estimate returned by EM is at a distance
$(d/n)^{\frac{1}{2}}$ from the true parameter $\thetastar$ with high
probability. On the other hand, our analysis shows that when
$\|\thetastar\|_2$ decays to zero---leading to an over-specified
setting---the convergence of EM becomes slow. In particular, the EM
algorithm takes significantly more steps and returns an estimate that
is statistically worse, lying at Euclidean distance of the order
$(d/n)^{\frac{1}{4}}$ from the true parameter. While the EM operators
in this case are slightly different when compared to the over-specified
Gaussian mixture analyzed before, the proof techniques remain similar.
More concretely, we first show that the convergence of population EM is
slow (similar to Theorem~\ref{thm:pop_over}) and then use the
annulus-based localization argument (similar to the proof of
Theorem~\ref{theorem:convergence_rate_sample_EM} from
Section~\ref{sec:new_techniques_for_analyzing_sample_em}) to derive a
sharp rate. For completeness, we present these results formally in
Lemma~\ref{lemma:mix_regress_pop} and
Corollary~\ref{cor:mix_regression_rate} in
Appendix~\ref{sec:mixture_regression_EM}.


\subsection{Slow rates for general mixtures} 
\label{sub:slow_rates_for_general_mixtures}

We now present several experiments that provide numerical backing to
the claim that the slow rate of order $\obs^{-\frac{1} {4}}$ is not
merely an artifact of the special balanced fit (\eqref{EqModelFitNew} with
$\weight=\frac{1}{2}$). We demonstrate that the slow convergence of EM is
very likely to arise while fitting general over-specified location Gaussian
mixtures with unknown weights (and known covariance). 
We consider three settings: (A) fitting several general over-specified location
Gaussian mixture fits to Gaussian data (Figure~\ref{fig:more_cases}), (B)
fitting a special three-component mixture fit to a two mixture of Gaussians (Figure~\ref{fig:two_mixture}), and
(C) fitting mixtures with unknown weights and location parameters when
the number of components in the fitted model is over-specified by two (Figure~
\ref{fig:more_mixtures}). We now turn to the details of these settings.

\paragraph{General over-specified mixture fits on Gaussian data} 
First, we remark that the fast convergence in the
unbalanced fit (Theorem~\ref{ThmUnbalanced}) 
was a joint result of the facts that (a) the weights were
fixed and unequal, and (b) the parameters were constrained to be a sign
flip. If either of these conditions is violated, the EM algorithm
exhibits slow convergence on both algorithmic and statistical
fronts. Theorems~\ref{thm:pop_over}, \ref{theorem:convergence_rate_sample_EM}
 and \ref{theorem:sample_lower_bound} provide rigorous details for the case
 of equal and fixed weights (balanced fit).
When the weights are unknown, EM can exhibit slow rate (see Section~\ref{sub:when_the_weights_are_unknown}
and Appendix~\ref{sec:sunknown_mixtures} for further details).
When the weights are fixed and unequal, but the
location parameters are estimated freely---that is, with the model
$\weight \normDensity(x; \theta_1, 1) + (1-\weight) \normDensity(x;
\theta_2, 1)$, as illustrated in Figure~\ref{fig:more_cases}(a)---then
the EM estimates have error\footnote{For more general cases, we
  measure the error of parameter estimation using the Wasserstein
  metric of second order $\widehat{W}_{2, n}$ to account for
  label-switching between the components. When the true model is
  standard Gaussian this metric is simply the weighted Euclidean
  error: $(\sum\weight_ {k}\widehat\theta_{k, n}^2)^{\frac{1}{2}}$, where
  $\weight_k$ and $\widehat\theta_ {k, n}$, respectively, denote the
  mixture weight and the location parameter of the $k$-th component of
  the mixture.}  of order $\obs^{-\frac{1}{4}}$.  In such cases, the
parameter estimates approximately satisfy the relation $\sum_k
\weight_{k}\widehat\theta_{k, n} \approx 0$, since the mean of the
data is close to zero; moreover, for a two-components mixture model,
the location estimates become weighted sign flips of each other.  The
features are the intuitive reason underlying the similarity of
behavior of EM between this fit and the balanced fit. Finally, when
we fit a two mixture model with unknown weight parameter and free location
parameters, the final error also has a scaling of order $\obs^{-\frac 14}$.
Refer to Figure~\ref{fig:more_cases} for a numerical validation of these
results.

\begin{figure}[h]
  \begin{center}
    \begin{tabular}{cc}
    \widgraph{0.45\textwidth}{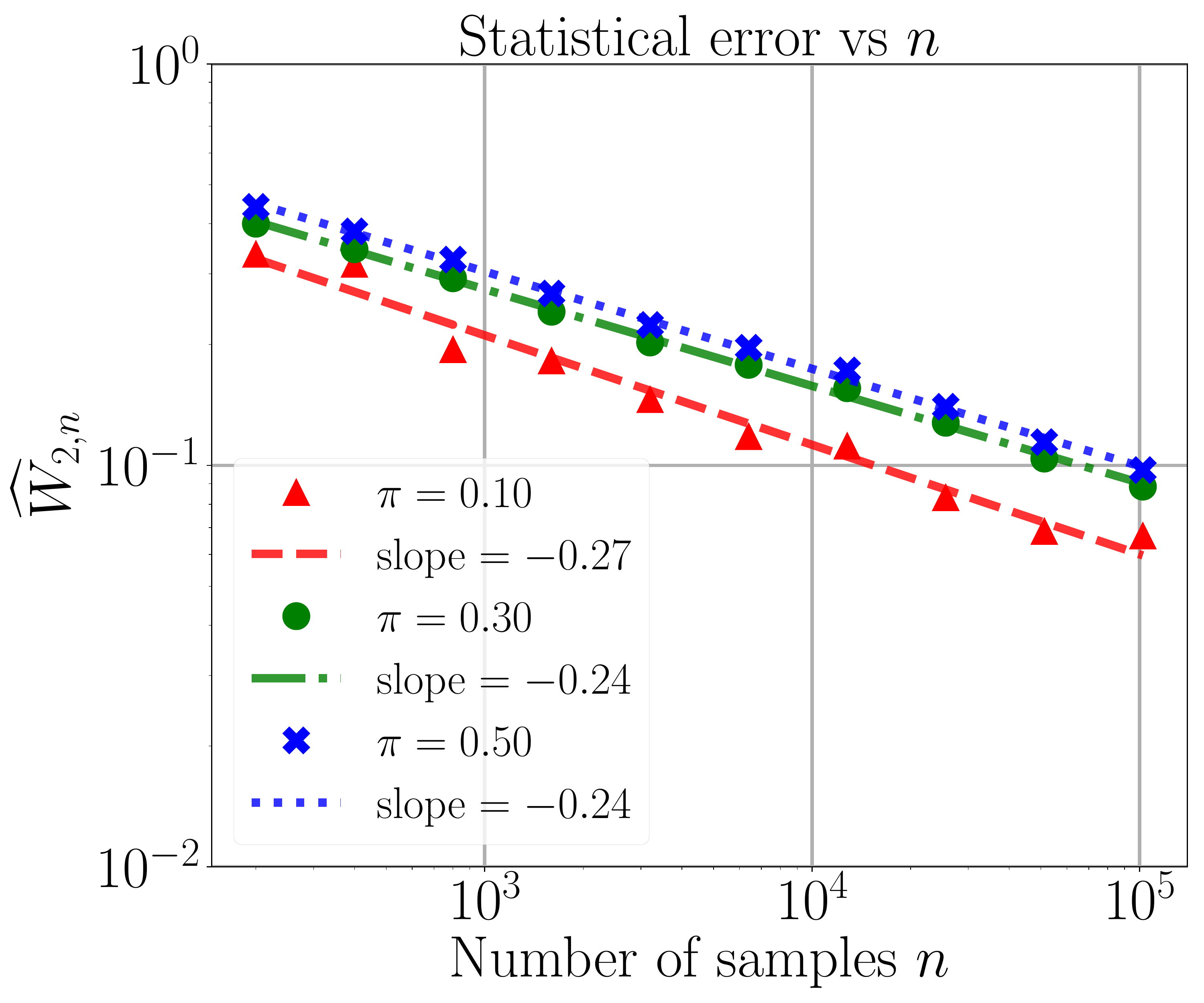} &
    \widgraph{0.45\textwidth}{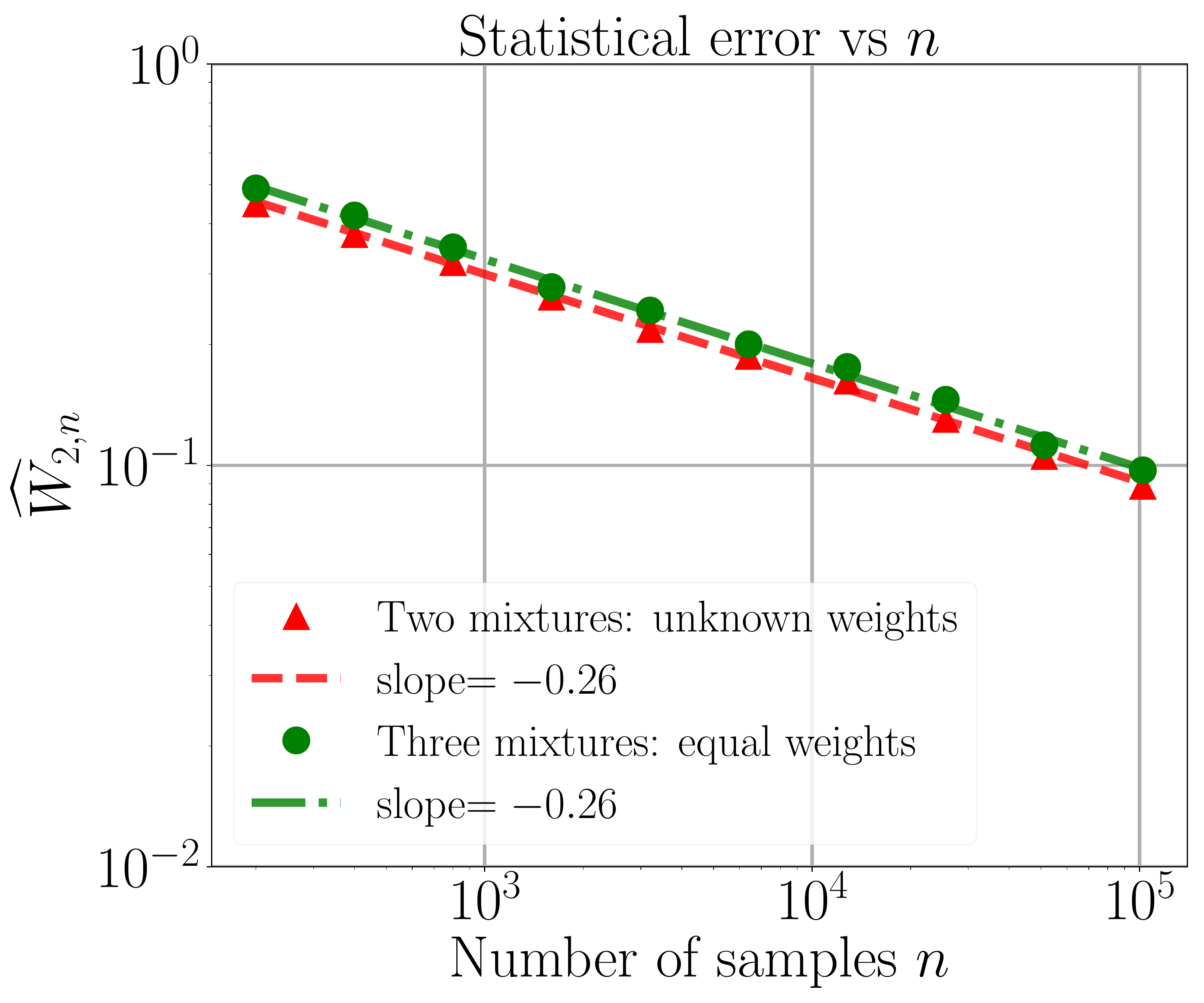} \\
    (a) & (b)
    \end{tabular}
    \caption{Plots of the Wasserstein error~$\widehat W_{2, n}$
      associated with EM fixed points versus the sample size for
      fitting various kinds of location mixture models to standard
      normal $\NORMAL(0, 1)$ data.  We fit mixture models with either
      two or three components, with all location parameters estimated
      in an unconstrained manner.  The lines are obtained by a linear
      regression of the log error on the sample size $n$.  (a) Fitting
      a two-mixture model $\weight\NORMAL(\theta_1, 1) +
      (1-\weight)\NORMAL(\theta_2, 1)$ with three different
      \emph{fixed} values of weights $\weight \in \braces{0.1, 0.3,
        0.5}$ and two (unconstrained) location parameters, along with
      least-squares fits to the log errors. (b) Data plotted as red
      triangles is obtained by fitting a two-component model with
      \emph{unknown} mixture weights and two location parameters
      $\weight\NORMAL(\theta_1, 1) + (1-\weight)\NORMAL(\theta_2, 1)$,
      whereas green circles correspond to results fitting a
      three-component mixture model
      $\sum_{i=1}^3\frac{1}{3}\NORMAL(\theta_i, 1)$.  In all cases,
      the EM solutions exhibit the slow $n^{-\frac{1}{4}}$ statistical
      rate for the error in parameter estimation. Also see
      Figure~\ref{fig:more_mixtures}.}
\label{fig:more_cases}
  \end{center}
\end{figure}

\paragraph{Over-specified fits for mixtures of Gaussian
data} 
\label{par:slow_rates_with_more_than_two_mixtures}
Using similar reasoning as above, let us sketch out how our
theoretical results also yield usable predictions for more general
over-specified models.  Roughly speaking, whenever there are extra
number of components to be estimated, parameters of some of them
are likely to end up satisfying certain form of local constraint.  
More concretely, suppose that we are given data generated from a
$k$-component mixture, and we use the EM algorithm to fit the
location parameters of a mixture model with $k + 1$ components.  
Loosely speaking, the EM estimates corresponding to a set of 
$k-1$ components are likely to converge
quickly, leaving the two remaining components to fit a single component in
the true model.  If the other components are far away, the EM updates for
the parameters of these two components are unaffected by them and start
to behave like the balanced case. See Figure~\ref{fig:two_mixture} for a
numerical illustration of this intuition in an idealized setting where
we use $k+1 = 3$ components to fit data generated from a $k= 2$ component
model.  In this idealized setting, the error for one of the parameter
scales at the fast rate of order $n^{-\frac 12}$, and that of the parameter
that is locally over-fitted exhibits a slow rate of order $n^{-\frac 14}$.
Finally, we see that the statistical error of order $n^{-\frac{1}{4}}$ also
arises when we over-specify the number of components by more than one. In
particular, we observe in Figure~\ref{fig:more_cases}(b) (green dashed dotted
line with solid circles) and Figure~\ref{fig:more_mixtures} (both curves)
that
a similar scaling of order $\obs^{-\frac{1}{4}}$ arises when we
over-specify the number of components by $2$ and estimate the weight
and location parameters. 

Besides formally analyzing EM in these general cases, several other future directions arise from our work which we now discuss.
\begin{figure}[h]
  \begin{center}
    \begin{tabular}{cc}
    \widgraph{0.45\textwidth}{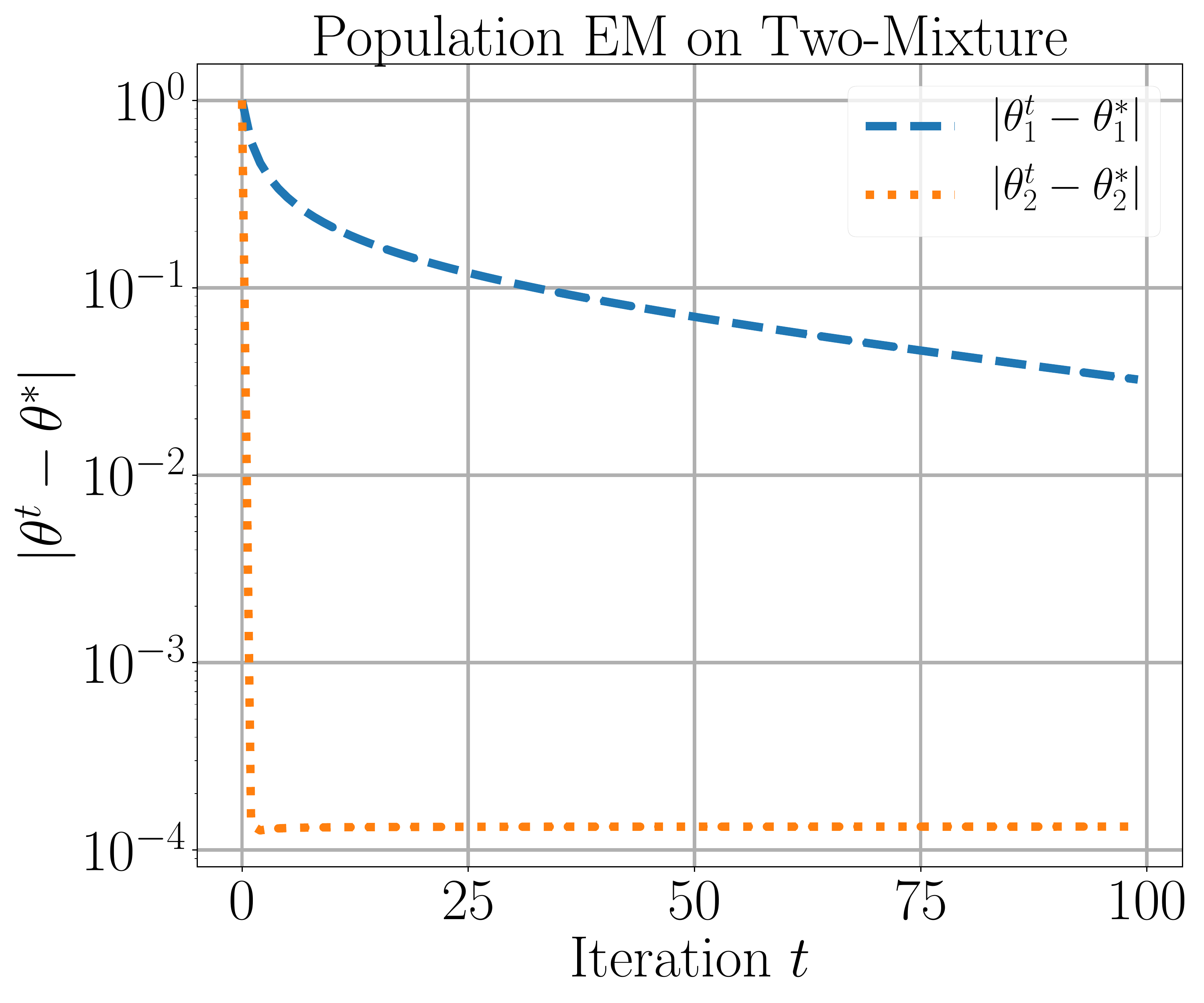} &
    \widgraph{0.45\textwidth}{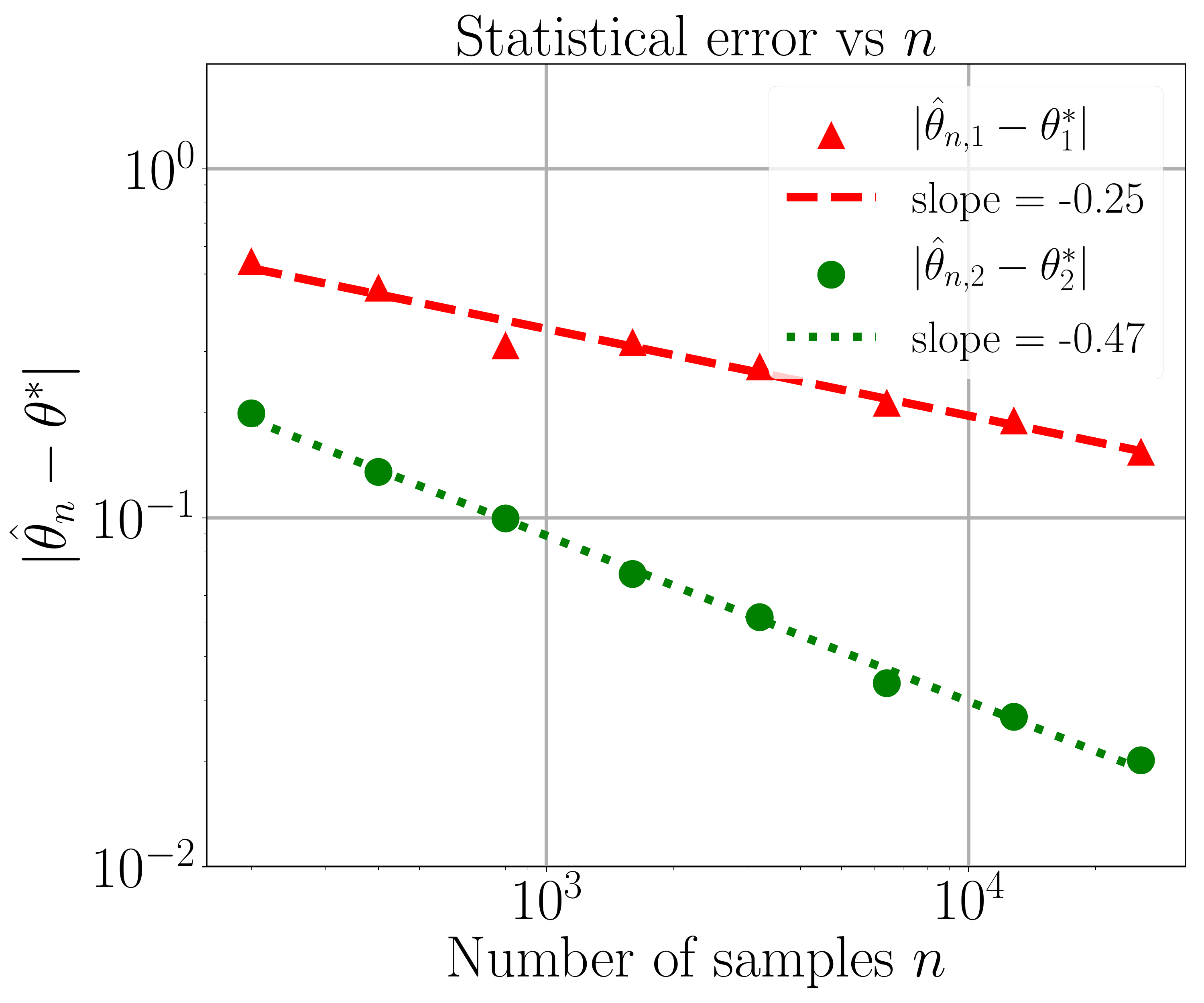} \\
    (a) & (b)
    \end{tabular}
\caption{Behavior of EM for an over-specified Gaussian mixture.  True
  model: $\frac{1}{2}\NORMAL(\thetastar_1, 1) +
  \frac{1}{2}\NORMAL(\thetastar_2,1)$ where $\thetastar_1 = 0 $ and
  $\thetastar_2 = 10$.  We fit a model $\frac{1}{4}
  \mathcal{N}(-\theta_1, 1) + \frac{1}{4} \mathcal{N}(\theta_1, 1)
  +\frac{1}{2}\mathcal{N}(\theta_2, 1)$, where we initialize
  $\theta_1^0$ close to $\thetastar_1$ and $\theta_2^0$ close to
  $\thetastar_2$.  (a)~Population EM updates: We observe that while
  $\theta_1^t$ converges slowly to $\thetastar_1 = 0$, the iterates
  $\theta_2^t$ converge exponentially fast to $\thetastar_2 = 10$.
  (b)~We plot the statistical error for the two parameters.  While the
  strong signal component has a parametric $n^{-\frac{1}{2}}$ rate,
  for the no signal component EM has the slower $n^{-\frac{1}{4}}$
  rate, which is in good agreement with the theoretical results
  derived in the paper.  (We remark that the error floor for
  $\theta_2^t$ in panel (a) arises from the finite precision inherent
  to numerical integration.)}
\label{fig:two_mixture}
  \end{center}
\end{figure}

\begin{figure}
    \begin{center}
        \begin{tabular}{c}
            \widgraph{0.45\textwidth}{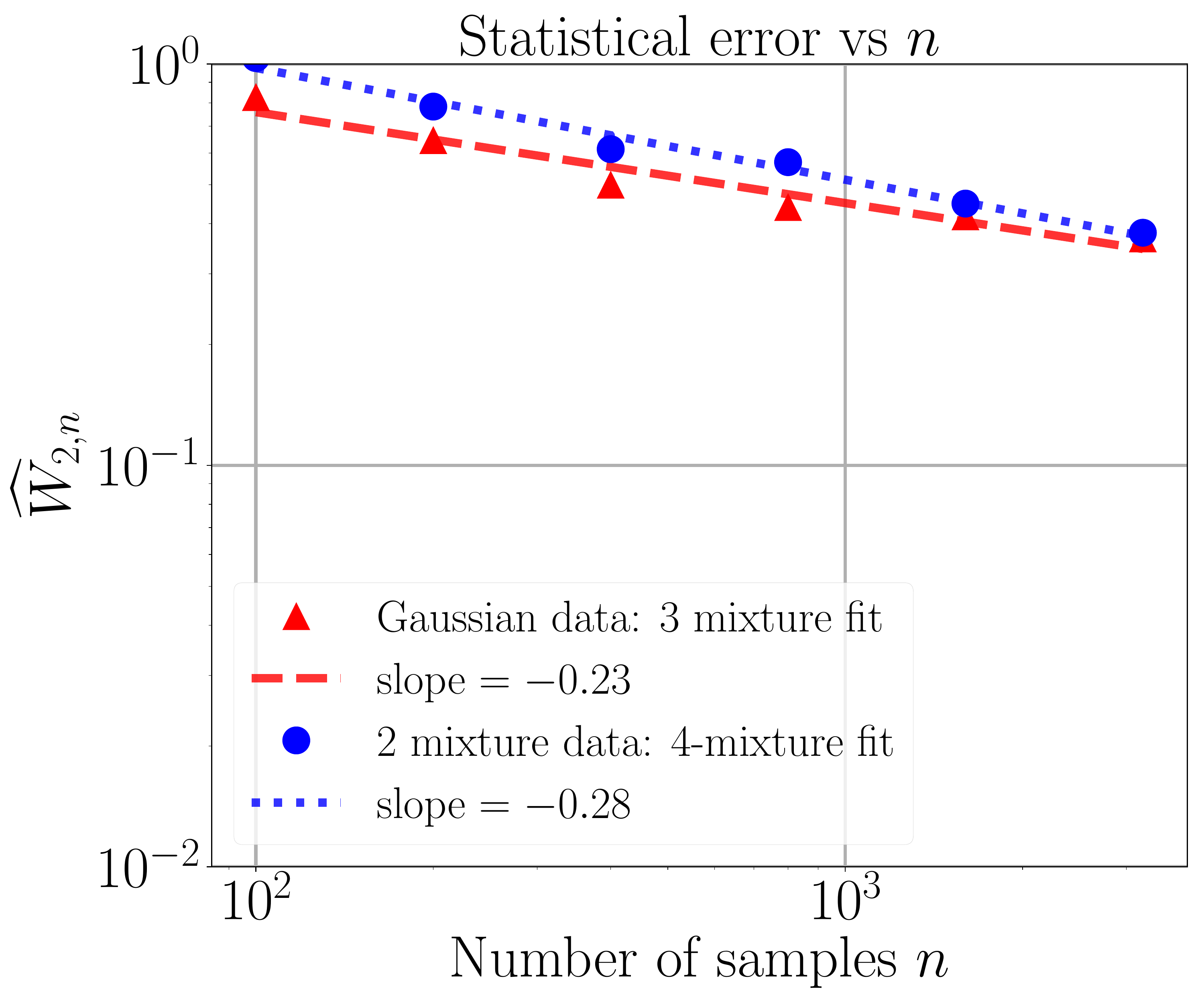}
        \end{tabular}
    \end{center}
    \caption{Plots of Wasserstein error when both weights and location
      parameters are unknown and estimated using EM and the fitted
      multivariate mixture model is over-specified.  (a) True model:
      $\NORMAL([0, 0]\tp, I_2)$, and fitted model $\sum_{i=1}^3 w_i
      \NORMAL(\theta_i, I_2)$ and (b) True model:
      $\frac{2}{5}\NORMAL([0, 0]\tp, I_2) + \frac{3}{5}\NORMAL ([4,
        4]\tp, I_2)$ and fitted model: $\sum_{i=1}^4 w_i
      \NORMAL(\theta_i, I_2)$. In both cases, once again we see the
      scaling of order $\obs^ {-\frac{1}{4}}$ for the final error
      (similar to results in Figure~\ref{fig:more_cases} and
      \ref{fig:two_mixture}).}
\label{fig:more_mixtures}
\end{figure}

\subsection{Future directions} 
\label{sub:future_directions}
In our current work, we assumed that only the location parameters were
unknown and that the scale parameters of the underlying model are
known. Nevertheless in practice, this assumption is rather restrictive
and it is natural to ask what happens if the scale parameters were
also unknown. We note that the MLE is known to have even slower
statistical rates for the estimation error with such higher-order
mixtures; therefore, it would be interesting to determine if the EM
algorithm also suffers from a similar slow down when the scale
parameters are unknown. We refer the readers to a recent
preprint~\cite{dwivedi2019challenges}, where we establish that the EM
algorithm can suffer from a further slow-down on the statistical and
computational ends when over-specified mixtures are fitted with an
unknown scale parameter.

Another important direction is to analyze the behavior of EM under
different models for generating the data. While our analysis is
focused on Gaussian mixtures, the non-standard statistical rate
$n^{-\frac{1}{4}}$ also arises in other types of over-specified
mixture models, such as those involving mixtures with other
exponential family members, or Student-$t$ distributions, suitable for
heavy-tailed data. We believe that the analysis of our paper can be
generalized to a broader class of finite mixture models that includes
the aforementioned models.

A final direction of interest is whether the behavior of EM---slow
versus fast convergence---can be used as a statistic in a classical
testing problem: testing the simple null of a standard multivariate
Gaussian versus the compound alternative of a two-component Gaussian
mixture.  This problem is known to be challenging due to the
break-down of the (generalized) likelihood ratio test, due the
singularity of the Fisher information matrix; see the
papers~\cite{Chen_2008, Chen_2009} for some past work on the problem.
The results of our paper suggest an alternative approach, which is
based on monitoring the convergence rate of EM.  If the EM algorithm
converges slowly for a balanced fit, then we may accept the null,
whereas the opposite behavior can be used as an evidence for rejecting
the null. We leave for future work the analysis of such a testing procedure
based on the convergence rates of EM.


\section*{Acknowledgments} 
\label{sec:acknowledgments}
This work was partially supported by Office of Naval Research grant
DOD ONR-N00014-18-1-2640 and National Science Foundation grant
NSF-DMS-1612948 to MJW, and National Science Foundation grant  
NSF-DMS-1613002 to BY, and by Army Research Office grant
W911NF-17-1-0304 to MIJ.


\begin{center}
\textbf{\Large{Supplementary material}}
\end{center}

We now collect several proofs and results
that were deferred from the main
paper. Appendices~\ref{sec:proofs_of_main_results}
and~\ref{sec:proofs_of_auxiliary_lemmas} contain, respectively, the
proofs of our main theorems, and the proofs of all auxiliary technical
lemmas.  In Appendix~\ref{sec:proofs_of_additional results}, we
provide some additional results that discuss (a) the behavior of EM when
the weights are unknown, and (b) the initial radius requirement for
EM.
Appendix~\ref{sec:unbalanced_vs_balanced_fits_closer_look_at_log_likelihood}
contains details on the nature of log-likelihood and Fisher
information matrix, in order to provide an intuition about the
different behavior of EM in the balanced and unbalanced case.
Finally, in Appendix~\ref{sec:mixture_regression_EM}, we show that the techniques
presented in this paper are not specific to mixture models, and use them
to establish slow convergence of EM for over-specified mixture of linear
regressions.

\appendix

\section{Proofs of main results} 
\label{sec:proofs_of_main_results}

In this appendix, we present the proofs of our main results, namely
Theorems~\ref{ThmUnbalanced},~\ref{thm:pop_over},~\ref{theorem:convergence_rate_sample_EM}, and~\ref{theorem:sample_lower_bound}.


\subsection{Proof of Theorem~\ref{ThmUnbalanced}} 
\label{sub:proof_of_theorem_thm:ThmUnbalanced}

As alluded to earlier in Section~\ref{sub:unbalanced_mixtures}---given
Lemma~\ref{lemma:bound_EM_operators}---it suffices to prove the contraction
property~\eqref{EqnUnbalancedPop} for the population operator
$\updateM$.  Recall that $\thetastar = 0$ is a fixed point of the
population EM operator (i.e., $\updateM(0) = 0$). This fact, combined
with the definition~\eqref{EqnPopMupdate} of the M-update, yields
\begin{align*}
\enorm{\updateM (\theta)} = \enorm{ \updateM (\theta) -
  \updateM (\thetastar)} = \enorm{\Exs \brackets{2 (\weights_\theta(X) -
    \weights_0(X) ) X}},
\end{align*}
where, in the unbalanced setting~\eqref{EqModelFitNew}, the weight function
$\weights_\theta$~\eqref{EqnWeightFun} and the 
gradient $\gradient_{\theta} \weights_\theta$ take the form
\begin{align*}
\weights_\theta (X) & = \frac{\weight}{\weight + (1 - \weight) e^{-
    \frac{2 \theta^\top X}{\sd^2}}}, \ \mbox{and} \
\gradient_\theta \weights_\theta(X)  = \frac{\frac{2 \weight (1 -
    \weight) X}{\sd^2}}{\left( \weight e^{- \frac{\theta^\top
      X}{\sd^2}} + (1 - \weight) e^{ \frac{\theta^\top X}{\sd^2}}
  \right)^2}.
\end{align*}
For a scalar $u \in [0,1]$, define the function $h(u) = \weights_{u
  \theta} (X)$, and note that $h'(u) = \nabla \weights_{u \theta}
(X)^\top \theta$.  Thus, using a Taylor series expansion along the
line \mbox{$\theta_u = u \theta, u \in [0, 1]$}, we find that
\begin{align}
\enorm{\updateM (\theta)} & = \enorm{\Exs \brackets{2 X
    \int_{0}^1 h'(u) du}} \notag \\
& \hspace{- 2 em} = 4 \weight (1 - \weight)
\enorm{\int \limits_{0}^{1} \Exs
  \left[\frac{X X^{\top}}{\sigma^{2} \left(
      \weight \exp \left(- \frac{\theta_{u}^\top
        X}{\sigma^{2}}\right) +
        (1 - \weight) \exp
      \left( \frac{\theta_{u}^\top X}{\sigma^{2}} \right)
       \right)^{2}} \right] \theta du}
\notag\\
\label{eq:thm_asym_gamma_first}
& \hspace{- 2 em}  \leq 4 \weight (1 - \weight) \|\theta \|_{2} \max \limits_{u \in [0,1]}
\opnorm{\Exs \left[\Gamma_{\theta_u}(X) \right]},
\end{align}
where in the last equation we have defined the matrix
\begin{align}
\label{eq:gamma_theta}
\Gamma_{\theta_u}(X) \defn \frac{X X^{\top}}{\sigma^{2} \parenth{
    \weight \exp \left(- \frac{\theta_{u}^\top X}{\sigma^{2}}\right) +
    (1 - \weight) \exp \left(\frac{\theta_{u}^\top
      X}{\sigma^{2}} \right)}^{2}}.
\end{align}
Writing the mixture weight as $\weight = \frac{1}{2} \parenth{1 -
  \contracasym}$, we claim that it suffices to show that
\begin{align}
\label{eq:gamma_op_asym}
\max \limits_{u \in [0,1]} \opnorm{\Exs  \left[\Gamma_{\theta_u}(X)
    \right]} & \leq \frac{1 - \contracasym^2/ 2}{1 - \contracasym^2}.
\end{align}
Indeed, taking the last bound as given and substituting it into
inequality~\eqref{eq:thm_asym_gamma_first}, we find that
\begin{align*}
\enorm{\updateM (\theta)} & \leq 4 \weight \,(1 - \weight) \, \frac{1
  - \contracasym^2 / 2}{1 - \contracasym^2} \; \enorm{\theta} \; = \;
(1 - \contracasym^2/ 2) \enorm{\theta},
\end{align*}
which yields the claim~\eqref{EqnUnbalancedPop} of Theorem~\ref{ThmUnbalanced}.


\paragraph{Proof of claim~\eqref{eq:gamma_op_asym}} We
begin by making a convenient change of coordinates. Let $R \in \real^{
  \dims \times \dims}$ be an orthonormal matrix such that $R
\theta_{u} = \enorm{\theta_u} e_1$, where $e_{1}$ denotes the first
canonical basis vector in dimension $\usedim$.  Define the random
vector $V \defn R X/\sd$.  Since the vector $X \sim \NORMAL(0, \sd^2
I_d)$ and the matrix $R$ is orthonormal, the random vector $V$ follows
a $\NORMAL(0, I_d)$ distribution.  Substituting $X = \sd R^\top V$ and
$R\theta_u = \enorm{\theta_u} e_1$ in the expression~\eqref{eq:gamma_theta} for
$\Gamma_{\theta_u}$ and using the fact that \mbox{$\opnorm{ R^{\top}
    \mymat R} = \opnorm{ \mymat}$} for any matrix $\mymat$ and any
orthogonal matrix $R$, we find that $\opnorm{ \Exs \left[
    \Gamma_{\theta_u} (X) \right] } = \opnorm{ \mymat_{\theta_u}}$,
where
\begin{align*}
\mymat_{\theta_u} \defn \Exs_V \brackets{ \frac{ V V^{\top}
  }{\parenth{ \weight \exp \parenth{- \enorm{\theta_u} V_1 / \sd} + (1
      - \weight) \exp \parenth{\enorm{\theta_u} V_1 / \sd } }^2 }}.
\end{align*}
Here $V_1 \defn V e_1$ denotes the first coordinate of the random
vector $V$. Note that the matrix $\mymat_{\theta_u}$ is a diagonal
matrix, with non-negative entries. Thus, in order to prove the
bound~\eqref{eq:gamma_op_asym}, it suffices to show that
\begin{align}
\label{eq:btheta_thm1}
  \max_{j \in [d]}[ \mymat_{ \theta_u}]_{jj} \leq
  \frac{1 - \contracasym^2 / 2}{1 - \contracasym^2}.
\end{align}
When $\theta_u = 0$, the matrix $\mymat_{\theta_u} = \Exs[V V^\top] =
I_d$ and the claim holds trivially.  Turning to the case $\theta_u
\neq 0$, we split our analysis into two cases, depending on whether $j
= 1$ or $j \neq 1$.


\paragraph{Bounding $[ \mymat_{\theta_u}]_{11}$} 
\label{par:bounding_b_11}

Denoting $\weight = \frac{1}{2}(1 - \contracasym)$, we observe that
\begin{align}
(\weight e^{- y} + (1 - \weight) e^y) & \in [\sqrt{(1 - \contracasym^2)},\,
    1], \quad \text{ if } e^y \in \brackets{ 1, \frac{1 +
      \contracasym}{1 - \contracasym}}, \quad \text{and} \notag \\
(\weight e^{- y} + (1 - \weight) e^y) & > 1, \quad \quad \quad \quad
  \quad \quad \quad \text{otherwise}.
  \label{eqn:unbalanced_bound_case}
\end{align}
Let $\event^c$ and $\mathbb{I}( \event)$ respectively denote the
complement and the indicator of any event $\event$.  Define the event
\begin{align*}
\event_{ \theta_u} \defn \braces{e^{ \enorm{ \theta_u} V_1 / \sd} \in
  \brackets{1, \frac{1 + \contracasym}{1 - \contracasym}}}.
\end{align*} 
Using the observation~\eqref{eqn:unbalanced_bound_case}
above and the fact that $V_1 \sim \Ncal( 0, 1)$, we obtain
\begin{align}
[\mymat_{ \theta_u}]_{11} & =
\Exs \brackets{ \frac{ V_1^2}{ \parenth{ \weight
      \exp \parenth{- \enorm{ \theta_u} V_1 / \sd} + (1 - \weight)
      \exp \parenth{ \enorm{ \theta_u} V_1 / \sd}}^2}} \notag\\ & \leq 
      \frac{1}{( 1 - \contracasym^2)} \Exs \brackets{ V_1^2 
      \,\mathbb{I}( \event_{\theta_u})} + \Exs \brackets{ V_1^2 \, 
      \mathbb{I}( \event_{ \theta_u}^c)} \notag\\ & = \frac{1
      - \contracasym^2 + 
      \contracasym^2 \Exs \brackets{ V_1^2 \,\mathbb{I}
      (\event_{ \theta_u})}}{(1 - \contracasym^2)}.
    \label{eq:v1_event_bound}
\end{align}
Note that whenever $\theta_u \neq 0$, we have that $\event_{ \theta_u}
\subseteq \big \{ V_1 \geq 0 \} $ and consequently, we obtain that
\begin{align}
\label{eq:v_1_half_bound}
\Exs \brackets{ V_1^2 \, \mathbb{I}(\event_{ \theta_u})} \leq \Exs
\brackets{ V_1^2 \,\mathbb{I}( V_1 \geq 0)} =  \frac{1}{2}.
\end{align}
Putting the inequalities~\eqref{eq:v1_event_bound} and
\eqref{eq:v_1_half_bound} together, we conclude that \\
\mbox{$[ \mymat_{\theta_u}]_{11} \leq
({1 - \contracasym^2 / 2})/({ 1 - \contracasym^2})$}.


\paragraph{Bounding $[ \mymat_{\theta_u}]_{jj}, \; j \neq 1$} 
\label{par:bounding_bjj_thm1}
Using arguments similar to the previous case, 
and the fact that the random variables $V_i, i \in [d]$,
are independent standard normal random variables,
we find that
\begin{align*}
    [ \mymat_{\theta_u}]_{jj} &=
        \Exs \brackets{ \frac{V_j^2}{ \parenth{ \weight
              \exp \parenth{- \enorm{ \theta_u} V_1 / \sd} + (1 - 
              \weight)
              \exp \parenth{ \enorm{\theta_u} V_1 / \sd}}^2}}\\ &=\Exs 
              \brackets{ \frac{1}{ \parenth{ \weight
              \exp \parenth{- \enorm{ \theta_u} V_1 / \sd} + (1 - 
              \weight)
              \exp \parenth{ \enorm{ \theta_u} V_1 / \sd}}^2}}.
\end{align*} 
Invoking the definition of the event $\event_{\theta_u}$, we have
\begin{align*}
[ \mymat_{\theta_u}]_{jj} &\leq
        \frac{1}{(1 - \contracasym^2)} \Exs \brackets{ \mathbb{I}
        (\event_{\theta_u})}
        + \Exs \brackets{ \mathbb{I}(\event_{ \theta_u}^c)}
        =
        \frac{1 - \contracasym^2 + \contracasym^2
          \Exs \brackets{ \mathbb{I}( \event_{\theta_u})}
        }{(1 - \contracasym^2)}.
\end{align*}
Finally, noting that
$\Exs \brackets{ \mathbb{I}( \event_{\theta_u})} \leq
\Exs \brackets{ \mathbb{I}( V_1 \geq 0)} = 1/2$ 
whenever $\theta_u \neq 0$, yields the claim.


\subsection{Proof of Theorem~\ref{thm:pop_over}} 
\label{sub:proof_of_theorem_thm:pop_over}

We split our proof into two parts, which correspond to the upper
bound~\eqref{EqnPopEMUpper} and the lower bound~\eqref{EqnPopEMLower}
respectively.


\subsubsection{Proof of the upper bound~\eqref{EqnPopEMUpper}}

For the balanced fit, we have
\begin{align*}
\weights_\theta( X) = \dfrac{1}{1 + e^{- 2 \theta^\top X/ \sd^2}}
\quad \mbox{and} \quad \gradient_\theta( \weights_\theta( X)) =
\dfrac{2 X^\top / \sd^2}{(e^{- \theta^\top X / \sd^2}+e^{ \theta^\top
    X / \sd^2})^2}.
\end{align*}
Using a Taylor expansion and repeating the preliminary computations as
those in the proof of Theorem~\ref{ThmUnbalanced} from the unbalanced
setting, we obtain that
\begin{align}
  \enorm{\updateM (\theta)} &= \enorm{\Exs \brackets{ 2 X \int_{0}^1
      \weights'_{\theta_u}(X)^\top \theta_u du}} \notag \\
& = 4 \bigg \| \int \limits_{0}^{1} \Exs \left[ \frac{X X^{\top}}
    {\sigma^2 \left( e^{- \theta_{u}^\top X / \sigma^2} + e^{
        \theta_{u}^\top X/\sigma^2} \right)^2} \right] \theta du
  \bigg\|_{2} \label{eq:e_int_gamma} \\
& \leq 4 \enorm{ \theta} \int \limits_{0}^{1} \opnorm{\Exs
    \left[\Gamma_{\theta_u}(X) \right]} du, \notag
\end{align}
where $\Gamma_{ \theta_u}(X) \defn \frac{ X X^{\top} / \sd^2} {( e^{-
    \theta_{u}^\top X / \sigma^2} + e^{ \theta_{u}^\top X /
    \sigma^2})^2}$.  Consequently, in order to prove the upper
bound~\eqref{EqnPopEMUpper}, it suffices to show that
\begin{align}
\label{eq:gamma_op}
\int_0^1 \opnorm{\Exs \left[\Gamma_{\theta_u}(X) \right]} du \leq
\frac{1}{4} \parenth{ p + \frac{1 - p}{1 + \enorm{\theta}^2 /2 \sd^2}}
=\frac{\gamup(\theta)}{4} 
\end{align}
where $p \defn (1 + \Prob_{Z \sim \Ncal(0, 1)}(\abss{Z} \leq 1))/2 < 1$.

We now establish the claim~\eqref{eq:gamma_op}. Like in proof of
Theorem~\ref{ThmUnbalanced}, we perform a change of coordinates using
an orthogonal matrix~$R$ such that \mbox{$R \theta_{u} =
  \enorm{\theta_u} e_1$}, where $e_{1}$ is the first canonical basis
in dimension $\usedim$.  Define the random vector $V \defn R X / \sd$.
Since the vector ${X \sim \NORMAL(0, \sd^2 I_d)}$ and the matrix $R$
is orthogonal, we have that the vector ${V \sim \NORMAL(0, I_d)}$.
Substituting the matrix \mbox{$X = \sd R^\top V$} and \mbox{${R
    \theta_u = \enorm{ \theta_u} e_1}$} in the expression for
$\Gamma_{ \theta_u}$ and using the equality \mbox{$\opnorm{ R^{\top}
    \mymat R} = \opnorm{ \mymat}$,} valid for any matrix $\mymat$ and
any orthogonal matrix $R$, we obtain that $\opnorm{ \Exs \left[
    \Gamma_{\theta_u} (X) \right] } = \opnorm{ \mymat_{\theta_u}}$,
where
\begin{align}
\label{eq:balanced_b_theta}
\mymat_{ \theta_u} \defn \Exs_{V} \brackets{ \frac{V V^{\top}}{
    \parenth{ \exp \parenth{- \enorm{ \theta_{u}} V_1 / \sd} + \exp
      \parenth{ \enorm{ \theta_{u}} V_1 / \sd}}^{2}}}.
\end{align}
Clearly, the matrix $\mymat_{ \theta_u}$ is a diagonal matrix with
non-negative entries (note the abuse of notation: the definitions
of the matrices $\Gamma_{\theta_u}$ and $\mymat_{\theta_u}$ is different
from the unbalanced case).
Consequently, to obtain a bound for the
operator norm of the matrix $\mymat_{ \theta_u}$, it is sufficient to
provide an upper bound on the diagonal entries of the matrix $\mymat_{ \theta_u}$. 
In order to do so, we introduce an auxiliary claim:
\begin{lemma}
\label{LemOperatorB}
The $\ell_2$-operator norm of the matrix $\mymat_{ \theta_u}$ defined 
in equation~\eqref{eq:balanced_b_theta}, is upper-bounded as
\begin{align}
  \label{eq:operator_B}
  \opnorm{ \mymat_{ \theta_u}} = \max_{j \in [d]} [\mymat_{
      \theta_u}]_{jj} \leq \frac{ p_2}{4} + \frac{(1 - p_2)}{4}
  \frac{1}{(1 + \enorm{ \theta_u}^2 / (2 \sd^2))^2},
\end{align}
where $p_2 = \Prob\parenth{|V_1| \leq 1 } < 1$.
\end{lemma}
\noindent See Appendix~\ref{AppLemOperatorB} for the proof of
Lemma~\ref{LemOperatorB}. \\

Using Lemma~\ref{LemOperatorB}, we now complete the proof.
Integrating both sides of the inequality~\eqref{eq:operator_B} with
respect to $u \in [0, 1]$, we find that
\begin{align*}
\int_{0}^1 \opnorm{ \mymat_{\theta_u}} du &\leq \int_0^1 \frac{
  p_2}{4} du + \int_0^1 \frac{(1 - p_2)}{4} \frac{1}{(1 +
  \enorm{\theta_u}^2 / (2 \sd^2))^2} du \\
& = \frac{p_2}{4} + \frac{(1 - p_2)}{4} \int_0^1 \frac{1}{(1 + u^2
    \enorm{\theta}^2 / (2 \sd^2))^2} du.
\end{align*}
Direct computation of the second integral yields
\begin{align*}
  \int \limits_{0}^{1} \frac{1}{(1 + u^2 \enorm{ \theta}^2 / (2
    \sd^2))^2} du & = \dfrac{1}{2} \biggr(\frac{1}{1 + \frac{\|
      \theta\|^{2}}{2 \sigma^{2}}} + \frac{ \tan^{- 1} (\| \theta\| /
    ( \sqrt{2} \sigma)}{\|\theta\|/( \sqrt{2} \sigma)} \biggr) \nonumber \\
    & \leq
  \dfrac{1}{2} \biggr( \frac{1}{ 1 + \frac{\| \theta\|^{2}}{2
      \sigma^{2}}} + 1 \biggr),
\end{align*}
where the last inequality above follows since $\tan^{-1}(y) \leq y$, 
for all $y \geq 0$. Putting together the pieces yields
\begin{align*}
  \int_0^1 \opnorm{\Exs \left[ \Gamma_{\theta_u}(X) \right]} du   
  = \int_0^1 \opnorm{ \mymat_{\theta_u}} du
  \leq  \frac{(1 + p_2)}{8} + \frac{(1 - p_2) / 8}{1 + \enorm{\theta}^2 / 
  (2 \sd^2)},
\end{align*}
which implies the claim~\eqref{eq:gamma_op} with $p = \frac{1 + p_2}{2}$.


\subsubsection{Proof of the lower bound~\eqref{EqnPopEMLower}}

We now prove the lower bound~\eqref{EqnPopEMLower} of Theorem~\ref{thm:pop_over}
on the population EM operator $\updateM$.  The argument involves 
Jensen's inequality and
certain properties of the moment generating function (MGF) of the
Gaussian distribution.

Recalling equation~\eqref{eq:e_int_gamma}, we find that
\begin{align}
\enorm{\updateM( \theta)} 
& = 4 \bigg \Vert \underbrace{ \int \limits_{0}^{1} \Exs \bigg[{ \frac{X
        X^{\top}}{ \sd^{2} \left( \exp \left (- { \theta_{u}^\top X} /
        {\sd^{2}} \right) + \exp \left ({\theta_{u}^\top X}/{\sd^{2}}
        \right) \right)^{2}}}\bigg] du}_{\rdefn \updateMat} \,\,
\theta \bigg \Vert_2 \nonumber \\
& \geq 4 \eigmin \parenth{ \updateMat} \enorm{ \theta},
\label{eq:norm_lambda_min_bd}
\end{align}
where $\eigmin( \updateMat)$ denotes the smallest eigenvalue of the
square matrix $\updateMat$. Following the change
of variable $V \defn R X / \sd$ used in the proof of upper
bound~\eqref{EqnPopEMUpper}, we obtain that
\begin{align}
\label{eq:defn_F_theta}
\eigmin \parenth{ \updateMat} = \eigmin \bigg( \underbrace{\Exs_{V} \bigg
[\int
  \limits_{0}^{1} \frac{V V^{\top}}{ \left ( \exp \left( - \|
    \theta_{u} \|_{2} V_1 / \sd \right) + \exp \left( \| \theta_{u}
    \|_{2} V_1 / \sd \right ) \right )^{2}} du \bigg]}_{\rdefn \mymatTheta} \bigg).
\end{align}
Clearly, the matrix $\mymatTheta$ is a diagonal matrix with
non-negative diagonal entries and consequently, we have
\begin{align}
\eigmin( \mymatTheta) = \min_{j \in [\dims]}[ \mymatTheta]_{jj}.
\label{Eq:lambda_min_relation}
\end{align}
In order to provide a lower bound on the diagonal entries
of the matrix $\mymatTheta$, we use the following 
auxiliary claim:
\begin{lemma}
  \label{LemRabbits}
For all vectors $\theta \in \real^\usedim$ such that 
$\enorm{ \theta}^2 \leq \frac{5  \sd^2}{8}$, the matrix 
$\mymatTheta$ defined in equation~\eqref{eq:defn_F_theta}, satisfies the bounds
\begin{align}
  \brackets{ \mymatTheta}_{jj} \geq \brackets{ \mymatTheta}_{11} \geq
  \frac{1}{4 (1 + {2 \enorm{ \theta}^2} / {\sd^2})} \quad \text{for
    all $j \in [\dims]$.}
 \label{eq:theorem_lower_reduction}
\end{align}
\end{lemma}
\noindent See Appendix~\ref{AppLemRabbits} for the proof of this
claim.

Finally, combining the result of Lemma~\ref{LemRabbits} with 
equations~\eqref{eq:norm_lambda_min_bd} and~\eqref{Eq:lambda_min_relation},
 we conclude that
\begin{align*}
\frac{ \enorm{ \updateM( \theta)}}{\enorm{ \theta}} & \geq 4 \eigmin
\parenth{ \updateMat} \; = \; 4 [ \mymatTheta]_{11} \; \geq \;
\frac{1}{({1 + 2 \enorm{ \theta}^2 / \sd^2})} = \gamlow(\theta),
\end{align*}
as claimed.


\subsection{Proof of Theorem~\ref{theorem:convergence_rate_sample_EM}} 
\label{sub:proof_of_theorem_thmsampleoverfit}

The reader should recall the framework that was laid out in
Section~\ref{SecSuboptimal}, especially
Lemma~\ref{lemma:bound_EM_operators} which was used to bound the
deviation between the sample and population EM operators, as well the
annulus-based localization argument (that breaks up the iterations
of EM into different epochs) sketched out in
Section~\ref{SecLocalize}.  The proof of
Theorem~\ref{theorem:convergence_rate_sample_EM} is based on making
this proof outline more precise.


\subsubsection{Epochs and non-expansivity}
\label{ssub:notation}

Let us introduce the notation required to formalize the
analysis that leads to the recursion~\eqref{EqnKeyRecursion}.  Recall
that the recursion~\eqref{EqnKeyRecursion} 
generates the sequence $\{ \seqalpha{\ind}\}_{\ind \geq 0}$ given by
\begin{subequations}
\begin{align}
  \label{eq:alpha_seq}
  \seqalpha{0} = 0 \quad \text{and} \quad \seqalpha{\ind + 1} = \frac{
    \seqalpha{\ind}}{3} + \frac{1}{6}, \quad \mbox{for $\ind = 0, 1,
    2, \ldots$.}
\end{align}
By inspection, this sequence is increasing and satisfies $\lim
\limits_{\ind \to \infty} \seqalpha{\ind} = 1/4$.  Furthermore, we
have $\seqalpha{\ind} \leq 1/4 - \smallthreshold$ for $\ind \geq
\lceil \log(4 / \smallthreshold) / \log 3 \rceil$.  For any given
$\delta \in (0, 1)$, define the following intermediate quantity
\begin{align}
\label{eq:define_big_and_small_delta}
\dndelta = \sd^2 \parenth{\frac{ \dims + \log(( 2\imax + 1) /
\threshold)}{ \obs}}  \text{ where } \imax
\defn \lceil \log(4
  /\smallthreshold) / \log 3 \rceil + 1.
\end{align}
Note that the lower bound on the sample size stated in the theorem
ensures that $\dndelta \leq 1$.  For the proof sketch provided in
Section~\ref{SecLocalize}, we used the rough approximation $\dndelta
\approx d / n$, which is adequate when tracking only the dependency on
the pair $(n,d)$.

For $\ind = 0, 1, 2, \ldots, \imax - 1$, define the scalars
$\numsteps_\ind$ and $\totalnumsteps_{\ind}$ as
\begin{align}
\label{eq:time_steps}
\numsteps_0 = \left \lceil \frac{2}{p} \log\frac{ \enorm{ \theta_0}}{
  \sqrt{2} \sd \sqrt{ \dndelta}} \right\rceil, \ \  \numsteps_{\ind}
= \left \lceil \frac{2}{ p \dndelta^{ 2 \seqalpha{\ind + 1}}} \log(1 /
\dndelta) \right \rceil, \ \text{ and} \  \totalnumsteps_{\ind}
= \sum_{j = 0}^\ind \numsteps_j,
\end{align}
\end{subequations}
where $\lceil y \rceil $ denotes the smallest integer greater than or equal
to $y$, and the
constant $p \in (0,1)$ is given by $p = \Prob(\abss{X} \leq 1) +
\frac{1}{2} \Prob(\abss{X}>1)$ where \mbox{$X \sim \Ncal(0,1)$}.  For each
$\ind = 1, 2, \ldots$, the term $\numsteps_{\ind}$ corresponds to the
number of iterations for the $\ind$-th epoch, whereas the quantity
$\totalnumsteps_{\ind}$ denotes the total number of iterations up to
the completion of that epoch.

Recall that Lemma~\ref{lemma:bound_EM_operators}, stated in the main
paper, gives us a bound on 
$\sup_{\| \theta \|_2 \leq r}\enorm{M_{n}( \theta) - M( \theta)}$
for a given radius $r$.  In the epoch-based argument, we have a
sequence of such radii (corresponding to the outer radii of the annulus
considered in each epoch), so that we need to control this same quantity
uniformly over all radii $r$ in the set $\radii$ given by
\begin{align}
\label{eq:R_and_c}
\radii & = \braces{ \enorm{ \theta^0}, \sqrt{2} \sd
  \dndelta^{ \seqalpha{0}}, \ldots, \sqrt{2}
  \sd \dndelta^{ \seqalpha{ \imax - 1}}, \uniconnew \sqrt{2} \sd
  \dndelta^{ \seqalpha{0}}, \ldots, \uniconnew \sqrt{2} \sd
  \dndelta^{ \seqalpha{ \imax - 1}}}.
\end{align}
Here ${\uniconnew = (2 \uniconone \sd / p + 1)}$ 
denotes a constant independent of $\obs, \dims, \delta$
and $\smallthreshold$ where $\uniconone$ is the universal
constant that appeared in the bound from Lemma~\ref{lemma:bound_EM_operators}.
In order to do so, we apply a standard union bound with
Lemma~\ref{lemma:bound_EM_operators} and obtain that
\begin{align}
  \label{eq:sup_m_bound}
  \sup_{\| \theta \|_2 \leq r}\enorm{M_{n}( \theta) - M( \theta)}
  \leq \uniconone \sigma r \sqrt{ \dndelta} \quad \text{ for all } r \in \radii,
\end{align}
with probability at least $1 - \delta$.
Let $\event(\obs,\dims,
\smallthreshold, \threshold)$ denote the event that the
bound~\eqref{eq:sup_m_bound} holds.

With this notation in place, we start with our first claim. The
sample-based EM operator is \emph{non-expansive} in the following
sense:
\begin{lemma}
\label{LemNonExpansive}
Consider the sample-based EM iteration $
\theta^{t + 1} = M_n( \theta^t)$ with a sample size
 $\obs \geq (2 \uniconone \sd / p)^{1 / (2
  \smallthreshold)} \sd^2 (\dims + \log((2 \imax + 1) / \delta))$. Suppose
that there exists an index $\ind \in \braces{0, 1, \ldots, \imax - 1}$
and an iteration number $t$ such that \mbox{$\enorm{ \theta^t}
  \leq \sqrt{2} \sd \dndelta^{ \seqalpha{\ind}}$}.  Then, conditional
on the event $\event(\obs, \dims, \smallthreshold, \delta)$ from
equation~\eqref{eq:sup_m_bound}, we have
\begin{align}
\Vert{ \theta^{t'}}\Vert_2 \leq \sqrt{2} \sd \dndelta^{\seqalpha{\ind}}
\qquad \mbox{for all $t' \geq t$.}
\end{align}
\end{lemma}
\noindent See Appendix~\ref{AppLemNonExpansive} for the 
proof of this claim.


\subsubsection{Core of the argument}
\label{ssub:proof_of_theorem_convergence_rate_sample_EM}

We now proceed to the core of the argument.  Suppose that the sample
size is lower bounded as
\begin{align}
  \label{eq:n_bound}
\obs \geq \max \braces{c_2, \parenth{ \frac{ 2c_1 \sd}{ p}+1}^{\frac {4
}{\smallthreshold}}, \parenth{\frac{\sqrt 2 c_1 \enorm{\theta^0}}{p}+1}^2
     } \cdot \sd^2  (d +\log \parenth{
        \frac{3 \log( 4 / \smallthreshold)}{ \delta}}),
\end{align}
where the constants $c_1$ and $c_2$ correspond to that from 
Lemma~\ref{lemma:bound_EM_operators}.
Moreover, recall that the quantity $\dndelta$ and the time-steps
$\totalnumsteps_{\ind}$ were defined in
equations~\eqref{eq:define_big_and_small_delta} and
\eqref{eq:time_steps} respectively.  The core of the proof consists of
the following:
\paragraph{Key claim}
For all $\ind \in \braces{ 0, 1, \ldots, \imax - 1}$, we have
\begin{align}
  \enorm{ \theta^t } & \leq \sqrt{2} \sd \dndelta^{ \seqalpha{\ind}}
  \qquad \text{ for all } t \geq \totalnumsteps_{\ind},
  \label{eq:general_step}
  \end{align}
with probability at least $1 - \delta$.

Taking this claim as given, let us now show how the bounds in
Theorem~\ref{theorem:convergence_rate_sample_EM} hold for all $t \geq
\totalnumsteps_{ \imax - 1}$.  Straightforward computations yield that
\begin{align}
\totalnumsteps_{\imax - 1} & \leq T_0 + ( \imax - 1) t_{ \imax - 1} \nonumber \\
&
\leq \frac{ 4}{ p} \brackets{ \log \frac{ \enorm{ \theta_0}}{ \sqrt{2}
    \sd \sqrt{ \dndelta}} + \log \frac{ 4}{ \smallthreshold} \cdot
  \dndelta^{ 1 / 2 - 2 \smallthreshold} \cdot \log \frac{ \obs}{ \sd^2
    \dims}} \notag \\
\label{eq:t_bound}
& \leq \frac{8}{p} \brackets{ \log \frac{ \enorm{ \theta_0}^2 \obs}{
    \sd^2 \dims} + \parenth {\frac{ \obs}{ \dims}}^{ \frac{1}{2} - 2
    \smallthreshold} \cdot \log \parenth{ \frac{4} { \smallthreshold}}
  \cdot \log \parenth{ \frac{ \obs}{ \sd^2 \dims}} \cdot \sd^{4
    \smallthreshold - 1}}.
\end{align}
In other words, equations~\eqref{eq:n_bound} and~\eqref{eq:t_bound}
provide the explicit expression for the number of samples and number
of steps required by sample-based EM to converge to a ball of radius
$(\dims / \obs)^{ 1 / 4 - \smallthreshold}$ around the truth
$\thetastar = 0$.

\paragraph{Proof of the claim~\eqref{eq:general_step}}
It remains to prove the key claim, and we do so by an induction on the
epoch index $\ind$.  All of the argument are performed conditioned on
the event $\event(\obs, \dims, \smallthreshold, \delta)$ defined in
equation~\eqref{eq:sup_m_bound}; note that this event occurs with
probability at least $1 - \delta$.  Moreover, we see that the sample
size assumption~\eqref{eq:n_bound} for
Theorem~\ref{theorem:convergence_rate_sample_EM} is larger than
required in Lemma~\ref{LemNonExpansive} and hence we can invoke the
non-expansiveness of the sample-based EM operator in our arguments
to follow.


\paragraph{Proof of base case: ($\ind = 0$)} 
\label{par:proof_of_claim_eq:first_step}

We adopt the shorthand \mbox{$\kinit = \enorm{ \theta_0} / { \sqrt 2
    \sd}$.}  The non-expansiveness property of the sample-based
EM-operator (Lemma~\ref{LemNonExpansive}) ensures that it is
sufficient to consider the case that $\enorm{ \theta^t} \in [ \sqrt{2}
  \sd, \kinit \sqrt{2} \sd]$ for all $ t \leq \totalnumsteps_0$.
Applying the triangle inequality yields
\begin{subequations}
\begin{align}
\label{eq:theta_t_first_step}   
\enorm{ \theta^{t + 1}} & \leq \enorm{ \updateM_\obs( \theta^t ) -
  \updateM( \theta^t )} + \enorm{ \updateM( \theta^t)} \\
\label{eq:theta_t_first_recursion}
& \stackrel{(i)}{ \leq} \uniconone \sd \cdot \kinit \sqrt{2}\sd \cdot
\sqrt{ \dndelta} + \gamup( \theta^t ) \enorm{ \theta^t },
\end{align}
\end{subequations}
where step~(i) follows from  using $r = \kinit \sqrt{2} \sd$ 
in the event~\eqref{eq:sup_m_bound}, and applying Theorem~\ref{thm:pop_over}
(for the two terms respectively).  
Noting that
$\enorm{ \theta^t} \geq \sqrt{2} \sd$, we also have that
\begin{align*}
  \gamup(\theta^t) = 1 - p + \frac{p}{ 1 + \enorm{ \theta^t}^2 /
    \sd^2} = 1 - \frac{ p \enorm{ \theta^t}^2}{ \enorm{ \theta^t}^2 +
    2 \sd^2} \leq \underbrace{ 1 - \frac{p}{2} }_{ \overline{\gamma}_0}.
\end{align*}
Recursing the inequalities~\eqref{eq:theta_t_first_step} and
\eqref{eq:theta_t_first_recursion} from $t = 0$ up to $t =
\totalnumsteps_0$, and using the fact that $ \gamup( \theta^t ) \leq
\overline{\gamma}_0$, we find that
\begin{align*}
  \enorm{ \theta^{ \totalnumsteps_0}} & \leq \uniconone \sd \cdot \kinit
  \sqrt{2}\sd \cdot \sqrt{ \dndelta} (1 + \overline{\gamma}_0 + \ldots +
  \overline{\gamma}_0^{ \totalnumsteps_0 - 1}) + \overline{\gamma}_0^{ \totalnumsteps_0}
  \enorm{ \theta^0} \\
  & \leq \frac{ \uniconone \sd \cdot \kinit \sqrt{2} \sd \cdot
    \sqrt{ \dndelta}}{1 - \overline{\gamma}_0} + \overline{\gamma}_0^{ \totalnumsteps_0}
  \kinit \sqrt{2} \sd.
\end{align*}
Substituting the expressions $\overline{\gamma}_0 = 1 - p / 2$ and 
$\totalnumsteps_0 = \lceil (2 / p) \log( \kinit / \sqrt{ \dndelta}) \rceil$,
we obtain that
\begin{align*}
  \enorm{ \theta^{ \totalnumsteps_0}} \leq \parenth{2 \kinit \uniconone
    \sd / p + 1} \sqrt{ \dndelta} \sqrt{2} \sd \leq \sqrt{2} \sd,
\end{align*}
where the last inequality follows from the fact that for the assumed
bound~\eqref{eq:n_bound} on $n$, we have $(2 \kinit \uniconone\sd / p + 1)
\sqrt{ \dndelta} \leq 1$.  The base-case now follows.


\paragraph{Proof of inductive step} 
\label{par:proof_of_claim_eq:general_step_}
\newcommand{\gamupbar}{\overline\gamma_{\ind}}
Now we prove the inductive step. In particular, we assume that
$\enorm{ \theta^{ \totalnumsteps_{\ind}}} \leq \sqrt{2} \sd \dndelta^{
  \seqalpha{\ind}}$ and show that $\enorm{ \theta^{
    \totalnumsteps_{\ind + 1}}} \leq \sqrt{2}\sd \dndelta^{ \seqalpha{
    \ind + 1}}$.  Once again, Lemma~\ref{LemNonExpansive} implies that
we may assume without loss of generality that \mbox{$\enorm{ \theta^t}
  \in [ \dndelta^{ \seqalpha{\ind + 1}}, \dndelta^{
      \seqalpha{\ind}}]$} for all $t \in \braces{
  \totalnumsteps_{\ind}, \ldots, \totalnumsteps_{\ind + 1}}$.  Under
this condition, we have that
\begin{align}
  \label{eq:gamma_max}
  \gamup( \theta^t) \leq 1 - \frac{ p \dndelta^{2 \seqalpha{\ind + 1}}
  }{1 + \dndelta^{2 \seqalpha{\ind + 1}}} \leq \underbrace{1 - \frac{p
      \dndelta^{2 \seqalpha{\ind + 1}} }{2}}_{\rdefn \gamupbar}
  \quad \text{ for all } t \in \braces{\totalnumsteps_{\ind}, \ldots,
    \totalnumsteps_{\ind + 1} - 1},
\end{align}
where the last step follows from the fact that $\dndelta \in [0, 1]$
and $\seqalpha{\ind} \geq 0$.  From our earlier
definition~\eqref{eq:time_steps}, we have $\totalnumsteps_{\ind + 1} =
\totalnumsteps_{\ind} + \numsteps_{\ind}$.  We split the remainder of
our proof in two parts, primarily to handle the constants. 
First, we show that
\begin{subequations}
\begin{align}
  \label{eq:theta_t_by_2}
  \enorm{\theta^{ \totalnumsteps_{\ind} + \lceil \numsteps_{\ind +
        1}/2 \rceil}} \leq \uniconnew \sqrt{2} \sd \dndelta^{
    \seqalpha{\ind + 1}} ,
\end{align}
where $\uniconnew = (2 \uniconone \sd / p + 1)$ is a constant
independent of $\obs, \dims, \delta$ and $\smallthreshold$.  Next we
use this result to show that
\begin{align}
  \label{eq:theta_T}
  \enorm{ \theta^{ \totalnumsteps_{\ind + 1}}} = \enorm{ \theta^{
      \totalnumsteps_{\ind} + \numsteps_{\ind + 1}}} \leq \sqrt{2} \sd
  \dndelta^{ \seqalpha{\ind + 1}},
\end{align}
\end{subequations}
which completes the proof of the induction step.  We now prove these two
claims one by one.


\paragraph{Proof of claim~\eqref{eq:theta_t_by_2}} 
\label{par:proof_of_claim_eq:theta_t_by_2}

Applying the triangle inequality yields
\begin{align}
  \label{eq:theta_t_i_first_recursion}  
  \enorm{ \theta^{ \totalnumsteps_{\ind} + 1}} & \leq \enorm{
    \updateM_\obs( \theta^{ \totalnumsteps_{\ind}}) - \updateM(
    \theta^{ \totalnumsteps_{\ind}})} + \enorm{ \updateM( \theta^{
      \totalnumsteps_{\ind}})} \nonumber \\
      & \stackrel{(i)}{\leq} \uniconone \sd
  \cdot \sqrt{2} \sd \dndelta^{ \seqalpha{\ind}}\cdot \sqrt{ \dndelta}
  + \gamup( \theta^{ \totalnumsteps_{\ind}}) \enorm{ \theta^{
      \totalnumsteps_{\ind}}},
\end{align}
where step~(i) follows from using $r =
\sqrt{2} \sd \dndelta^{ \seqalpha{\ind}}$ in the
event~\eqref{eq:sup_m_bound} and applying Theorem~\ref{thm:pop_over}. 
Recursing the inequality~\eqref{eq:theta_t_i_first_recursion} for $T \leq
\lceil
\numsteps_{\ind}/2 \rceil$ steps, and invoking the
bound~\eqref{eq:gamma_max}, i.e., $\gamup( \theta^t) \leq
\gamupbar$ for all \mbox{$t \in \braces{ \totalnumsteps_{\ind},
    \ldots, \totalnumsteps_{\ind} + T}$}, we obtain that
\begin{align*}
\enorm{ \theta^{ \totalnumsteps_{\ind} + T}} & \leq \uniconone \sd \cdot
\sqrt{2}\sd \dndelta^{ \seqalpha{\ind}} \cdot \sqrt{ \dndelta} \cdot (1 +
\gamupbar + \ldots + \gamupbar^{T - 1}) + \gamupbar^T
\enorm{ \theta^{ \totalnumsteps_{\ind}}} \\
& \leq \frac{ \uniconone \sd \cdot \sqrt{2} \sd
  \dndelta^{ \seqalpha{\ind}} \cdot \sqrt{ \dndelta} }{1 - \gamupbar} +
\sqrt{2} \sd \gamupbar^T \dndelta^{ \seqalpha{\ind}} \notag \\
& \stackrel{(i)}{\leq} \uniconone \sd \cdot \sqrt{2} \sd \cdot (2 / p)
\cdot \dndelta^{ \seqalpha{\ind} + 1 / 2 - 2 \seqalpha{\ind + 1}} +
e^{- T p \dndelta^{ 2 \seqalpha{\ind + 1}} / 2} \cdot \sqrt{2} \sd
\dndelta^{ \seqalpha{\ind}} \\
& \stackrel{(ii)}{\leq} \sqrt{2} \sd \dndelta^{ \seqalpha{\ind} + 1 /
  2 - 2 \seqalpha{ \ind + 1}} \cdot (2 \uniconone \sd / p + 1) \\
& \stackrel{(iii)}{=} \uniconnew \sqrt{2} \sd \dndelta^{
  \seqalpha{\ind + 1}},
\end{align*}
where step~(i) follows from the inequality~\eqref{eq:gamma_max} and
the consequent bound $\gamupbar \leq e^{-
  p / (2 \dndelta^{ 2 \seqalpha{ \ind + 1}})}$. Furthermore, in step~(ii), we
used the following bound
\begin{align}
\label{eq:gamma_t_bound}  
\gamupbar^T \leq e^{- T p \dndelta^{2 \seqalpha{\ind + 1}} / 2} \leq
\dndelta^{ 1 / 2 - 2 \seqalpha{\ind + 1}} \quad \text{for } T \geq \frac{(1 -
  4 \seqalpha{\ind + 1})}{p \dndelta^{ 2 \seqalpha{\ind +
      1}}}\log \frac{1}{ \dndelta},
\end{align}
and in step~(iii) we invoked the relation~\eqref{eq:alpha_seq}, i.e.,
$3 \seqalpha{\ind + 1} = 1/2 + \seqalpha{\ind}$.  The claim now
follows from noting that $T = \lceil \numsteps_{\ind} / 2 \rceil$
satisfies the condition of equation~\eqref{eq:gamma_t_bound}.


\paragraph{Proof of claim~\eqref{eq:theta_T}} 
\label{par:proof_of_claim_eqre}

The proof of this claim makes use of arguments similar to those used
above in the proof of claim~\eqref{eq:theta_t_by_2}.  Starting at time
$\totalnumsteps_{\ind} + \lceil \numsteps_{\ind + 1}/2 \rceil$, and applying
the triangle inequality, we find that
\begin{align*}
\enorm{ \theta^{ \totalnumsteps_{\ind} + \lceil \numsteps_{ \ind + 1} / 2 \rceil
+
    1}} \leq \uniconone \sd \cdot \uniconnew \sqrt{2} \sd
\dndelta^{ \seqalpha{\ind + 1}} \cdot \sqrt{ \dndelta} + \gamupbar
\enorm{ \theta^{ \totalnumsteps_{\ind} + \lceil \numsteps_{\ind + 1} / 2 \rceil}},
\end{align*}
where we have used the bound~\eqref{eq:sup_m_bound} with $r =
\uniconnew \sqrt{2} \sd \dndelta^{ \seqalpha{\ind + 1}}$.  Repeating this
inequality for $T \geq \frac{( 1 - 4 \seqalpha{ \ind +
    1})}{p \dndelta^{2 \seqalpha{ \ind + 1}}} \log \frac{1}{\dndelta}$ steps
and performing computations similar to the proof above, we find that
\begin{align*}
\enorm{ \theta^{ \totalnumsteps_{\ind} + \lceil \numsteps_{\ind + 1} / 2 \rceil +
    T}} & \leq \sqrt{2} \sd \dndelta^{ \seqalpha{\ind + 1} + 1 / 2 -
  2 \seqalpha{\ind + 1}} \cdot \uniconnew \cdot (2 \uniconone \sd / p + 1)  \nonumber \\
  & =
      {\uniconnew}^2 \dndelta^{ 1 / 2 - 2 \seqalpha{\ind + 1}} \cdot \sqrt{2} \sd
      \dndelta^{ \seqalpha{\ind + 1}}.
\end{align*}
Observe that $2 \seqalpha{\ind + 1} - 1 / 2 \leq - 2 \smallthreshold$
for all $\ind \leq \imax - 1$ and that the sample size given by
bound~\eqref{eq:n_bound} satisfies $\obs \geq (\uniconnew)^{4 / \smallthreshold}
\sd^2(\dims + \log(2 \imax /
\delta)) $; together, these facts imply that ${ \uniconnew}^2 \dndelta^{
1 / 2 - 2 \seqalpha{ \ind + 1}}
\leq 1$.  The claim now follows.

 \newcommand{\Aevent}{\ensuremath{\mathcal{A}}}
\newcommand{\Bevent}{\ensuremath{\mathcal{B}}}

\subsection{Proof of Theorem~\ref{theorem:sample_lower_bound}} 
\label{sub:proof_of_theorem_theorem:lower_bound_fixed_point}

We now turn to the proof of the lower bound on the accuracy of EM
fixed points, as stated in Theorem~\ref{theorem:sample_lower_bound}.
Recalling the definition~\eqref{eq:sample_em_operator} of a
sample-based EM operator $\mupdate_\obs$, the fixed point relation
$\mupdate_\obs(\fixedpointtheta) = \fixedpointtheta$ can be re-written
as
\begin{align}
\label{eq:fixed_point}  
\fixedpointtheta = \frac{1}{ \obs} \sum_{i = 1}^\obs X_i \tanh
\parenth{ \frac{ \fixedpointtheta X_i}{ \sd^2}},
\end{align}
where  $\fixedpointtheta$ denotes a fixed point solution.
Our proof makes use of the following elementary bounds on the
hyperbolic tangent function:
\begin{subequations}
\begin{align}
\label{eq:tanh_positive}
 x \cdot \tanh( \scalar x) & \geq \scalar x^2 - \frac{1}{3} \scalar^3
 x^4, \quad \text{for } \scalar \geq 0, x \in \real  \quad \text{ and}\\
\label{eq:tanh_negative}     
x \cdot \tanh( \scalar x) & \leq \scalar x^2 - \frac{1}{3} \scalar^3
x^4, \quad \text{for }\scalar < 0, x \in \real.
\end{align} 
\end{subequations}
In order to keep the proof self-contained, we prove these bounds at
the end of this section.  Now plugging in $\scalar = \fixedpointtheta /
\sd^2$ and using the bound~\eqref{eq:tanh_positive} for the case
$\fixedpointtheta \geq 0$ and the bound~\eqref{eq:tanh_negative} for
the case $\fixedpointtheta < 0$, we find that
\begin{align*}
\vert \fixedpointtheta \vert \geq \frac{ \vert \fixedpointtheta
  \vert}{ \sd^2} \cdot \frac{1}{ \obs} \sum_{i = 1}^\obs X_i^2 -
\frac{ \vert \fixedpointtheta \vert^3}{3 \sd^6} \cdot
\frac{1}{\obs}\sum_{i = 1}^\obs X_i^4.
\end{align*}
Denoting $Y_{i} = X_{i} / \sd $ for $i \in
\brackets{n}$ and re-arranging the inequality above yields that
\begin{align}
    |\fixedpointtheta|^3 \geq \frac{3 \sigma^2
   \parenth{ \frac{1}{ \obs}\sum_{i = 1}^\obs Y_i^2 -
     1} |\fixedpointtheta|}{\frac{1}{ \obs}\sum_{i = 1}^\obs Y_i^4}.
     \label{eqn:Theta_cube_lower_bd}
\end{align}
Note that the random variables $Y_i \stackrel{\textrm{i.i.d.}}{\sim} 
\Ncal (0, 1)$ and thereby the
quantity on the RHS above is a ratio of empirical moments of Gaussian
random variables. In order to obtain a lower bound for
$|\fixedpointtheta|$ from the
inequality~\eqref{eqn:Theta_cube_lower_bd}, we exploit a few standard
probability bounds for the concentration of moments of standard
Gaussian distribution (refer to Theorem~5.2 in
Inglot~\cite{inglot2010inequalities} and \mbox{Theorem~6.7 in
  Janson~\cite{janson1997gaussian}}). In particular, we have
  \begin{subequations}
    \begin{align}
  \Prob \brackets{ \frac{ \sum_{i = 1}^\obs Y_i^2}{ \obs} - 1 \geq
    \frac{ \log 17}{ \obs} + \frac{ \sqrt{ \log 17} / 4}{ \sqrt{
        \obs}}} & \geq \frac{1}{17}, \qquad \text{ and}
        \label{eq:first_bnd}\\
    \Prob \brackets{ \frac{ \sum_{i = 1}^\obs Y_i^4}{ \obs} \leq c} & \geq 1 -
    \frac{1}{34}, \label{eq:second_bnd}
    \end{align}      
  \end{subequations}
where $c = (e \log(34) / 2)^2 \sqrt{6}$. 
Plugging these bounds in the inequality~\eqref{eqn:Theta_cube_lower_bd},
we find that
\begin{align}
    \frac{ \frac{1}{ \obs}\sum_{i = 1}^\obs Y_i^2 -
     1}{\frac{1}{ \obs}\sum_{i = 1}^\obs Y_i^4} 
     \geq \frac{ \sqrt{ \log 17} / 4}{c} \cdot
\frac{1}{ \sqrt{\obs}},
\label{eqn:const-prob_lb}
\end{align}
with probability at least $1/34$, where we have used the following elementary fact for two events $\Aevent_1, \Aevent_2$:
\begin{align*}
 \Prob(\Aevent_1 \cap \Aevent_2) = \Prob(\Aevent_1) + \Prob(\Aevent_2)-
\Prob(\Aevent_1\cup \Aevent_2) \geq \Prob(\Aevent_1) + \Prob(\Aevent_2)-
1.
\end{align*}

Let $\Aevent$ denote the event that ``there are at least two non-zero
fixed points $\fixedpointtheta$''.  We claim that $\Aevent$ is
contained within the event $\Bevent$, defined as follows
\begin{align}
  \label{eqn:set_inclusion}
  \Aevent \subseteq \underbrace{\braces{\frac{1}{n} \sum_{i =
        1}^{\obs} Y_i^2 > 1}}_{= \,: \Bevent}.
\end{align}
Deferring the proof of this claim to the end of this section, we now complete
the proof of our original claim.
Note that the event $\Bevent$ is implied by the event in the
bound~\eqref{eq:first_bnd}, and hence we have non-zero fixed points
under the same event.  Now, for any of these non-zero fixed points,
dividing both sides of inequality~\eqref{eqn:Theta_cube_lower_bd} by
$|\fixedpointtheta|$ and using the bound~\eqref{eqn:const-prob_lb}, we
conclude that
\begin{align*}
    \Prob \brackets{ |\fixedpointtheta|^2 \geq \frac{ 
    \sqrt{ \log 17} / 4}{c_1} \cdot \frac{1}{ \sqrt{\obs}}}
     \geq \frac{1}{34},
\end{align*}
as claimed in the theorem.

We now prove our earlier claims~\eqref{eq:tanh_positive}-\eqref{eq:tanh_negative}
and \eqref{eqn:set_inclusion}.

\paragraph{Proof of the bounds~\eqref{eq:tanh_positive} 
and~\eqref{eq:tanh_negative}} 
\label{par:proof_of_claims_eq:tanh_positive_and_eq:tanh_negative_}
Note that it suffices to establish that
\begin{align}
  \label{eq:tanh_bound}
  y \tanh( y) \geq y^2 - y^4 / 3,\quad\mbox{ for all $y \in \real$. }
\end{align}
Indeed, a change of variable $y = \scalar x$ and dividing both sides
by $\scalar$ yield the desired claims.  Using the fact that $\tanh(y)
= (e^y - e^{- y})/( e^y + e^{- y})$, it remains to verify that
\begin{align*}
y (e^y - e^{- y}) \geq (e^y + e^{- y}) \cdot (y^2 - y^4 / 3)
\end{align*}
or equivalently that
\begin{align*}
\sum_{k=0}^\infty \frac{2 y^{2 k + 2}}{(2 k + 1)!}  \geq
\sum_{k=0}^\infty \frac{2 y^{2 k}}{(2 k)!} \cdot (y^2 - y^4 / 3) =
\sum_{k=0}^\infty \frac{2 y^{2 k + 2}}{(2k)!} \cdot (1 - y^2 / 3),
\end{align*}
which simplifies to
\begin{align*}
\sum_{k = 1}^\infty \frac{y^{2 k + 2}}{(2 k + 1)!}  \parenth{\frac{1}{(2 k +
    1)!} - \frac{1}{(2 k)!} +\frac{1}{3 (2k - 2)!}} \geq 0.
\end{align*}
Since only even powers of $y$ exist on both sides in the power series,
it suffices to verify that each coefficient on the LHS is
non-negative.  After some algebra, we find that the condition above
reduces to
\begin{align*}
\frac{1}{2 k + 1} + \frac{(2 k - 1)2 k}{3} -1\geq 0, \quad \mbox{for
  all $k \geq 1$.}
\end{align*}
This elementary inequality is indeed true, and so the proof is complete.

\paragraph{Proof of set-inclusion~\eqref{eqn:set_inclusion}}
Consider the (random) function $g : \real \to \real$ such that
\mbox{$g( \theta) \defn M_n( \theta) - \theta$.}  Also introduce the
shorthand \mbox{$Z = \sum_{i=1}^n Y_i^2 / n$,} and note that $\Bevent
= \{ Z > 1 \}$. Note that any fixed point of the operator $M_n$ is a
zero of the function $g$ and vice-versa.  It is easy to see that the
function $g$ is twice continuously differentiable.  Now for the event
$\{ Z > 1 \}$, the function $g$ satisfies $g(0) = 0$ and $g'(0) > 0$
and hence there exists $c>0$ such that $g(c) > 0$.  Furthermore for
any sequence of $Y_i$'s, we have that $\lim_{\theta \to \infty} g(
\theta) = - \infty$.  Putting the two pieces together, we obtain that
under the event $\Bevent$, the function $g$ has at least one strictly
positive root.  Since $g$ is an odd function, we also have that under
the same event, the function $g$ has at least one strictly negative
root. The claim now follows.


\section{Proofs of auxiliary lemmas} 
\label{sec:proofs_of_auxiliary_lemmas}
In this appendix, we present the proofs of the auxiliary lemmas used in
the proofs of our main theorems.

\subsection{Proof of Lemma~\ref{lemma:bound_EM_operators}}
\label{sec:proof_of_lemma_bound_em_operator}

The proof of this lemma is based on standard arguments to derive
Rademacher complexity bounds~\cite{Vaart_Wellner_2000, Wai19}.  First,
we reduce the supremum of random variables over an uncountable set to
a finite maximum.  We then symmetrize with Rademacher variables, and
then apply the Ledoux-Talagrand contraction inequality.  Finally, we
exploit tail bounds on sub-Gaussian and sub-exponential random
variables so as to obtain the desired claim.

Let $\sphere^{d} = \left\{ \myvec \in \Rspace^{d} \mid \enorm{ \myvec}
= 1 \right\}$ denote the unit sphere in $d$-dimensions.  Then, we have
\begin{align*}
Z \defn \sup_{\theta \in \ball(0, r)} \enorm{ \updateM_\obs (\theta) -
  \updateM (\theta)} & = \sup_{ \theta \in \ball(0, r)} \sup_{\myvec \in
  \sphere^\dims} ( \updateM_\obs (\theta) - \updateM (\theta))^\top
\myvec \\
& =\sup_{\myvec \in \sphere^\dims} \underbrace{ \sup_{\theta \in
    \ball(0, r)}(\updateM_\obs (\theta) - \updateM (\theta))^\top
  \myvec}_{\rdefn Z_\myvec}.
\end{align*}
Note that $Z$ is defined as the supremum over the sphere
$\sphere^\dims$.  Using a standard discretization argument, we reduce
our problem to a maximum over a finite cover. In particular, we denote
$\left\{\myvec^{1}, \ldots, \myvec^{\cover} \right\}$ a 1/8-cover for
the unit sphere $\sphere^{\dims}$.  It is well known that we can find
such a set with $\cover \leq 17^\dims$.  Using the usual
discretization argument (see Chapter 6,~\cite{Wai19}), we can show
that
\begin{align}
\label{eq:Z_Z_u}
Z \leq \max_{j \in [\cover]} \frac{8 Z_{\myvec^j}}{7}.
\end{align}
Consequently, it is sufficient to study the behavior of the random
variables $Z_{\myvec^j}$ for $j \in [\cover]$, which we do next.

Substituting this relation into the definitions~\eqref{EqnPopMupdate}
and~\eqref{eq:sample_em_operator} of the EM operators $\updateM_\obs$
and $\updateM$, respectively, we find that
\newcommand{\zonek}[1][k]{A_{\myvec^{#1}}}
\newcommand{\ztwok}[1][k]{B_{\myvec^{#1}}}
\begin{align*}
Z_{\myvec^k} &= \sup \limits_{\theta \in \ball(0,r)} \biggr \{
\dfrac{1}{\obs} \sum \limits_{i=1}^{\obs} ( 2
\weightFun_{\theta}(X_{i}) - 1 ) X_{i}^\top \myvec^k - \Exs \brackets{(2
  \weightFun_{\theta}(X) - 1) X^\top \myvec^k} \biggr \} \\
  &= \parenth{\frac{1}{n}\sum_{i=1}^n X_i\tp \myvec^k - \Exs[X\tp \myvec^k]}
  \cdot (2\weight-1)
   \\
  &\qquad+ \sup \limits_{\theta \in
  \ball
  (0,r)} \biggr \{
\parenth{\dfrac{1}{\obs} \sum \limits_{i=1}^{\obs}  2
(\weightFun_{\theta}(X_{i})-\weight) X_{i}^\top - \Exs \brackets{2
  (\weightFun_{\theta}(X)-\weight) \cdot X^\top }}\myvec^k \biggr \} \\
  &= \zonek + \ztwok,
\end{align*}
and thereby that
\begin{align}
\label{eq:z_bound}
  Z \leq \frac{8}{7} \left \{ \max_{j \in [\cover]} \zonek[j] +
  \max_{j \in [\cover]} \ztwok[j] \right \}.
\end{align}
Noting that $X_i\tp \myvec^j \stackrel{i.i.d.}{\sim} \NORMAL(0,
\sigma^2)$ and that $\cover \leq 17^\dims$, standard concentration
bounds yield that
\begin{subequations}
\label{eq:bounds_a_and_b}
\begin{align}
  \Prob\brackets{\max_{j \in [\cover]} \zonek[j] \leq \abss{2\weight-1}
  \sigma
  \sqrt{
  \frac{\dims
    \log 17
    +
    \log
    (1/\delta)}{\obs}}} \geq 1-\delta.
\end{align}
On the other hand, for the random variables $\ztwok[j]$, we claim the following
bound
\begin{align}
\label{eq:bound_bk}
\Prob\brackets{
\max_{ j \in [\cover]} \ztwok[j] \leq c' r \sigma^{2} 
 \sqrt{
\dfrac{d + 
    \log(1 / \delta)}{n} }} \geq 1-\delta.
\end{align}
\end{subequations}
Putting the bounds~\eqref{eq:z_bound} and \eqref{eq:bounds_a_and_b} 
together yields the claim of the lemma.

\paragraph{Proof of the bound~\eqref{eq:bound_bk}}
Using a symmetrization bound~\cite{Vaart_Wellner_2000, Wai19}, we find
that
\begin{align}
\label{eqn:bound_EM_operators_zero}
\Exs[\exp(\lambda \ztwok) ] \leq \Exs \brackets{\exp
  \parenth{\sup_{\theta \in \ball(0,r)} \dfrac{2 \lambda}{n} \sum
    \limits_{i=1}^{n} \radem_{i} 2(\weightFun_{\theta}(X_{i})-\weight)X_
    {i}^\top
    \myvec^k)} },
\end{align}
for any
$\lambda >0$ where $\radem_{1}, \ldots, \radem_{n}$ denote
i.i.d.\ Rademacher random variables which are independent of
$\braces{X_i, i \in [\obs]}$.  We now make use of the Ledoux-Talagrand
contraction inequality for Lipschitz functions of Rademacher
processes~\cite{Ledoux_Talagrand_1991}.  For each fixed $x$, define
the function $f_x(\theta) \defn 2 \left( \weightFun_{\theta}(x) -
\weight \right)$.  Since $\weightFun_0(x) = \weight$ for all $x$, we
have $f_x(0) = 0$, so that this function is centered.  Moreover, for
any pair $(\theta, \theta')$, we have
\begin{align*}
\abss{f_x(\theta) - f_x(\theta')} & = \abss{2 \weightFun_{\theta}(x) -
  2 \weightFun_{\theta'}(x)} \leq 2\abss{\theta^\top x -
  (\theta')^{\top} x},
\end{align*}
so that $f_x(\theta)$ is $2$-Lipschitz in the quantity $\theta^\top
x$.  Consequently, applying the Ledoux-Talagrand contraction
inequality for this map, we find that 
\begin{align*}
& \hspace{- 4 em} \Exs \left[ \exp \left( \sup_{\theta \in \ball(0,r)}
    \dfrac{2\lambda}{n} \sum_{i=1}^n \radem_{i}(2
    (\weightFun_{\theta}(X_{i})-\weight) X_{i}^ {\top} \myvec^k)
    \right) \right] \nonumber \\ & \hspace{5 em} \leq \Exs \left[ \exp
    \left( \sup_{\theta \in \ball(0,r)} \dfrac{4 \lambda}{n}
    \sum_{i=1}^n \radem_i \theta^\top X_i X_i^\top \myvec^k \right)
    \right].
\end{align*}
Furthermore, using the fact that $\enorm{\myvec^k} =1$ and the
standard bound $\myvec^\top \mymat \myvectwo \leq
\enorm{\myvec} \opnorm{\mymat} \enorm{\myvectwo}$, we obtain that
\begin{align}
& \hspace{ -4 em} \Exs \left[ \exp \left( \sup_{\theta \in \ball(0,r)} \dfrac{4 \lambda}{n}
  \sum_{i=1}^n \radem_i \theta^\top X_i X_i^\top \myvec^k \right) \right] \nonumber \\
& \hspace{3 em} \leq \Exs \left[ \exp \left( \sup \limits_{\theta \in \ball(0, r)}
  4 \lambda \Vert{ \myvec^k} \Vert_2 \enorm{ \theta} \,
  \opnorm{ \dfrac{1}{n} \sum_{i=1}^{n} \radem_{i} X_{i} X_{i}^{\top}}
  \right) \right] \nonumber \\
\label{eqn:bound_EM_operators_first}
& \hspace{3 em} \leq \Exs \left[ \exp \left( 4 \lambda r\,
  \opnorm{\dfrac{1}{n} \sum_{i=1}^{n} \radem_{i} X_{i} X_{i}^{\top}}
  \right) \right].
\end{align}  
We now make two auxiliary claims: \\
\begin{subequations}
(a) The operator norm of the matrix $\sum_{i=1}^{n}
  \radem_{i}X_{i} X_{i}^{\top}/n$ can be bounded as follows:
\begin{align}
\label{eqn:bound_EM_operators_second}
\opnorm{\dfrac{1}{n} \sum_{i=1}^{n} \radem_{i} X_{i} X_{i}^{\top}} \leq
2 \max_{j \in [\cover]} \abss{\dfrac{1}{n} \sum_{i=1}^{n}
  \radem_{i} (X_{i}^\top \myvec^j)^2 }.
\end{align}
(b) For all $(i, j) \in [\obs] \times [\cover]$, we have
\begin{align}
\label{eqn:sub_exp_bound}
\Exs \brackets{\exp(t \radem_{i}(X_{i}^\top \myvec^j)^{2})} \leq
\exp(17 \cdot t^{2} \sd^4) \quad \mbox{for all $\abss{t} \leq
  \frac{1}{4 \sd^2}$.}
\end{align}
\end{subequations}
The claim~\eqref{eqn:bound_EM_operators_second} follows by the same
discretization argument that we used before (see Chapter 6 in the
book~\cite{Wai19}). We return to prove the
claim~\eqref{eqn:sub_exp_bound} at the end of this appendix.

Taking these claims as given for the moment, let us now complete the
proof of the bound~\eqref{eq:bound_bk}. 
Putting together the pieces, we find that
\begin{align*}
\Exs[ \exp( \lambda \ztwok) ] &
\stackrel{\mathmakebox[\widthof{======}]{(\mathrm{bnd.}~\eqref{eqn:bound_EM_operators_zero},
    ~\eqref{eqn:bound_EM_operators_first} ) }} {\leq} \Exs \left [
  \exp \left( 4\lambda r\, \opnorm{\dfrac{1}{n} \sum_{i=1}^{n}
    \radem_{i} X_{i} X_{i}^{\top}} \right) \right] \\
& \stackrel{\mathmakebox[\widthof{======}]{(\mathrm{eqn.}
    ~\eqref{eqn:bound_EM_operators_second} ) }} {\leq} \Exs \brackets{
  \exp \parenth{\max_{j \in [\cover]} \frac{8 \lambda r}{\obs}
    \abss{\sum_{i = 1}^\obs \radem_i (X_i^\top \myvec^j)^2 }} } \\
& \stackrel{\mathmakebox[\widthof{======}]{}} {\leq} \Exs \brackets{
  \exp \parenth{ \max_{j \in [\cover]} \frac{- 8 \lambda
      r}{\obs}{\sum_{i=1}^\obs \radem_i (X_i^\top \myvec^j)^2 }} }
\nonumber \\
& \hspace{ 8 em} + \Exs \brackets{ \exp \parenth{ \max_{j \in [\cover]}
    \frac{8 \lambda r}{\obs}{\sum_{i=1}^\obs \radem_i
      (X_i^\top \myvec^j)^2 }} }\\
&
\stackrel{\mathmakebox[\widthof{======}]{(\mathrm{eqn.}
~\eqref{eqn:sub_exp_bound}
    ) }} {\leq} 2 \cover \cdot \prod_{i=1}^\obs \exp \parenth{17
  \cdot \frac{64 \lambda^2 r^2}{\obs^2} \cdot \sd^4}
\end{align*}
for any $\abss{\lambda} \leq \obs/(32 r \sd^2)$.  Now invoking the
inequality $2 \cover \leq 34^\dims \leq e^{4 \dims}$, we find that
\begin{align*}
\Exs[\exp(\lambda \ztwok) ] \leq \exp \parenth{c \cdot
  \lambda^2 r^2 \sd^4 / \obs + 4 \dims} \quad \text{ for any } k \in
    [\cover],
\end{align*}
and sufficiently small $\lambda$.  Now using the fact that
$\cover \leq e^{3\dims}$, we obtain that
\begin{align*}
\Exs[ \exp(\lambda \max_{ j \in [\cover]}
  \ztwok[j])] 
  \leq \cover \exp( c \cdot
  4 \lambda^{2} r^{2} \sigma^{4} / \obs + 4 \dims)
  \leq \exp( c \cdot \lambda^{2} r^{2} \sigma^{4} / \obs + 7 \dims),
\end{align*}
for some constant $c$.
Using the standard approach for applying Chernoff bound, we have that
\begin{align*}
\max_{ j \in [\cover]} \ztwok[j] \leq c r \sigma^{2} \cdot \sqrt{
\dfrac{d +
    \log(1 / \delta)}{n} }, \quad \mbox{with probability at least $1 -
  \delta$,}
\end{align*}
as long as $n \geq c'(\dims +\log(1/\delta)$ for some suitable constants
$c$ and $c'$. \\

\noindent We now return to prove our earlier
claim~\eqref{eqn:sub_exp_bound}.


\paragraph{Proof of claim~\eqref{eqn:sub_exp_bound}} 
\label{par:proof_of_claim_eq:sub_exp_bound}

Noting that $X_i\tp \myvec^j \stackrel{i.i.d.} {\sim} \Ncal(0,\sigma^2) $,
and the fact that square of a sub-Gaussian random variable with parameter
$\sigma$ is a
sub-exponential random variable with parameter $(4\sigma^2, 4\sigma^2)$,
we obtain the following inequality~\cite{Vershynin_2011}:
\begin{align}
\label{eqn:square_sub_gaussian}
    \Exs \brackets{ \exp \parenth{ t (X_{i}^\top \myvec^j)^2 - t \Exs( X_
    {i}^\top \myvec^j)^2}}
    \leq e^{16 t^2 \sigma^4} 
    \quad \text{for all }\abss{t} \leq \frac{1}{4 \sigma^2}.
\end{align}
Noting that the random variable $\radem_i$ is independent of $X_i^\top
\myvectwo$, we find that
\begin{align*}
\Exs \brackets{\exp(t \radem_i (X_i^\top \myvec^j)^2)} &= \frac{1}{2}
\Exs\brackets{\exp(t (X_i^\top \myvec^j)^2)} + \frac{1}{2}
\Exs\brackets{\exp(- t(X_i^\top \myvec^j)^2)} \\
& \stackrel{(i)}{\leq} e^{16 t^2 \sigma^4} \cdot \frac{1}{2}
\brackets{e^{t \sigma^2}+e^{- t \sigma^2}}\\ &\stackrel{(ii)}{\leq}
e^{17 t^2 \sigma^4},
\end{align*}
for all $\abss{t} \leq \frac{1}{4 \sigma^2}$.  In asserting the above
sequence of steps, we have applied the
inequality~\eqref{eqn:square_sub_gaussian} along with the fact that
$\Exs(X_{i}^\top \myvec^j)^2 = \sigma^2$ to conclude step~(i), and
step~(ii) follows from the inequality $e^{x} + e^{- x} \leq 2 e^{ x^2}$
for all $x \in \real$. The claim now follows.


\subsection{Proof of Lemma~\ref{LemOperatorB}}
\label{AppLemOperatorB}

We begin with the elementary inequality \mbox{$\exp(y) + \exp(- y)
  \geq 2 + y^{2}$,} valid for all $y \in \real$, to find that
\begin{align}
\label{eqn:contractive_over_specified_first}  
[\mymat_{ \theta_u}]_{11} & = \Exs_V \left[ \frac{ V_{1}^{2}}{\left(
    \exp \left( - \| \theta_{u} \|_{2} V_1 / \sd \right) + \exp
    \left( \| \theta_{u} \|_{2} V_1 / \sd \right) \right)^{2}}\right] \nonumber \\
& \leq
 \Exs_{V_1} \left[ \frac{ V_{1}^{2}}{\left(2 +
    V_{1}^{2}\| \theta_{u} \|_{2}^{2} / \sigma^{2} \right)^{2}}\right].
\end{align}
Letting $\mathbb{I}_A$ denote the indicator random variable for event $A$,
i.e., it takes value $1$ when the event $A$ occurs and $0$ otherwise.
Then we have
\begin{align}
\Exs \left[ \frac{ V_1^{2}}{\left( 2 +
    V_1^{2}\| \theta_{u} \|_{2}^{2} / \sigma^{2} \right)^{2}}\right] & =
\Exs \left[ \frac{ V_1^{2}}{ \left( 2 +
    V_1^{2} \| \theta_{u} \|_{2}^{2} / \sigma^{2} \right)^{2}}
  \mathbb{I}_{\left\{ | V_1 | \leq 1 \right\}} \right] \nonumber \\
  & \hspace{7 em} + \Exs
\left[ \frac{ V_1^{2}}{\left( 2 +
    V_1^{2} \| \theta_{u} \|_{2}^{2} / \sigma^{2} \right)^{2}}
  \mathbb{I}_{\left\{ |V_1| > 1 \right\}} \right] \nonumber \\
\label{eqn:contractive_over_specified_second}
& \leq \dfrac{1}{4} \Exs \left[ V_1^{2} \mathbb{I}_{ \left\{ |V_1| \leq 1
     \right\}} \right] + \Exs \left[ \frac{ V_1^{2}}{\left( 2 +
    \| \theta_{u} \|_{2}^{2} / \sigma^{2} \right)^{2}} \mathbb{I}_{ \left\{|V_1| >
    1 \right\}} \right].
\end{align}
Here the final inequality is a consequence of the following observation:
\begin{align}
\label{eq:observation_V_1}
\frac{V_1^{2}}{\left( 2 + V_1^{2}
  \| \theta_{u} \|_{2}^{2} / \sigma^{2} \right)^{2}} \leq
\begin{cases}
  \displaystyle \frac{ V_1^{2}}{4}
  \quad& \text{if}\ \abss{ V_1} \leq 1,
  \\ \displaystyle \frac{ V_1^{2}}{(2 + 
  \| \theta_{u} \|_{2}^{2} / \sigma^{2})^{2}}
  \quad& \text{if} \ \abss{V_1} > 1.
\end{cases}
\end{align}
Putting the inequalities~\eqref{eqn:contractive_over_specified_first}
and~\eqref{eqn:contractive_over_specified_second} together, we
conclude that
\begin{align*}
[\mymat_{\theta_u}]_{11} \leq
\dfrac{1}{4} \Exs \left[ V_1^{2} \mathbb{I}_{\left\{ |V_1| \leq 1 \right\}} \right] +
\Exs \left[ \frac{V_1^{2}}{\left( 2 +
    \| \theta_{u} \|_{2}^{2} / \sigma^{2} \right)^{2}} \mathbb{I}_{ \left\{ |V_1| >
    1 \right\}} \right],
\end{align*}
where $V_1 \sim \Ncal(0,1)$.  Define $p_1 \defn
\Exs \brackets{ V_1^{2} \mathbb{I}_{\left\{ |V_1| \leq 1 \right\}}}$.  Then
we can directly verify that
$\Exs \brackets{ V_1^{2} \mathbb{I}_{ \left\{ |V_1| \geq 1 \right\}}} = 1 - p_1$ and
consequently obtain that
\begin{align}
  \label{eqn:contractive_over_specified_third}
[\mymat_{ \theta_u}]_{11} \leq \frac{p_1}{4} + \frac{(1 - p_1)}{4}
\frac{1}{(1 + \enorm{\theta_u}^2 / (2\sd^2))^2}.
\end{align}
Now we bound the entries $[\mymat_{\theta_u}]_{jj}$, $j \neq 1$.
Using the standard inequality \mbox{$\exp(y) + \exp(- y) \geq 2 + y^{2}$} once
again and noting that $V_j \stackrel{\textrm{i.i.d.}}
{\sim}\Ncal(0,1)$, we find that
\begin{align}
[\mymat_{ \theta_u}]_{jj} & = \Exs_V
\left[ \frac{ V_j^2}{ \left( \exp \left(- \| \theta_{u} \|_{2} V_1 / \sd \right) 
+ \exp \left( \| \theta_{u} \|_{2} V_1 / \sd \right)\right)^{2}} \right] \nonumber \\
& \leq
\Exs_{V_1} \left[ \frac{1}{ \left(2 +
    V_1^{2} \| \theta_{u} \|_{2}^{2} / \sigma^{2} \right)^{2}} \right]. 
    \label{eqn:contractive_over_specified_fourth}
\end{align}
Similar to observation~\eqref{eq:observation_V_1}, we also have that
\begin{align}
[\mymat_{ \theta_u}]_{jj} & \leq \dfrac{1}{4} 
\Exs \left[ \mathbb{I}_{ \left\{ |V_1| \leq 1 \right\}} \right] +
  \Exs \left[ \frac{1}{ \left(2 +
      \| \theta_{u}\|_{2}^{2} / \sigma^{2}
       \right)^{2}} \mathbb{I}_{ \left\{ |V_1| >
      1 \right\}} \right].
\label{eqn:contractive_over_specified_fifth}
\end{align}
Define $p_2 \defn \Prob\parenth{|V_1| \leq 1 }$.
Putting together the
inequalities~\eqref{eqn:contractive_over_specified_fourth} and
\eqref{eqn:contractive_over_specified_fifth}, we obtain that
\begin{align}
[ \mymat_{\theta_u}]_{jj} \leq \frac{p_2}{4} + \frac{(1 - p_2)}{4}
\frac{1}{(1 + \enorm{ \theta_u}^2 / (2\sd^2))^2} \quad \text{ for } j = 2,
\ldots, d.
   \label{eqn:contractive_over_specified_sixth}
\end{align}
Note that
\begin{align*}
p_2 = \Prob\parenth{|V_1| \leq 1 }=\Exs \left[\mathbb{I}_{ \left\{
|V_1| \leq 1 \right\}}\right]
> \Exs \left[ V_1^{2} \mathbb{I}_{ \left\{|V_1| \leq 1 \right\}} \right]
= p_1,
\end{align*}
and consequently, the bound on the RHS
of inequality~\eqref{eqn:contractive_over_specified_sixth} is larger
than the RHS of
inequality~\eqref{eqn:contractive_over_specified_third}.  As a result,
we have
\begin{align*}
  \opnorm{\mymat_{\theta_u}} = \max_{j \in [d]}
         [\mymat_{\theta_u}]_{jj} \leq \frac{p_2}{4} +
         \frac{(1 - p_2)}{4}
         \frac{1}{(1 + \enorm{\theta_u}^2 / (2 \sd^2))^2},
\end{align*}
where $p_2 =  \Prob\parenth{|V_1| \leq 1 }$ and the 
claim~\eqref{eq:operator_B} follows.


\subsection{Proof of Lemma~\ref{LemRabbits}}
\label{AppLemRabbits}
We now prove the claim~\eqref{eq:theorem_lower_reduction} in two
steps.  First, we show that $\brackets{ \mymatTheta}_{jj} \geq
\brackets{ \mymatTheta}_{11}$ for all $j \in [ \dims]$. Then, we derive
the claimed lower bound for $\brackets{ \mymatTheta}_{11}$.


\paragraph{Proof of $\brackets{ \mymatTheta}_{jj} 
\geq \brackets{ \mymatTheta}_{11}$}
For all $j \neq 1$, by changing the order of integration, we obtain
that
\begin{align*}
\brackets{ \mymatTheta}_{jj} & = \int \limits_{0}^{1} \Exs_{V} \left[
  \frac{ V_j^2}{\left( \exp \left(- \| \theta_{u} \|_{2}V_1 / \sd \right) +
    \exp \left(\| \theta_{u} \|_{2} V_1 / \sd \right)  \right)^{2}} \right] du
\\
& \stackrel{(i)}{=} \int \limits_{0}^{1} \Exs_{V} \left[
  \frac{1}{ \left( \exp \left(- \| \theta_{u} \|_{2} V_1 / \sd \right) +
    \exp \left(\| \theta_{u} \|_{2} V_1 / \sd \right) \right)^{2}} \right] du
\\
& \stackrel{(ii)}{\geq} \int \limits_{0}^{1} \Exs_{V} \left[
  \frac{V_1^2}{ \left( \exp \left(- \| \theta_{u} \|_{2} V_1 / \sd \right) +
    \exp \left(\| \theta_{u} \|_{2} V_1 / \sd \right) \right)^{2}} \right] du
= \brackets{ \mymatTheta}_{11},
\end{align*}
where step (i) follows since $\Exs[ V_j^2] = 1$, and from the
fact that the random variables $\{ V_j, j \neq 1\}$ are independent of
the random variable $V_1$.  Finally, note that the map $|V_1| \mapsto
V_1^2$ is increasing in $| V_1|$, and for any fixed value of
$\theta_{u}$ the function ${| V_1| \mapsto
  \frac{1}{ \left( \exp \left(- \| \theta_{u} \|_{2} V_1 / \sd\right) +
    \exp \left(\| \theta_{u} \|_{2} V_1 / \sd \right) \right)^{2}}}$ is a
decreasing function of $|V_1|$; consequently, step (ii) above follows
from a standard application of the Harris inequality.\footnote{Harris
  inequality: Given any pair of functions $(f,g)$ such that the
  function $f: \real \mapsto \real$ is increasing, and the function
  $g: \real \mapsto \real$ is decreasing. Then for any real-valued
  random variable $U$ we have ${\Exs \parenth{f(U) g(U)} \leq
    \Exs( f(U)) \Exs( g(U))}$. Here we have assumed that all three
  expectations exist and are finite.  }
\paragraph{Lower bound on $\brackets{ \mymatTheta}_{11}$}
Substituting $\theta_u = u \theta$ in the expression for
$\brackets{ \mymatTheta}_{11}$, and noting that $\int_{0}^1(e^{a u} +
e^{- a u})^{- 2} du = \tanh(a) / (4 a)$, we obtain that
\begin{align*}
\brackets{ \mymatTheta}_{11} &= \Exs_{V_1} \left[ \int \limits_{0}^{1}
  \frac{ V_1^2}{\left( \exp \left(- u \enorm{ \theta} V_1 / \sd \right) +
    \exp \left(u \enorm{ \theta} V_1 / \sd \right) \right)^{2}} du \right] \\
& = \Exs_{V_1} \left[ \frac{ \sd V_1}{4 \enorm{\theta}}
  \tanh{ \frac{ \enorm{ \theta} V_1}{ \sd}} \right] \\ & \stackrel{(i)}{=}
\frac{1}{4} \Exs_{V_1} \brackets{\sech^2
  \parenth{ \frac{ \enorm{ \theta} V_1}{\sd}}} \nonumber \\
& = \Exs_{V_1} \left[
  \frac{1}{\left( \exp \left(- \| \theta \|_{2} V_1 / \sd \right) +
    \exp \left(\| \theta \|_{2} V_1 / \sd \right) \right)^{2}} \right],
\end{align*}  
where step (i) follows from Stein's Lemma for standard Gaussian
distribution\footnote{Stein's Lemma: For any differentiable function
  $g: \real \mapsto \real$, we have $\Exs \brackets{Y g(Y)} =
  \Exs \brackets{g'(Y)}$ where \mbox{$Y \sim \Ncal(0,1)$} provided that
  expectations $\Exs \brackets{g'(Y)}$ and $\Exs \brackets{Y g(Y)}$
  exist.}. Expanding the expression in the denominator, we obtain
\begin{align}
\brackets{ \mymatTheta}_{11} & = \Exs_{V_1} \left[ \frac{1}{2 +
    \exp \left(- 2\|\theta\|_{2} V_1 / \sd \right) +
    \exp \left( 2\| \theta \|_{2} V_1 / \sd \right) } \right] \notag \\
\label{eq:b_11}
& \geq \frac{1}{ \Exs_{V_1} \left[ 2 +
    \exp \left(- 2 \| \theta \|_{2} V_1 / \sd \right) +
    \exp \left(2 \| \theta \|_{2} V_1 / \sd \right) \right] },
\end{align}
where the last inequality follows from Jensen's inequality applied with
the convex function $y \mapsto \frac{1}{y}$ on $y \in (0,\infty)$.
Noting that $V_1 \sim \Ncal(0,1)$ and consequently that $\Exs_{V_1}
\parenth{\exp \parenth{y V_1}} = e^{y^2 / 2}$ for all $y \in \real$, we
obtain that
\begin{align}
  \label{eq:b_11_last}
\Exs_{V_1} \Big[ 2 + \exp \left(- 2\| \theta \|_{2} V_1 / \sd \right) +
  \exp \left( 2\| \theta \|_{2} V_1 / \sd \right) \Big] & = 2 (1 +
e^{2 \enorm{ \theta}^2 / \sd^2}) \nonumber \\
& \leq 4 (1 + 2 \enorm{ \theta}^2 / \sd^2),
\end{align}
for all $\theta$ such that $\enorm{\theta}^2 \leq {5\sd^2}/{8}$.  Here
the last step follows from the fact that $e^t \leq 1 + 2t $, for all
$t \in [0, 5/4]$.  Putting the bounds~\eqref{eq:b_11} and
\eqref{eq:b_11_last} together yields the claimed lower bound \mbox{for
  $[\mymatTheta]_{11}$.}


\subsection{Proof of Lemma~\ref{LemNonExpansive}}
\label{AppLemNonExpansive}
Note that it is sufficient to show that a one-step update is
non-expansive.  Without loss of generality, we can assume that
$\enorm{ \theta^t} \geq \sqrt{2} \sd \dndelta^{ \seqalpha{\ind + 1}}$,
else we can start with the assumption $\enorm{ \theta^t} \geq \sqrt{2}
\sd \dndelta^{ \seqalpha{ \ind + 2}}$ and mimic the arguments that
follow.  Applying the triangle inequality, we find that
\begin{align*}
  \enorm{ \theta^{t + 1}} = \enorm{ \updateM_\obs( \theta^t )} & \leq
  \enorm{ \updateM_\obs ( \theta^t) - \updateM ( \theta^t)} +
  \enorm{ \updateM ( \theta^t)} \\
& \stackrel{(i)}{\leq} \uniconone \sd \cdot \sqrt{2} \sd
  \dndelta^{ \seqalpha{\ind}}\cdot \sqrt{ \dndelta} + \gamup( \theta^t)
  \enorm{ \theta^t} \\
& \stackrel{(ii)}{\leq} \uniconone \sd \cdot \sqrt{2} \sd
  \dndelta^{\seqalpha{\ind}}\cdot \sqrt{\dndelta} + \parenth{1 - \frac{p
      \dndelta^{2 \seqalpha{\ind + 1}} }{2}} \sqrt{2}\sd
  \dndelta^{\seqalpha{\ind}} \\
& = \parenth{1 - \frac{p \dndelta^{2\seqalpha{\ind + 1}} }{2} + \uniconone
    \sd \sqrt{\dndelta} } \sqrt{2}\sd \dndelta^{\seqalpha{\ind}},
\end{align*}
where step~(i) follows from the bound~\eqref{eq:sup_m_bound} with $r =
\enorm{\theta^t} \leq \sqrt{2} \sd \dndelta^{\seqalpha{\ind}}$, and applying
Theorem~\ref{thm:pop_over}, and
step~(ii) follows from the condition that $\enorm{\theta^t} \geq
\sqrt{2}\sd \dndelta^{\seqalpha{\ind + 1}} $ and consequently that
$\gamup(\theta^t) \leq 1 - p\dndelta^{2\seqalpha{\ind + 1}} / 2$.
Note that $2\seqalpha{\ind+1} - 1 / 2 \leq - 2 \smallthreshold$ for
all $\ind \leq \imax - 1$ and $\dndelta \leq 1$.  As a result, for
${\obs \geq (2 \uniconone \sd / p)^{1 / (2 \smallthreshold)} \sd^2
  \dims \log(2 \imax / \delta)}$, we have that $\dndelta^{2
  \seqalpha{\ind + 1} - 1 / 2} \geq \dndelta^{- 2 \smallthreshold}
\geq 2 \uniconone \sd / p$ and thereby that
\begin{align*}
 \parenth{1 - \frac{p \dndelta^{2 \seqalpha{\ind + 1}} }{2} +
   \uniconone \sd \sqrt{ \dndelta} } \leq 1.
 \end{align*}
Putting all the pieces together yields the
result.


\section{Additional results}
  \label{sec:proofs_of_additional results}
  
In this appendix, we provide additional results to support several
claims in the paper.


\subsection{Initial conditions}
\label{sec:ini_unbalanced}

The next lemma shows that  for the mixture models analyzed in this
paper, the population EM operator $\updateM$ maps any $\theta \in
\real^\dims$ to a ball of radius $\sqrt{2/ \pi}$ namely, the radius
is independent of the dimension $\dims$.  Given the uniform bounds
provided in Lemma~\ref{lemma:bound_EM_operators}, loosely speaking,
$\enorm{\updateM_\obs(\theta^0)}$ is upper bounded by $\sqrt{2/\pi} +
\enorm{\theta^0}\sqrt{d/n}$ with high probability. Consequently, we
make an implicit assumption while elaborating our results that
$\enorm{\theta^0}$ is a constant and does not scale with dimension (provided
that the sample size is large enough to keep the second term small).

\begin{lemma}
\label{lemma:bound_initialization}
For both the unbalanced or balanced model fits~\eqref{EqModelFitNew},
when the true model is standard Gaussian, we have
\begin{align*} 
  \enorm{\updateM(\theta^{0})} & \leq \sqrt{\frac{2}{\pi}}
  \quad\text{for any}\quad \theta^0 \in \real^\dims.
\end{align*}  
\end{lemma}
\begin{proof}
The proof of this lemma is a direct consequence of the change of basis ideas
used in the proofs of Theorems~\ref{ThmUnbalanced} and \ref{thm:pop_over}
before. Using the definition of $\updateM$ and applying the transformation
$V = R X/ \sigma$ where $R \in \Rspace^{d \times d}$ is an orthonormal matrix
such that $R \theta = \| \theta \|_{2} e_{1}$,  and $e_{1}$ is the first
canonical basis vector in $\realdim$, we find that 
\begin{align*}
  \enorm{\updateM(\theta)} 
  & = \enorm{\Exs_{X} \brackets{ \parenth{\frac{\weight e^{- \frac{\theta^\top
      X}{\sd^2}} - (1 - \weight) e^{ \frac{\theta^\top X}{\sd^2}}}{\weight e^{- \frac{\theta^\top
      X}{\sd^2}} + (1 - \weight) e^{ \frac{\theta^\top X}{\sd^2}}}} X}}\\
  & = \enorm { \Exs_{V} \brackets{\parenth{\frac{\weight 
  e^{- \frac{\enorm{\theta} V_{1}}{\sd}} - (1 - \weight) 
  e^{ \frac{\enorm{\theta} V_{1}}{\sd}}}{\weight 
  e^{- \frac{\enorm{\theta} V_{1}}{\sd}} + (1 - \weight) 
  e^{ \frac{\enorm{\theta} V_{1}}{\sd}}}} V}} \\
  & = \abss { \Exs_{V_{1}} \brackets{\parenth{\frac{\weight 
  e^{- \frac{\enorm{\theta} V_{1}}{\sd}} - (1 - \weight) 
  e^{ \frac{\enorm{\theta} V_{1}}{\sd}}}{\weight 
  e^{- \frac{\enorm{\theta} V_{1}}{\sd}} + (1 - \weight) 
  e^{ \frac{\enorm{\theta} V_{1}}{\sd}}}} V_{1}}} \\
  & \leq \Exs_{V_{1}} \brackets{\abss{\frac{\weight 
  e^{- \frac{\enorm{\theta} V_{1}}{\sd}} - (1 - \weight) 
  e^{ \frac{\enorm{\theta} V_{1}}{\sd}}}{\weight 
  e^{- \frac{\enorm{\theta} V_{1}}{\sd}} + (1 - \weight) 
  e^{ \frac{\enorm{\theta} V_{1}}{\sd}}}} \abss{V_{1}}} 
  \leq \Exs_{V_{1}} \brackets{ \abss{V_{1}}} 
  = \sqrt{\frac{2}{\pi}}.
\end{align*}
The claim now follows.
\end{proof}
\subsection{Behavior of EM when the weight is unknown}
\label{sec:sunknown_mixtures}
We now discuss the case when the mixture weight $\weight \in (0, 1/ 2]$
in the model fit~\eqref{EqModelFitNew} is assumed to be unknown and is estimated
jointly with the (single) location parameter $\mean$ using EM. 
(The scale parameter is still assumed to be known and fixed to the true
value.)
For our case, given a set of i.i.d. samples $\braces{X_i}_{i=1}^\obs$,
the sample EM operators 
$\updateM_{1, \obs}:\realdim \times (0, 1/ 2] \mapsto (0, 1/ 2]$
and $\updateM_{2, \obs}: \realdim \times (0, 1/ 2]   \mapsto \realdim$ for the weight and location
parameters respectively take the form
\begin{align}
\label{eq:sample_em_operator_unknown}
\updateM_{1, \obs} (\theta, \weight) \defn \frac{1}{\obs} \sum_{i=1}^n \weightFun_
{\theta, \weight} (X_i),\quad\text{and}\quad
\updateM_{2, \obs}(\theta, \weight) \defn \frac{1}{\obs} \sum_{i=1}^n (2
\weightFun_{\theta, \weight}(X_i) - 1) X_i,
\end{align}
where the weight function $\weightFun_{\theta, \weight}$ is defined as
\begin{align}
  \label{EqnWeightFun_unknown}
  \weightFun_{\theta, \weight}(x) \mydefn
  \frac{\weight \exp
    \parenth{-\frac{\enorm{\theta - x}^2}{2\sd^2}}}{\weight \exp
    \parenth{-\frac{\enorm{\theta - x}^2}{2\sd^2}} + (1-\weight) \exp
    \parenth{-\frac{\enorm{\theta + x}^2}{2\sd^2}}}.
\end{align} 
Taking the infinite sample limit, we can define the corresponding population
EM operators $\updateM_{1}$ and \mbox{$\updateM_{2}$} for the weight and location parameters as follows:
\begin{align}
\label{eq:population_em_operator_unknown}
\updateM_{1} (\theta, \weight) \defn \Exs_{X} \brackets{ \weightFun_{\theta,
\weight}
(X)}, \ \ \text{and}\ \ 
\updateM_{2} (\theta, \weight) \defn \Exs_{X} \brackets{ (2 \weightFun_
{\theta, \weight}
(X) - 1) X},
\end{align}
where the expectation is over the true model $X \sim \NORMAL(0, I_\dims)$.
The next results characterize the contraction properties of these population
EM operators.
\begin{lemma}
\label{lemma:contraction_pop}
For any $\mean \in \Rspace^{d}$ and $\weight \in (0, 1/ 2]$, 
the population EM operators $\updateM_{1}$ and $\updateM_{2}$ satisfy
\begin{align}
\abss{ \updateM_{1} (\theta, \weight) - \weight} \leq \frac{ \parenth{
1 - c\contracasym^2} \enorm{ \mean}}{2},\ \text{and}\
\enorm{ \updateM_{2} (\theta, \weight)} \leq \parenth{ 1 - \frac{\contracasym^2}
{2}} \enorm{ \mean} \label{eq:contract_bound_unknown}
\end{align}
where $\contracasym \defn 1 - 2\weight \in (0, 1)$ and $c \in (1/ 2, 1)$
denotes a universal constant.
\end{lemma}
\noindent See the end of this appendix for the proof.

An immediate consequence of Lemma~\ref{lemma:contraction_pop} is the
following. Let $\overline{ \weight} < 1/ 2$ be any fixed constant.
Consider the population EM sequence $(\weight^{t}, \mean^{t})$
generated as $(\weight^{t + 1}, \mean^{t + 1})= (\updateM_{1}
(\theta^t, \weight^{t},), \updateM_{2} (\mean^{t} ,\weight^{t}))$
starting with an initialization $(\weight^{0}, \mean^{0}) \in (0, 1/2]
  \times \real^\dims$ such that
\begin{align}
\label{eq:intial_cond_prop}
  \weight^0 + \frac{ \enorm{ \mean^0}}{(1 - 2 \overline{ \weight})^2} \leq
  \overline{ \weight}.
\end{align}
Then we have
\begin{align*}
  \weight^{t} \leq \overline{ \weight}, \ \quad \ 
  \enorm{ \mean^{t}} \leq \parenth{1 - \frac{\parenth{1 
  - 2 \overline{ \weight}}^2}{2}}^{t + 1} \enorm{ \theta^0}.
\end{align*}
In simple words, the weight sequence $\weight^t$ remains bounded above by
$\bar{\weight}$ and the sequence $\theta^t$ for the location parameter 
converges geometrically to $\thetastar = 0$.
On the other hand, when the initialization does not satisfy the condition~\eqref{eq:intial_cond_prop},
the convergence of location parameter can become sub-linear, especially
when $\weight^0 \approx 1/2$. In simple words, if the
initial mixture is highly unbalanced, we would observe a geometric convergence
and as we show in the next corollary sample EM estimates would have
a statistical error of order $\obs^{-\frac{1}{2}}$.
When the condition~\eqref{eq:intial_cond_prop} is violated, loosely speaking
the initial parameters are close to those of a balanced mixture
and EM would depict the slower convergence on both algorithmic and (consequently)
statistical fronts similar to the results stated in Theorem~\ref{thm:pop_over}.
However, a rigorous proof for the later case is beyond the scope of this
paper and we only provide some numerical evidence in Figure~\ref{fig:unknown_weights}.

\begin{figure}[h]
  \begin{center}
    \begin{tabular}{cc}
      \widgraph{0.45\textwidth}{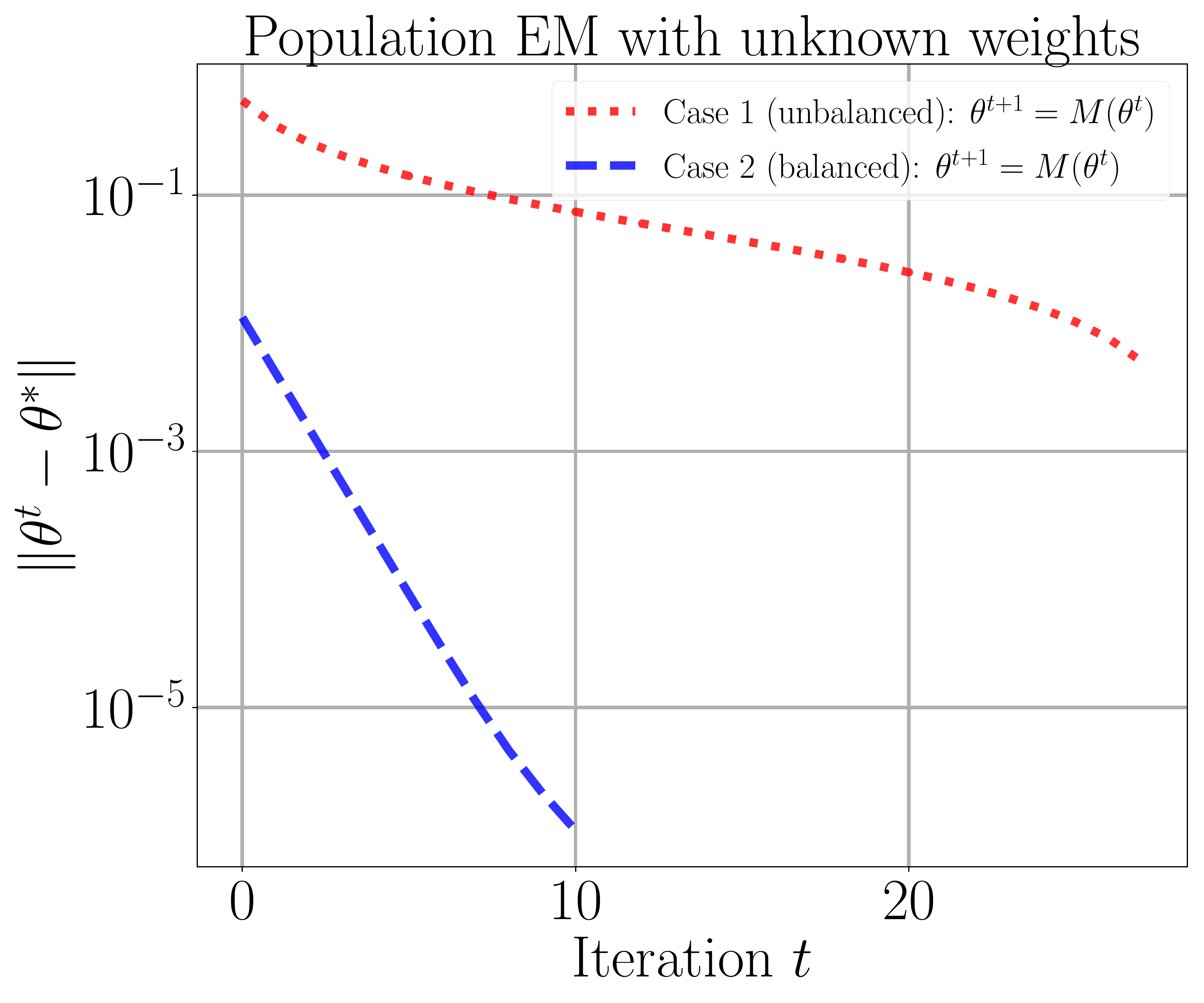} 
      & \widgraph{0.45\textwidth}{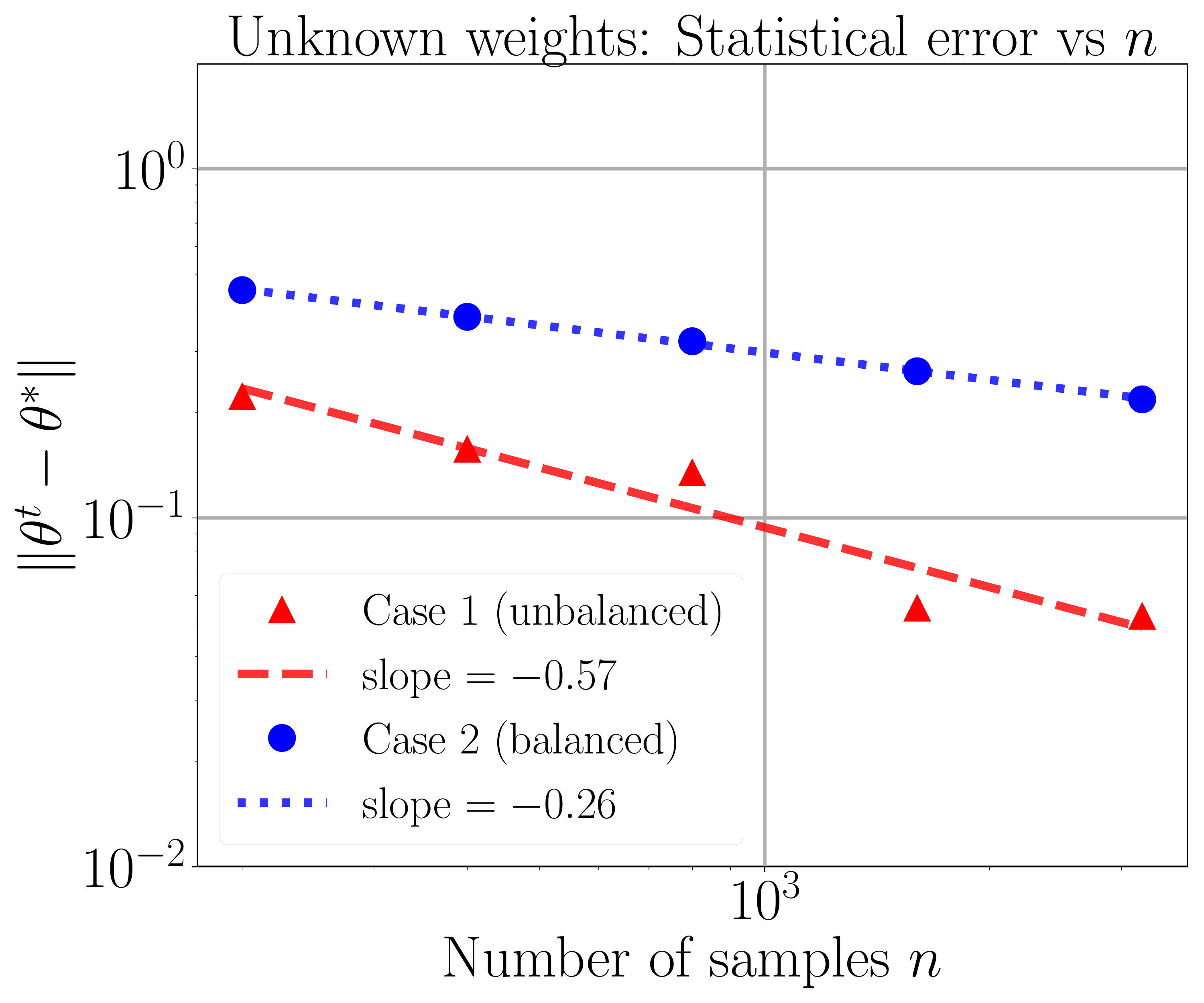}
      \\
      (a) & (b)
    \end{tabular}
    \caption{Behavior of EM for the two-mixture over-specified fit~\eqref{EqModelFit}
    with unknown weights where the true model is $\NORMAL(0, I_2)$.
    We consider two different initializations. Case 1 (unbalanced):
    When the initialization condition~\eqref{eq:intial_cond_prop} is met
    (in particular we set $\weight$ much smaller than $\frac{1}{2}$).
    In Case 2 (balanced), we initialize the weight parameter very close
    to $\frac{1}{2}$. Panel~(a) characterizes the population EM updates
    and panel~(b) depicts the statistical error with sample size $\obs$
    for the two cases. We see that when the condition~\eqref{eq:intial_cond_prop}
    is met, EM converges in few steps within error $\obs^{-\frac{1}{2}}$
    error and on the other hand when the initial weight is near $\frac{1}
    {2}$ we observe a slow convergence of EM with a larger statistical
    error of order $\obs^{-\frac{1}{4}}$.
    }
    \label{fig:unknown_weights}
  \end{center}
\end{figure}

\begin{corollary} 
\label{ThmUnbalanced_unknown}
Consider the sample EM sequences $\pi^{t + 1} = \updateM_{1, \obs}( \weight^
{t})$ and $\theta^{t+1} = \updateM_{2, \obs}(\theta^t)$ with an initialization
that satisfies the condition~\eqref{eq:intial_cond_prop} for some $\overline{\pi}<1/2$.
Then for any fixed $\delta \in (0, 1)$ and  \mbox{$\obs \geq
\unicon \, \dims \log(1/\delta) \frac{\sd^2}{\overline{\rho}^4}$}, 
we have
\begin{align}
  \label{EqnUnbalancedSample_unknown}
\weight^{t} \leq \overline{\weight}, \ \quad \ \enorm{\theta^t} \leq \enorm{\theta^0}
\brackets{ \parenth{ 1 -
    {\overline{\rho}^2} /{2} }^t + \frac{\uniconnew
    \sd^2}{\overline{\rho}^2} \sqrt{\frac{\dims
      \log(1/\delta)}{\obs}}\;},
\end{align}
with probability at least $1-\delta$, where $c, c'$ and $\overline{\rho}
= 1-2\overline{\pi} \in (0, 1)$ are universal constants.
\end{corollary}
\noindent The proof is fairly straightforward given the proof of 
Theorem~\ref{ThmUnbalanced} and is thereby omitted. However, it remains to 
prove Lemma~\ref{lemma:contraction_pop}.

\paragraph{Proof of Lemma~\ref{lemma:contraction_pop}} 

The upper bound for $\enorm{ \updateM_{2} (\mean, \weight)}$ follows
directly from the proof of Theorem~\ref{ThmUnbalanced}.  Turning to
the other bound in equation~\eqref{eq:contract_bound_unknown}, we see
that
\begin{align*}
  \abss{ \updateM_{1} (\theta, \weight) - \weight} 
  & = \weight (1 - \weight) \abss{ \Exs_{X} 
  \brackets{ \dfrac{ \exp \parenth{ X^{ \top} \mean} - 
  \exp \parenth{ - X^{ \top} \mean}}{\weight 
  \exp \parenth{ X^{ \top} \mean} + (1 - \weight) 
  \exp \parenth{ - X^{ \top} \mean}}}} \\
  & \leq 2 \weight (1 - \weight) \enorm{ \mean} 
  \max_{u \in [0, 1]} \opnorm{\Exs \left[\bar{\Gamma}_{\theta_u}(X) \right]},
\end{align*}
where $\mean_{u} = u \mean$ for $u \in [0, 1]$ and the matrix $\bar{\Gamma}_{\theta_u}(X)$ is defined as
\begin{align}
  \bar{\Gamma}_{\theta_u}(X) 
  \defn \frac{X}{\sigma^{2} \parenth{
    \weight \exp \left(- \frac{\theta_{u}^\top X}{\sigma^{2}}\right) +
    (1 - \weight) \exp \left(\frac{\theta_{u}^\top
      X}{\sigma^{2}} \right)}^{2}}.
\end{align}
Invoking the transformation as that from the proof of Theorem~\ref{ThmUnbalanced} 
and mimicking the arguments presented there, we can verify that
\begin{align*}
  \abss{ \Exs_{X} \brackets{ \bar{\Gamma}_{\theta_u}(X)}} 
  & = \abss{ \Exs_{V_{1}} \brackets{\dfrac{V_{1}}{\sigma 
  \parenth{\weight \exp \left(- \enorm{ \theta_{u}} 
  V_{1}/ \sigma \right) +
    (1 - \weight) \exp \left( \enorm{\theta_{u}}
    V_{1}/ \sigma \right)}^{2}}}} \\
    & \leq \Exs_{ V_{1}} \brackets{\dfrac{ \abss{ V_{1}} }
    {\sigma \parenth{\weight \exp \left(- \enorm{ \theta_{u}} 
    V_{1}/ \sigma \right) +
    (1 - \weight) \exp \left( \enorm{\theta_{u}}
    V_{1}/ \sigma \right)}^{2}}} \\
    & \leq \dfrac{(1 - \contracasym^2) 
    + \contracasym^2 \Exs_{V_{1}} \brackets{ 
    \abss{V_{1}} \mathbb{I}( V_1 \geq 0)}} {(1 - \contracasym^2)} \\
    & = \frac{1 - c \contracasym^2}{(1 - \contracasym^2)}
\end{align*}
where $V_{1} \sim \NORMAL(0, 1)$ and $c = 1 - \Exs_{V_{1}} \brackets{ \abss{V_
{1}} \mathbb{I}( 
V_1 \geq 0)} \in (1/ 2, 1)$.
Putting the above results together yields the claimed bound.

\section{Unbalanced vs balanced fits: Closer look at log-likelihood} 
\label{sec:unbalanced_vs_balanced_fits_closer_look_at_log_likelihood}

In this appendix, we provide a further discussion on the difference
between the unbalanced and balanced mixtures corresponding to the
model~ \eqref{EqModelFit} considered throughout our work.  Recall that
the expected (population) log-likelihood for the model
fit~\eqref{EqModelFit} is given by
\begin{align*}
  \mathcal{L}^\weight(\mean) 
  & = \Exs \brackets{\log \parenth{ \weight
      \normDensity\parenth{X;\mean, \sd^2 I_{d}} + (1 -
      \weight)\normDensity\parenth{X;-\mean,\sd^2 I_{d}}}},
\end{align*}
where $\normDensity(\cdot;\theta,\sd^2I_{\dims})$ denotes the
probability density of the Gaussian distribution $\NORMAL(\theta, \sd^2I_\dims)$.
Observe that
\begin{align*}
    \arg\max_{\theta} \mathcal{L}^\weight(\mean) = \arg\min_{\theta}
    \text{KL}(\NORMAL(0, \sd^2 I_d) \| \weight\NORMAL(\theta, \sd^2 I_d)
    +(1-\weight)\NORMAL(-\theta, \sd^2 I_d)),
\end{align*}
where $\text{KL}(P\| Q)$ denotes the Kullback-Leibler divergence between
the distributions $P$ and $Q$. Since the true distribution belongs to the
fitted class with $\thetastar=0$, finding maximizer of the
population log-likelihood would yield the true parameter $\thetastar$. 
As alluded to in the main paper, in
practical situations, when one has
access to only $n$ i.i.d. samples $\{X_i \}_{i=1}^n$, the most popular
choice to estimate $\thetastar$ is the maximum likelihood estimate
(MLE) given by equation~\eqref{eq:sample_likelihood}.

We now use the nature of log-likelihood to justify the difference
between unbalanced and balanced fits.
Note that the Fisher information matrix 
$\Fishermatrix^\weight(\theta) \defn -\nabla^2_\theta \mathcal{L}^\weight
(\mean)$ for the fit~\eqref{EqModelFit} with mixture weights $ (\weight,
1-\weight)$ is given by
\begin{align*}
[\Fishermatrix^\weight(\theta)]_{ii} 
&= - 4 \weight (1 - \weight) \Exs \brackets{\frac{Y_{i}^2}{(\weight \exp( \mean^{\top} Y) + (1 - \weight) \exp(- \mean^{\top}
    Y))^2}} + 1
\end{align*} 
for $i \in [d]$ and
\begin{align*}
[\Fishermatrix^\weight(\theta)]_{ij}
&= - 4 \weight (1 - \weight) \Exs \brackets{\frac{Y_{i} Y_{j}}{(\weight \exp( \mean^{\top} Y) + (1 - \weight) \exp(  -\mean^{\top}Y))^2}}
\end{align*}
for $i, j \in [d]$ such that $i \neq j$. Here the expectations are taken
under the true model $Y = (Y_{1},\ldots,Y_{d}) \sim \Ncal(0,I_{d})$.
Clearly, at $\theta=\thetastar=0$, we have
\begin{align}
\label{eq:fisher}
\Fishermatrix^\weight(\thetastar)= \beta^\weight I_\dims, \quad \text{where}
\quad \beta^\weight = - 4\weight (1 - \weight) + 1.
\end{align}
Note that  $\beta^\weight > 0$ for any $\weight \in (0, 1)$ such that $\weight
\neq 1/2$. On the other hand, for $\weight=1/2$, we have $\beta^{\weight} =
0$.
Consequently, we find that for any unbalanced fit with $\weight\neq 1/2$,
the Fisher matrix is positive definite at $\thetastar$, and, for the balanced
fit with $\weight=1/2$, it is singular at $\thetastar$.
Equivalently, the log-likelihood is strongly log-concave around $\thetastar$
for the unbalanced fit and weakly log-concave for the balanced fit.

\begin{figure}[h]
  \begin{center}
    \begin{tabular}{cc}
      \widgraph{0.45\textwidth}{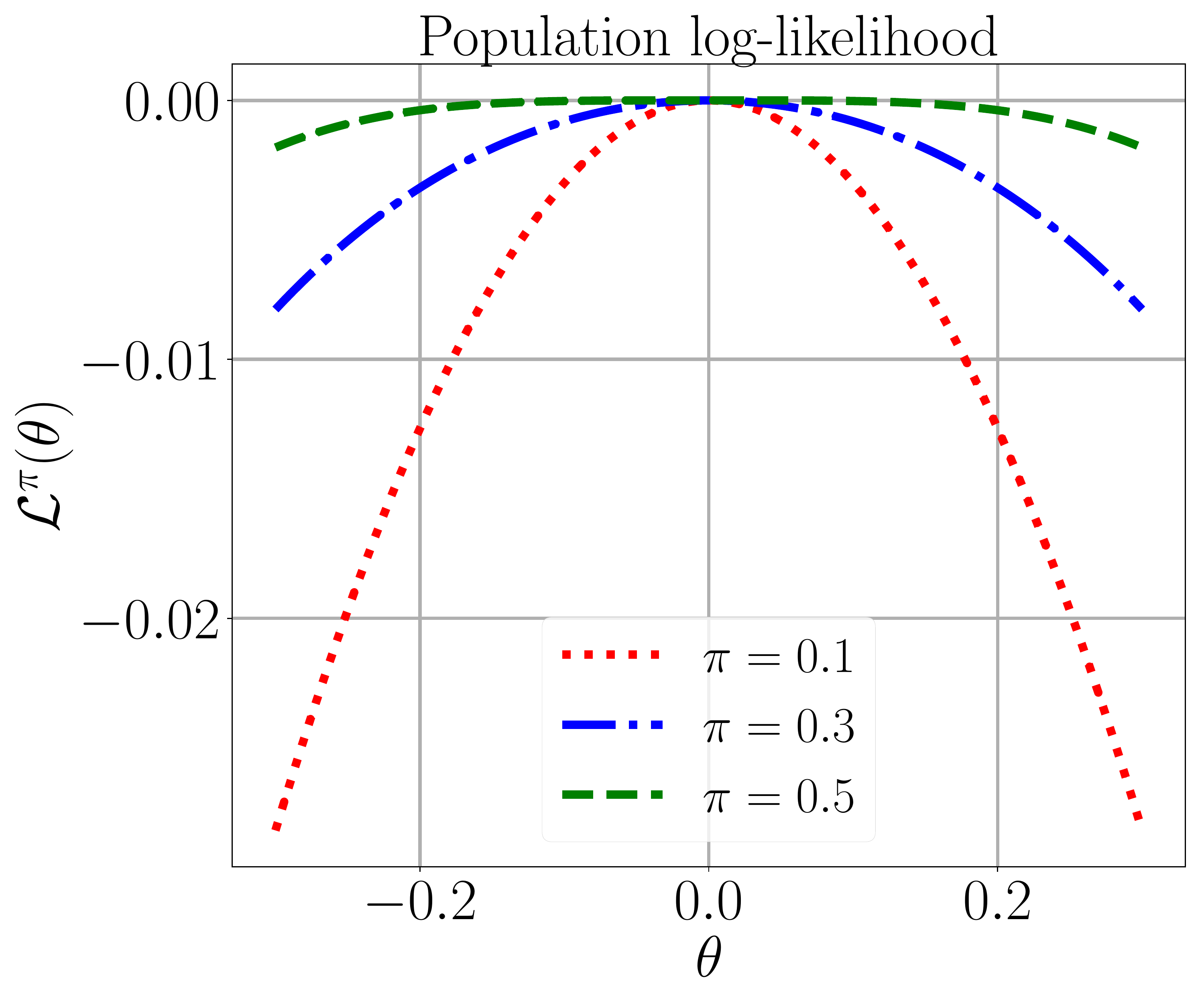} 
      & \widgraph{0.45\textwidth}{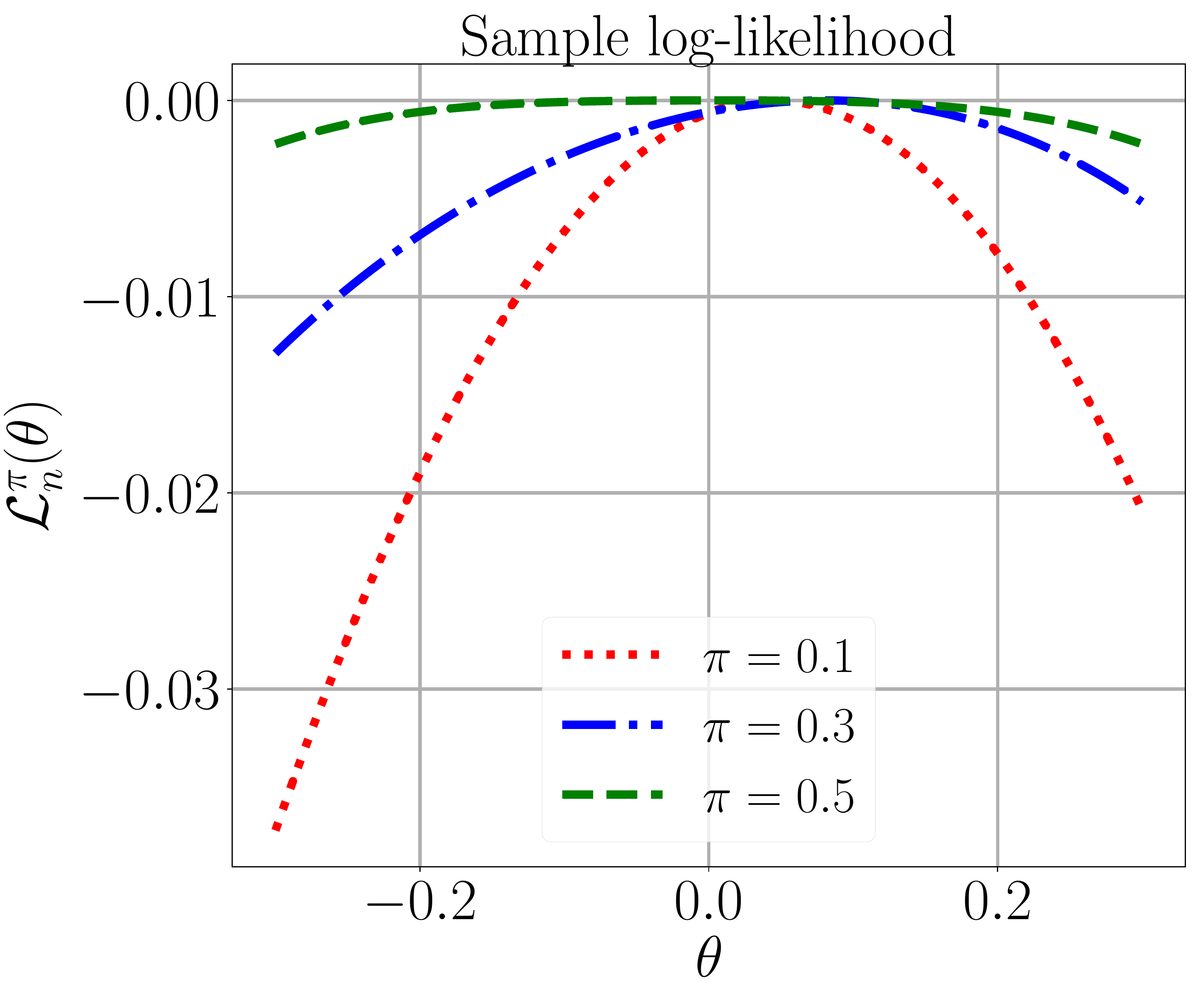} 
      \\
      (a) & (b)
    \end{tabular}
    \caption{Plots of the log-likelihood for the unbalanced and balanced fit
      for
      data generated from $\NORMAL(0, 1)$. 
      (a) Behavior of population log-likelihood 
      $\mathcal{L}^\weight$~\eqref{eq:pop_likelihood} (computed using
      numerical integration) as a function of $\theta$ for different weights
      $\weight \in \braces{0.1, 0.3, 0.5}$. 
      (b) Behavior of sample log-likelihood
      $\mathcal{L}_n^\weight$~\eqref{eq:sample_likelihood}
      with $n=1000$ samples for $\weight \in \braces{0.1, 0.3, 0.5}$.
      The plots in these panels portray a stark contrast in the shapes of the
      log-likelihood functions in the balanced and unbalanced case, it
      gets flatter around $\thetastar=0$ as $\weight \to 0.5$. 
      More concretely, in unbalanced case we see a quadratic type behavior
      (strongly concave); whereas in balanced case, the log-likelihood 
      function is flatter and depicts a fourth degree polynomial type (weakly
      concave) behavior.
    }
    \label{fig:pop_likelihood}
  \end{center}
\end{figure}

We numerically computed the population log-likelihood and plotted it in
Figure~\ref{fig:pop_likelihood}(a)\footnote{Figure~\ref{fig:pop_likelihood}(b),
  shows the sample likelihoods $\mathcal{L}_n^\weight$ based on $n=1000$
  samples, and weights $\weight \in \braces{0.1, 0.5}$.  We observe that
  while the sample-likelihood may have more critical points, its
  curvature resembles very closely the curvature of the corresponding
  population log-likelihood.}, where we observe that when
the mixture
weights are unbalanced ($\weight < 1/2$), the population
log-likelihood for the model has more curvature, and in fact is
(numerically) well-approximated as $\mathcal{L}^{\weight}(\theta)
\approx -c^\weight\theta^2$.  On the other hand, for the balanced
model with $\weight = \frac{1}{2}$, the likelihood is quite flat near
origin and is (numerically) well-approximated as
$\mathcal{L}^{\weight}(\theta) \approx -c\, \theta^4$. 
It is a folklore that the convergence rate
of optimization methods has a phase transition: optimizing
strongly concave functions is exponentially fast than weakly concave
functions.  As a result, we might expect why population
EM may have fundamentally different rate of algorithmic convergence in the
two model
fits as observed in Figure~\ref{fig:population_em} (in the main paper).

Moreover, the singularity of Fisher matrix is known to lead to a slow down
the statistical rate of MLE. It is well
established~\cite{vanderVaart-98} that when the Fisher matrix is invertible
in a neighborhood of the true parameter, MLE has the parametric rate of
$n^{-\frac1 2}$, i.e., the MLE estimate is at a distance of order $n^{-\frac
12}$ from the true parameter $\thetastar$. Moreover, as discussed in the
introduction of the main paper, several works~\cite{Chen1992,Nguyen-13}
have also shown than the singularity of the Fisher matrix may lead to a
slower than $n^{-\frac 12}$ rate for the MLE. 
Since EM algorithm is designed to estimate MLE (and converges only to local
maxima), we may loosely conclude
that, for the singular case (balanced fit), a slower
than parametric rate for the EM estimate is also expected.


\section{Mixture of regression}
\label{sec:mixture_regression_EM}
In this appendix, we provide formal results for the slow convergence of
EM for over-specified mixture of linear regression 
(as discussed in Section~\ref{sub:slow_rates_for_mixture_of_regressions}).

Given $n$ samples from the mixture of regressions model~\eqref{eq:truemodel_regression},
we use EM to fit the following model:
\begin{align}
\label{eq:fitted_regression}
  Y|X \sim \frac{1}{2}\NORMAL(\theta\tp X, 1)
  +\frac{1}{2}\NORMAL(-\theta\tp X, 1),
\end{align}
where we assume the knowledge of covariate design $X\sim\NORMAL(0, I_d)$.
Given this joint model on $(X, Y)$ the population log-likelihood for the
model is given by
\begin{align*}
  \mathcal{L}(\theta) = \Exs_{X, Y} \brackets{\log \parenth{ \weight
      \normDensity\parenth{Y;\theta\tp X, \sd^2 I_{d}} + (1 -
      \weight)\normDensity\parenth{Y;-\theta\tp X,\sd^2 I_{d}}}},
\end{align*}
where $\normDensity(\cdot;\theta,\sd^2I_{\dims})$ denotes the
probability density of the Gaussian distribution $\NORMAL(\theta, \sd^2I_\dims)$.
In Figure~\ref{fig:mix_of_reg}, we plot this log-likelihood as a function
of $\theta$ for two different values of $\thetastar$ and observe the following.
When the mixture has strong
signal ($\theta^* = 0.7$),
the Hessian of log-likelihood is negative definite (strongly concave) but
in the case of no signal $\theta^* = 0$ the Hessian degenerates at $\theta^*$
and the log-likelihood becomes weakly concave. 
\begin{figure}[h]
  \begin{center}
    \begin{tabular}{c}
    \widgraph{0.65\textwidth}{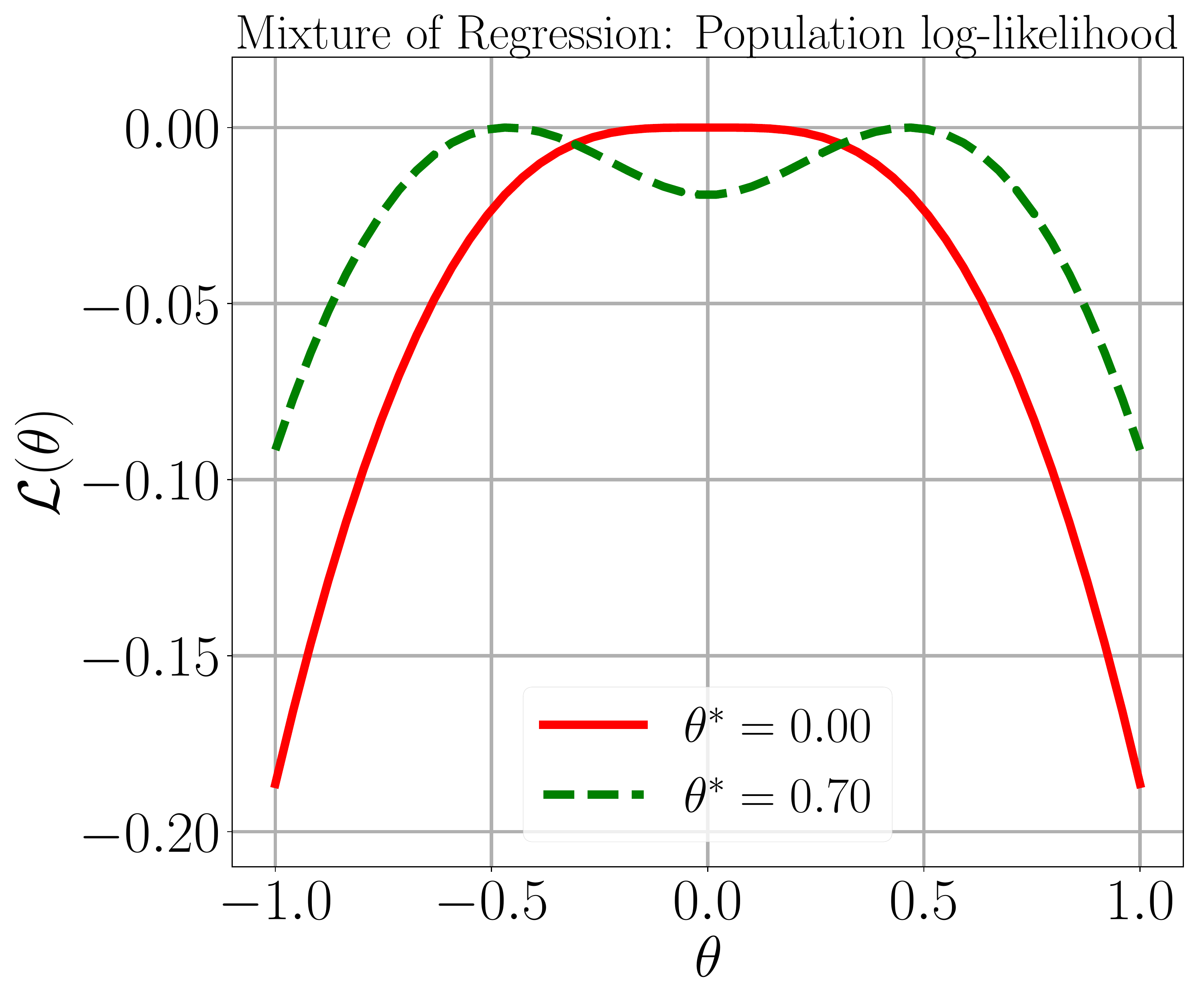} 
    \end{tabular}
    \caption{Plots of the population log-likelihood for the mixture of 
    regression model for $\theta^*\in\braces{0, 0.7}$. We see that while
    the log-likelihood is clearly locally strongly concave around $\theta^*$
    when $\theta^* = 0.7$, and it is rather flat (and weakly concave) for
    the case of no signal $\theta^*=0$. This flatness in log-likelihood
    results in a slower rate of algorithmic and statistical convergence
    of EM in this setting thereby providing further evidence of the usefulness
    of our analysis of EM.
}
\label{fig:mix_of_reg}
  \end{center}
\end{figure}

The behavior observed in Figure~\ref{fig:mix_of_reg} is reminiscent
of the behavior of log-likelihood in the case of over-specified Gaussian
mixtures considered in the main paper (see Appendix~\ref{sec:unbalanced_vs_balanced_fits_closer_look_at_log_likelihood}
and Figure~\ref{fig:pop_likelihood}). We now show that such a similarity
also implies a similar behavior for EM, which converges slowly on both 
algorithmic and statistical fronts (just like the over-specified Gaussian
mixture case) for the fit~\eqref{eq:fitted_regression}.

Given this model, the sample EM operator $\updateMreg_{\obs}:
\realdim \mapsto \realdim$ takes the form
\begin{align}
\label{eq:sample_em_operator_regress}
\updateMreg_\obs(\theta) & \defn \parenth{ \sum_{i = 1}^{n} X_{i} X_{i}^{\top} }^{- 1} \parenth{ \frac{1}{\obs} \sum_{i=1}^n (2
\weightFun_\theta(X_i, Y_{i}) - 1) X_i Y_{i}}
\end{align}
where we define
\begin{align}
  \weightFun_\theta(x, y) \mydefn \frac{\weight \exp
    \parenth{-\frac{\parenth{y - \mean^{\top} x}^2}{2}}}{\weight \exp
    \parenth{- \frac{\parenth{y - \mean^{\top} x}^2}{2}} + (1-\weight) \exp
    \parenth{-\frac{\parenth{y + \mean^{\top} x}^2}{2}}}.
\end{align}
Consequently, the population EM operator
$\updateMreg:
\realdim \mapsto \realdim$ is given by
\begin{align}
\label{eq:pop_EM_mix_regress}
  \updateMreg (\theta) & \mydefn \Exs_{(Y, X)} \brackets{ (2 \weightFun_\theta(X, Y) - 1) X Y},
\end{align}
where the outer expectation is taken with respect to $X \sim \NORMAL(0,
I_{d})$ and $Y|X \sim \NORMAL((\theta^*)^{ \top} X, 1)$ ($= \NORMAL(0,
I_{d})$ when $\thetastar=0$).
Given these notation, we now characterize the slow convergence of the 
population EM operator:
\begin{lemma}
\label{lemma:mix_regress_pop}
Given the balanced model fit~\eqref{eq:fitted_regression} to the 
true model~\eqref{eq:truemodel_regression} with $\thetastar=0$, the population
EM operator $\updateMreg$~\eqref{eq:pop_EM_mix_regress} 
satisfies the following bounds
\begin{align}
  \enorm{ \mean} (1 - 3 \enorm{ \mean}^2) 
  \leq \enorm{ \updateMreg ( \mean)} \leq \enorm{ \mean} (1 
  - 2 \enorm{ \mean}^2)
\end{align}
for any $\mean \in \Rspace^{d}$ such that $\enorm{ \mean} \leq 1/ 2$. 
\end{lemma}
\noindent Proof is deferred to the end of this appendix.\\

We note that the assumption $\enorm{ \mean}  \leq 1/ 2$ is a convenient
technical assumption and is possibly loose in a similar manner as noted
in Lemma~\ref{lemma:bound_initialization} for the Gaussian mixture case.
Applying the localization argument in the paper in conjunction with the
sub-geometric convergence of the population EM (Lemma~\ref{lemma:mix_regress_pop}) 
yields the slow statistical convergence (of order $(d/n)^{\frac 14}$)
of the sample EM:
\begin{corollary}
\label{cor:mix_regression_rate}
Consider the over-specified model fit~\eqref{eq:fitted_regression} to the
true model~\eqref{eq:truemodel_regression} with $\thetastar=0$,
and initialize the sample EM sequence $\mean^{t + 1} = \updateMreg_{\obs}( \mean^{t})$
with a $\theta^0$ such that $\enorm{ \theta^0} \leq \frac12$. 
Then, for any $\varepsilon \in (0, 1/ 4)$, $\delta \in (0, 1)$, given a
large sample size  $n \geq c'_{1} d \log( \log(1 / \varepsilon)/ \delta)$,
the sample EM updates  satisfy
\begin{align*}
  \enorm{ \mean^{t}} \leq \brackets{ \enorm{\theta^{0}} 
  \prod_{j = 0}^{t - 1} \parenth{ 1 - 2 \enorm{ \mean^{j}}^2}} 
  + \sqrt{2}\parenth{\frac{
    (\dims + \log \frac{\log(4 /
        \smallthreshold)}{\delta})}{\obs} }^{\frac{1}{4} - \smallthreshold},
\end{align*}
 for any iterate $t \geq c'_{2} \parenth{\frac{n}{d}}^{\frac{1}{2} 
- 2 \varepsilon} \log (n / d) \log (1 / \varepsilon)$, 
with probability at least $1 - \delta$. Here, $c'_{1}$ and $c'_{2}$ denote
universal constants.
\end{corollary}
Given Lemma~\ref{lemma:mix_regress_pop}, the proof of Corollary~\ref{cor:mix_regression_rate}
follows the same annulus-based localization road-map as of the proof of
Theorem~\ref{theorem:convergence_rate_sample_EM}; and is thereby omitted.
We now prove Lemma~\ref{lemma:mix_regress_pop}.

\paragraph{Proof of Lemma~\ref{lemma:mix_regress_pop}}
We provide a proof sketch for the lemma based on an application of Taylor expansion. 
In particular, we define a transformation $V : = R X$ where 
$R$ is an orthonormal matrix such that $R \mean = \enorm{ \mean} e_{1}$ and 
$e_{1}$ denotes the first canonical basis vector in dimension $d$. 
After similar algebra as that of Theorem~\ref{thm:pop_over}, we can verify that
\begin{align*}
  \enorm{ \updateMreg ( \mean)} = \Exs_{(Y, V_{1})} 
  \brackets{ \tanh(V_{1} Y \enorm{ \mean}) V_{1} Y},
\end{align*}
where the outer expectation is taken with respect to $V_{1}, Y \sim
\NORMAL(0, 1)$ and $V_{1}$ and $Y$ are independent. 
Using arguments similar to the bounds~\eqref{eq:tanh_positive}
and \eqref{eq:tanh_negative}, we can derive that
\begin{align*}
  x^2 - \frac{x^4}{3} \leq \tanh( x) \leq x^2 - \frac{x^4}{3} +
        \frac{2 x^6}{15}
\end{align*}
for all $x \in \Rspace$. Given these bounds, we find that
\begin{align*}
  &\Exs_{(Y, V_{1})} \brackets{ \tanh(V_{1} Y 
  \enorm{ \mean}) V_{1} Y}  \\
  & \leq \Exs \brackets{(V_{1} Y)^2} \enorm{ \mean}  
  - \frac{\Exs \brackets{(V_{1} Y)^4} \enorm{ \mean}^3}{3}
  + \frac{2 \Exs \brackets{(V_{1} Y)^6} \enorm{ \mean}^5}{15} \\
  & = \enorm{ \mean} - 3 \enorm{ \mean}^3 + 30 
  \enorm{ \mean}^5 \leq \enorm {\mean} \parenth{ 1 - 2 \enorm{ \mean}^2},
\end{align*}
and
\begin{align*}
  \Exs_{(Y, V_{1})} \brackets{ \tanh(V_{1} Y \enorm{ \mean})
          V_{1} Y} & \geq \Exs \brackets{(V_{1} Y)^2} \enorm{ \mean} -
        \frac{\Exs \brackets{(V_{1} Y)^4} \enorm{ \mean}^3}{3} \\ & =
        \enorm{ \mean} - 3 \enorm{ \mean}^3 = \enorm{ \mean} \parenth{
          1 - 3 \enorm{ \mean}^2}
\end{align*}
for all $\mean \in \Rspace^{d}$ such that $\enorm{ \mean} \leq 1/ 2$.
Putting the above results together yields the lemma.

\bibliography{Nhat}

\end{document}

%% file: final_macros.tex
\newcommand{\gamup}{\gamma_{\textrm{up}}}
\newcommand{\gamlow}{\gamma_{\textrm{low}}}
\newcommand{\usedim}{\ensuremath{d}}

\newcommand{\gammamin}{\ensuremath{\widetilde{\gamma}}}

\newcommand{\ind}{\ensuremath{\ell}}
\newcommand{\mle}{ \widehat\theta_n^{\text{MLE}}}

\newcommand{\fixedpointtheta}{\widehat\theta_\obs}

\newcommand{\cover}{N}
\newcommand{\kinit}{\nu}

\newcommand{\dndelta}{\omega}

\newcommand{\radii}{\mathcal{R}}
\newcommand{\event}{\mathcal{E}}

\newcommand{\dims}{\ensuremath{d}}
\newcommand{\samples}{\ensuremath{n}}
\newcommand{\obs}{\samples}
\newcommand{\real}{\ensuremath{\mathbb{R}}}
\newcommand{\realdim}{\ensuremath{\real^\dims}}

\newcommand{\smallthreshold}{\epsilon}
\newcommand{\powerr}{\alpha}
\newcommand{\smallradius}{\epsilon}

\newcommand{\noise}{\xi}
\newcommand{\order}[1]{\ensuremath{\mathcal{O}\parenth{#1}}}

\newcommand{\thetastar}{\ensuremath{\theta^*}}

\newcommand{\NORMAL}{\ensuremath{\mathcal{N}}}

\newcommand{\brackets}[1]{\left[ #1 \right]}
\newcommand{\parenth}[1]{\left( #1 \right)}

\newcommand{\braces}[1]{\left\{ #1 \right \}}
\newcommand{\abss}[1]{\left| #1 \right |}

\newcommand{\tp}{^\top}

\newcommand{\myvec}{u} 
\newcommand{\myvectwo}{v} 
\newcommand{\mymat}{B}

\newcommand{\scalar}{\alpha}

\newcommand{\numsteps}{t}
\newcommand{\totalnumsteps}{T}
\newcommand{\seqalpha}[1]{\alpha_{#1}}

\newcommand{\imax}{\ell_\smallthreshold}

\newcommand{\radem}{\varepsilon}
\newcommand{\Tone}{T\ensuremath{_{0}}}

\newcommand{\hidden}{Z}

\newcommand{\weights}{w}

\newcommand{\iter}{t}
\newcommand{\mupdate}{M}

\newcommand{\updateMat}{\ensuremath{\Gamma_{\theta}}}
\newcommand{\mymatTheta}{\ensuremath{F_{\theta}}}
\newcommand{\sech}{\text{sech}}

\newcommand{\Rspace}{\ensuremath{\mathbb{R}}}

\newcommand{\Ncal}{\ensuremath{\mathcal{N}}}
\newcommand{\gradient}{\ensuremath{\nabla}}

\newcommand{\ball}{\ensuremath{\mathbb{B}}}
\newcommand{\sphere}{\ensuremath{\mathbb{S}}}

\newcommand{\weight}{\ensuremath{\uppi}}
\newcommand{\contracasym}{\ensuremath{\rho}}
\newcommand{\Fishermatrix}{\ensuremath{\mathcal{I}}}

\newcommand{\normDensity}{\ensuremath{\phi}}
\newcommand{\mean}{\ensuremath{\theta}}
\newcommand{\sd}{\ensuremath{\sigma}}

\newcommand{\funQ}{\ensuremath{Q}}
\newcommand{\updateM}{\ensuremath{M}}
\newcommand{\updateMreg}{\ensuremath{\overline{M}}}
\newcommand{\weightFun}{\ensuremath{w}}
\newcommand{\mydefn}{\ensuremath{:=}}

\newcommand{\unicon}{c}
\newcommand{\uniconnew}{c'}
\newcommand{\uniconone}{c_1}
\newcommand{\unicontwo}{c_2}

\newcommand{\threshold}{\delta}

\newcommand{\defn}{:=}
\newcommand{\rdefn}{=:}
\newcommand{\etal}{{et al.}}

\newcommand{\matsnorm}[2]{|\!|\!| #1 | \! | \!|_{{#2}}}
\newcommand{\vecnorm}[2]{\left\| #1\right\|_{#2}}
\newcommand{\enorm}[1]{\vecnorm{#1}{2}} 

\newcommand{\opnorm}[1]{\ensuremath{\matsnorm{#1}{\tiny{\mbox{op}}}}}

\newcommand{\Exs}{\ensuremath{{\mathbb{E}}}}
\newcommand{\Prob}{\ensuremath{{\mathbb{P}}}}

\newcommand{\eig}[1]{\ensuremath{\lambda_{#1}}}

\newcommand{\eigmin}{\ensuremath{\eig{\min}}}

\newtheoremstyle{named}{}{}{\itshape}{}{\bfseries}{.}{.5em}{\thmnote{#3's }#1}
\theoremstyle{named}

\theoremstyle{plain}

\newtheorem{theorem}{Theorem}

\newtheorem{lemma}{Lemma}

\newtheorem{corollary}{Corollary}

\newlength{\widebarargwidth}
\newlength{\widebarargheight}
\newlength{\widebarargdepth}

\makeatletter
\long\def\@makecaption#1#2{
        \vskip 0.8ex
        \setbox\@tempboxa\hbox{\small {\bf #1:} #2}
        \parindent 1.5em  
        \dimen0=\hsize
        \advance\dimen0 by -3em
        \ifdim \wd\@tempboxa >\dimen0
                \hbox to \hsize{
                        \parindent 0em
                        \hfil
                        \parbox{\dimen0}{\def\baselinestretch{0.96}\small
                                {\bf #1.} #2
                                }
                        \hfil}
        \else \hbox to \hsize{\hfil \box\@tempboxa \hfil}
        \fi
        }
\makeatother

\long\def\comment#1{}
\definecolor{battleshipgrey}{rgb}{0.52, 0.52, 0.51}
\definecolor{darkgray}{rgb}{0.66, 0.66, 0.66}
\definecolor{darkgreen}{rgb}{0.0, 0.2, 0.13}
\definecolor{darkspringgreen}{rgb}{0.09, 0.45, 0.27}
\definecolor{dukeblue}{rgb}{0.0, 0.0, 0.61}
\definecolor{olivedrab7}{rgb}{0.24, 0.2, 0.12}
\definecolor{darkblue}{rgb}{0.0, 0.0, 0.55}
\definecolor{darkscarlet}{rgb}{0.34, 0.01, 0.1}
\definecolor{candyapplered}{rgb}{1.0, 0.03, 0.0}
\definecolor{ao(english)}{rgb}{0.0, 0.5, 0.0}
\definecolor{applegreen}{rgb}{0.55, 0.71, 0.0}

\newcommand{\widgraph}[2]{\includegraphics[keepaspectratio,width=#1]{#2}}

%% file: arXiv.bbl
\begin{thebibliography}{38}

\bibitem[\protect\citeauthoryear{Balakrishnan, Wainwright and
  Yu}{2017}]{Siva_2017}
\begin{barticle}[author]
\bauthor{\bsnm{Balakrishnan},~\bfnm{S.}\binits{S.}},
  \bauthor{\bsnm{Wainwright},~\bfnm{M.~J.}\binits{M.~J.}} \AND
  \bauthor{\bsnm{Yu},~\bfnm{B.}\binits{B.}}
(\byear{2017}).
\btitle{Statistical guarantees for the {EM} algorithm: From population to
  sample-based analysis}.
\bjournal{The Annals of Statistics}
\bvolume{45}
\bpages{77-120}.
\end{barticle}
\endbibitem

\bibitem[\protect\citeauthoryear{Bartlett, Bousquet and
  Mendelson}{2005}]{Bar05}
\begin{barticle}[author]
\bauthor{\bsnm{Bartlett},~\bfnm{P.~L.}\binits{P.~L.}},
  \bauthor{\bsnm{Bousquet},~\bfnm{O.}\binits{O.}} \AND
  \bauthor{\bsnm{Mendelson},~\bfnm{S.}\binits{S.}}
(\byear{2005}).
\btitle{Local {R}ademacher complexities}.
\bjournal{The Annals of Statistics}
\bvolume{33}
\bpages{1497--1537}.
\end{barticle}
\endbibitem

\bibitem[\protect\citeauthoryear{Cai et~al.}{2019}]{Cai_2018}
\begin{barticle}[author]
\bauthor{\bsnm{Cai},~\bfnm{T~Tony}\binits{T.~T.}},
  \bauthor{\bsnm{Ma},~\bfnm{Jing}\binits{J.}},
  \bauthor{\bsnm{Zhang},~\bfnm{Linjun}\binits{L.}} \betal{et~al.}
(\byear{2019}).
\btitle{CHIME: Clustering of high-dimensional Gaussian mixtures with EM
  algorithm and its optimality}.
\bjournal{The Annals of Statistics}
\bvolume{47}
\bpages{1234--1267}.
\end{barticle}
\endbibitem

\bibitem[\protect\citeauthoryear{Chen}{1995}]{Chen1992}
\begin{barticle}[author]
\bauthor{\bsnm{Chen},~\bfnm{J.}\binits{J.}}
(\byear{1995}).
\btitle{Optimal rate of convergence for finite mixture models}.
\bjournal{The Annals of Statistics}
\bvolume{23}
\bpages{221-233}.
\end{barticle}
\endbibitem

\bibitem[\protect\citeauthoryear{Chen}{2018}]{ychen_2018}
\begin{barticle}[author]
\bauthor{\bsnm{Chen},~\bfnm{Y.~C.}\binits{Y.~C.}}
(\byear{2018}).
\btitle{Statistical inference with local optima}.
\bjournal{arXiv preprint arXiv:1807.04431}.
\end{barticle}
\endbibitem

\bibitem[\protect\citeauthoryear{Chen and Li}{2009}]{Chen_2009}
\begin{barticle}[author]
\bauthor{\bsnm{Chen},~\bfnm{J.}\binits{J.}} \AND
  \bauthor{\bsnm{Li},~\bfnm{P.}\binits{P.}}
(\byear{2009}).
\btitle{Hypothesis test for normal mixture models: the {EM} approach}.
\bjournal{The Annals of Statistics}
\bvolume{37}
\bpages{2523-2542}.
\end{barticle}
\endbibitem

\bibitem[\protect\citeauthoryear{Daskalakis, Tzamos and
  Zampetakis}{2017}]{Daskalakis_colt2017}
\begin{binproceedings}[author]
\bauthor{\bsnm{Daskalakis},~\bfnm{C.}\binits{C.}},
  \bauthor{\bsnm{Tzamos},~\bfnm{C.}\binits{C.}} \AND
  \bauthor{\bsnm{Zampetakis},~\bfnm{M.}\binits{M.}}
(\byear{2017}).
\btitle{Ten steps of {EM} suffice for mixtures of two {G}aussians}.
In \bbooktitle{Proceedings of the 2017 Conference on Learning Theory}.
\end{binproceedings}
\endbibitem

\bibitem[\protect\citeauthoryear{Dempster, Laird and Rubin}{1997}]{Rubin-1977}
\begin{barticle}[author]
\bauthor{\bsnm{Dempster},~\bfnm{A.~P.}\binits{A.~P.}},
  \bauthor{\bsnm{Laird},~\bfnm{N.~M.}\binits{N.~M.}} \AND
  \bauthor{\bsnm{Rubin},~\bfnm{D.~B.}\binits{D.~B.}}
(\byear{1997}).
\btitle{Maximum likelihood from incomplete data via the {EM} algorithm}.
\bjournal{Journal of the Royal Statistical Society: Series B (Statistical
  Methodology)}
\bvolume{39}
\bpages{1-38}.
\end{barticle}
\endbibitem

\bibitem[\protect\citeauthoryear{Dwivedi et~al.}{2018}]{Raaz-misspecified}
\begin{binproceedings}[author]
\bauthor{\bsnm{Dwivedi},~\bfnm{R.}\binits{R.}},
  \bauthor{\bsnm{Ho},~\bfnm{N.}\binits{N.}},
  \bauthor{\bsnm{Khamaru},~\bfnm{K.}\binits{K.}},
  \bauthor{\bsnm{Wainwright},~\bfnm{M.~J.}\binits{M.~J.}} \AND
  \bauthor{\bsnm{Jordan},~\bfnm{M.~I.}\binits{M.~I.}}
(\byear{2018}).
\btitle{Theoretical guarantees for {EM} under misspecified {G}aussian mixture
  models}.
In \bbooktitle{NeurIPS 31}.
\end{binproceedings}
\endbibitem

\bibitem[\protect\citeauthoryear{Dwivedi et~al.}{2019}]{dwivedi2019challenges}
\begin{barticle}[author]
\bauthor{\bsnm{Dwivedi},~\bfnm{Raaz}\binits{R.}},
  \bauthor{\bsnm{Ho},~\bfnm{Nhat}\binits{N.}},
  \bauthor{\bsnm{Khamaru},~\bfnm{Koulik}\binits{K.}},
  \bauthor{\bsnm{Wainwright},~\bfnm{Martin~J}\binits{M.~J.}},
  \bauthor{\bsnm{Jordan},~\bfnm{Michael~I}\binits{M.~I.}} \AND
  \bauthor{\bsnm{Yu},~\bfnm{Bin}\binits{B.}}
(\byear{2019}).
\btitle{Challenges with EM in application to weakly identifiable mixture
  models}.
\bjournal{arXiv preprint arXiv:1902.00194}.
\end{barticle}
\endbibitem

\bibitem[\protect\citeauthoryear{Ghosal and van~der Vaart}{2001}]{Ghosal-2001}
\begin{barticle}[author]
\bauthor{\bsnm{Ghosal},~\bfnm{S.}\binits{S.}} \AND \bauthor{\bparticle{van~der}
  \bsnm{Vaart},~\bfnm{A.}\binits{A.}}
(\byear{2001}).
\btitle{Entropies and rates of convergence for maximum likelihood and {B}ayes
  estimation for mixtures of normal densities}.
\bjournal{The Annals of Statistics}
\bvolume{29}
\bpages{1233-1263}.
\end{barticle}
\endbibitem

\bibitem[\protect\citeauthoryear{Hao et~al.}{2018}]{Cheng_2018}
\begin{barticle}[author]
\bauthor{\bsnm{Hao},~\bfnm{B.}\binits{B.}},
  \bauthor{\bsnm{Sun},~\bfnm{W.}\binits{W.}},
  \bauthor{\bsnm{Liu},~\bfnm{Y.}\binits{Y.}} \AND
  \bauthor{\bsnm{Cheng},~\bfnm{G.}\binits{G.}}
(\byear{2018}).
\btitle{Simultaneous clustering and estimation of heterogeneous graphical
  models}.
\bjournal{Journal of Machine Learning Research}
\bvolume{18}
\bpages{1 - 58}.
\end{barticle}
\endbibitem

\bibitem[\protect\citeauthoryear{Heinrich and Kahn}{2018}]{Jonas-2016}
\begin{barticle}[author]
\bauthor{\bsnm{Heinrich},~\bfnm{P.}\binits{P.}} \AND
  \bauthor{\bsnm{Kahn},~\bfnm{J.}\binits{J.}}
(\byear{2018}).
\btitle{Strong identifiability and optimal minimax rates for finite mixture
  estimation}.
\bjournal{The Annals of Statistics}
\bvolume{46}
\bpages{2844-2870}.
\end{barticle}
\endbibitem

\bibitem[\protect\citeauthoryear{Inglot}{2010}]{inglot2010inequalities}
\begin{barticle}[author]
\bauthor{\bsnm{Inglot},~\bfnm{T.}\binits{T.}}
(\byear{2010}).
\btitle{Inequalities for quantiles of the chi-square distribution}.
\bjournal{Probability and Mathematical Statistics}
\bvolume{30}
\bpages{339--351}.
\end{barticle}
\endbibitem

\bibitem[\protect\citeauthoryear{Ishwaran, James and Sun}{2001}]{Ishwaran-2001}
\begin{barticle}[author]
\bauthor{\bsnm{Ishwaran},~\bfnm{H.}\binits{H.}},
  \bauthor{\bsnm{James},~\bfnm{L.~F.}\binits{L.~F.}} \AND
  \bauthor{\bsnm{Sun},~\bfnm{J.}\binits{J.}}
(\byear{2001}).
\btitle{{B}ayesian model selection in finite mixtures by marginal density
  decompositions}.
\bjournal{Journal of the American Statistical Association}
\bvolume{96}
\bpages{1316-1332}.
\end{barticle}
\endbibitem

\bibitem[\protect\citeauthoryear{Janson}{1997}]{janson1997gaussian}
\begin{bbook}[author]
\bauthor{\bsnm{Janson},~\bfnm{S.}\binits{S.}}
(\byear{1997}).
\btitle{{G}aussian {H}ilbert {S}paces}
\bvolume{129}.
\bpublisher{Cambridge University Press}.
\end{bbook}
\endbibitem

\bibitem[\protect\citeauthoryear{Klusowski, Yang and
  Brinda}{2019}]{klusowski_2017}
\begin{barticle}[author]
\bauthor{\bsnm{Klusowski},~\bfnm{Jason~M}\binits{J.~M.}},
  \bauthor{\bsnm{Yang},~\bfnm{Dana}\binits{D.}} \AND
  \bauthor{\bsnm{Brinda},~\bfnm{WD}\binits{W.}}
(\byear{2019}).
\btitle{Estimating the coefficients of a mixture of two linear regressions by
  expectation maximization}.
\bjournal{IEEE Transactions on Information Theory}
\bvolume{65}
\bpages{3515--3524}.
\end{barticle}
\endbibitem

\bibitem[\protect\citeauthoryear{Koltchinskii}{2006}]{Kolt06}
\begin{barticle}[author]
\bauthor{\bsnm{Koltchinskii},~\bfnm{V.}\binits{V.}}
(\byear{2006}).
\btitle{Local {R}ademacher complexities and oracle inequalities in risk
  minimization}.
\bjournal{The Annals of Statistics}
\bvolume{34}
\bpages{2593--2656}.
\end{barticle}
\endbibitem

\bibitem[\protect\citeauthoryear{Kumar and Schmidt}{2017}]{kumar_nips2017}
\begin{barticle}[author]
\bauthor{\bsnm{Kumar},~\bfnm{R.}\binits{R.}} \AND
  \bauthor{\bsnm{Schmidt},~\bfnm{M.}\binits{M.}}
(\byear{2017}).
\btitle{Convergence rate of expectation-maximization}.
\bjournal{10th NIPS Workshop on Optimization for Machine Learning}.
\end{barticle}
\endbibitem

\bibitem[\protect\citeauthoryear{Ledoux and
  Talagrand}{1991}]{Ledoux_Talagrand_1991}
\begin{bbook}[author]
\bauthor{\bsnm{Ledoux},~\bfnm{M.}\binits{M.}} \AND
  \bauthor{\bsnm{Talagrand},~\bfnm{M.}\binits{M.}}
(\byear{1991}).
\btitle{Probability in {B}anach {S}paces: {I}soperimetry and {P}rocesses}.
\bpublisher{Springer-Verlag, New York, NY}.
\end{bbook}
\endbibitem

\bibitem[\protect\citeauthoryear{Li, Chen and Marriott}{2009}]{Chen_2008}
\begin{barticle}[author]
\bauthor{\bsnm{Li},~\bfnm{P.}\binits{P.}},
  \bauthor{\bsnm{Chen},~\bfnm{J.}\binits{J.}} \AND
  \bauthor{\bsnm{Marriott},~\bfnm{P.}\binits{P.}}
(\byear{2009}).
\btitle{Non-finite {F}isher information and homogeneity: an {EM} approach}.
\bjournal{Biometrika}
\bvolume{96}
\bpages{411-426}.
\end{barticle}
\endbibitem

\bibitem[\protect\citeauthoryear{Ma, Xu and Jordan}{2000}]{Jordan-2000}
\begin{barticle}[author]
\bauthor{\bsnm{Ma},~\bfnm{J.}\binits{J.}},
  \bauthor{\bsnm{Xu},~\bfnm{L.}\binits{L.}} \AND
  \bauthor{\bsnm{Jordan},~\bfnm{M.~I.}\binits{M.~I.}}
(\byear{2000}).
\btitle{Asymptotic convergence rate of the {EM} algorithm for {G}aussian
  mixtures}.
\bjournal{Neural Computation}
\bvolume{12}
\bpages{2881-2907}.
\end{barticle}
\endbibitem

\bibitem[\protect\citeauthoryear{Nguyen}{2013}]{Nguyen-13}
\begin{barticle}[author]
\bauthor{\bsnm{Nguyen},~\bfnm{X.}\binits{X.}}
(\byear{2013}).
\btitle{Convergence of latent mixing measures in finite and infinite mixture
  models}.
\bjournal{The Annals of Statistics}
\bvolume{4}
\bpages{370-400}.
\end{barticle}
\endbibitem

\bibitem[\protect\citeauthoryear{Redner and Walker}{1984}]{redner1984mixture}
\begin{barticle}[author]
\bauthor{\bsnm{Redner},~\bfnm{Richard~A}\binits{R.~A.}} \AND
  \bauthor{\bsnm{Walker},~\bfnm{Homer~F}\binits{H.~F.}}
(\byear{1984}).
\btitle{Mixture densities, maximum likelihood and the EM algorithm}.
\bjournal{SIAM review}
\bvolume{26}
\bpages{195--239}.
\end{barticle}
\endbibitem

\bibitem[\protect\citeauthoryear{Richardson and Green}{1997}]{Green_JRSSB-97}
\begin{barticle}[author]
\bauthor{\bsnm{Richardson},~\bfnm{S.}\binits{S.}} \AND
  \bauthor{\bsnm{Green},~\bfnm{P.~J.}\binits{P.~J.}}
(\byear{1997}).
\btitle{On {B}ayesian analysis of mixtures with an unknown number of
  components}.
\bjournal{Journal of the Royal Statistical Society: Series B (Statistical
  Methodology)}
\bvolume{59}
\bpages{731-792}.
\end{barticle}
\endbibitem

\bibitem[\protect\citeauthoryear{Rousseau and Mengersen}{2011}]{Rousseau-2011}
\begin{barticle}[author]
\bauthor{\bsnm{Rousseau},~\bfnm{J.}\binits{J.}} \AND
  \bauthor{\bsnm{Mengersen},~\bfnm{K.}\binits{K.}}
(\byear{2011}).
\btitle{Asymptotic behaviour of the posterior distribution in overfitted
  mixture models}.
\bjournal{Journal of the Royal Statistical Society: Series B (Statistical
  Methodology)}
\bvolume{73}
\bpages{689-710}.
\end{barticle}
\endbibitem

\bibitem[\protect\citeauthoryear{Stephens}{2002}]{Stephens-2002}
\begin{barticle}[author]
\bauthor{\bsnm{Stephens},~\bfnm{M.}\binits{M.}}
(\byear{2002}).
\btitle{Dealing with label switching in mixture models}.
\bjournal{Journal of the Royal Statistical Society: Series B (Statistical
  Methodology)}
\bvolume{62}
\bpages{795-809}.
\end{barticle}
\endbibitem

\bibitem[\protect\citeauthoryear{van~de Geer}{2000}]{Vandegeer-2000}
\begin{bbook}[author]
\bauthor{\bparticle{van~de} \bsnm{Geer},~\bfnm{S.}\binits{S.}}
(\byear{2000}).
\btitle{Empirical Processes in M-estimation}.
\bpublisher{Cambridge University Press}.
\end{bbook}
\endbibitem

\bibitem[\protect\citeauthoryear{van~der Vaart}{1998}]{vanderVaart-98}
\begin{bbook}[author]
\bauthor{\bparticle{van~der} \bsnm{Vaart},~\bfnm{A.~W.}\binits{A.~W.}}
(\byear{1998}).
\btitle{Asymptotic Statistics}.
\bpublisher{Cambridge University Press}.
\end{bbook}
\endbibitem

\bibitem[\protect\citeauthoryear{van~der Vaart and
  Wellner}{2000}]{Vaart_Wellner_2000}
\begin{bbook}[author]
\bauthor{\bparticle{van~der} \bsnm{Vaart},~\bfnm{A.~W.}\binits{A.~W.}} \AND
  \bauthor{\bsnm{Wellner},~\bfnm{J.~A.}\binits{J.~A.}}
(\byear{2000}).
\btitle{Weak {C}onvergence and {E}mpirical {P}rocesses: {W}ith {A}pplications
  to {S}tatistics}.
\bpublisher{Springer-Verlag, New York, NY}.
\end{bbook}
\endbibitem

\bibitem[\protect\citeauthoryear{Vershynin}{2011}]{Vershynin_2011}
\begin{barticle}[author]
\bauthor{\bsnm{Vershynin},~\bfnm{R.}\binits{R.}}
(\byear{2011}).
\btitle{Introduction to the non-asymptotic analysis of random matrices}.
\bjournal{arXiv:1011.3027v7}.
\end{barticle}
\endbibitem

\bibitem[\protect\citeauthoryear{Wainwright}{2019}]{Wai19}
\begin{bbook}[author]
\bauthor{\bsnm{Wainwright},~\bfnm{M.~J.}\binits{M.~J.}}
(\byear{2019}).
\btitle{High-dimensional statistics: A non-asymptotic viewpoint}.
\bpublisher{Cambridge University Press}, \baddress{Cambridge, {U}{K}}.
\end{bbook}
\endbibitem

\bibitem[\protect\citeauthoryear{Wang et~al.}{2015}]{HanLiu_nips2015}
\begin{binproceedings}[author]
\bauthor{\bsnm{Wang},~\bfnm{Z.}\binits{Z.}},
  \bauthor{\bsnm{Gu},~\bfnm{Q.}\binits{Q.}},
  \bauthor{\bsnm{Ning},~\bfnm{Y.}\binits{Y.}} \AND
  \bauthor{\bsnm{Liu},~\bfnm{H.}\binits{H.}}
(\byear{2015}).
\btitle{High-dimensional expectation-maximization algorithm: {S}tatistical
  optimization and asymptotic normality}.
In \bbooktitle{Advances in Neural Information Processing Systems 28}.
\end{binproceedings}
\endbibitem

\bibitem[\protect\citeauthoryear{Wu}{1983}]{Jeff_Wu-1983}
\begin{barticle}[author]
\bauthor{\bsnm{Wu},~\bfnm{C.~F.~Jeff}\binits{C.~F.~J.}}
(\byear{1983}).
\btitle{On the convergence properties of the {EM} algorithm}.
\bjournal{The Annals of Statistics}
\bvolume{11}
\bpages{95-103}.
\end{barticle}
\endbibitem

\bibitem[\protect\citeauthoryear{Xu, Hsu and Maleki}{2016}]{Hsu-nips2016}
\begin{binproceedings}[author]
\bauthor{\bsnm{Xu},~\bfnm{J.}\binits{J.}},
  \bauthor{\bsnm{Hsu},~\bfnm{D.}\binits{D.}} \AND
  \bauthor{\bsnm{Maleki},~\bfnm{A.}\binits{A.}}
(\byear{2016}).
\btitle{Global analysis of expectation maximization for mixtures of two
  {G}aussians}.
In \bbooktitle{Advances in Neural Information Processing Systems 29}.
\end{binproceedings}
\endbibitem

\bibitem[\protect\citeauthoryear{Xu and Jordan}{1996}]{Jordan-1996}
\begin{barticle}[author]
\bauthor{\bsnm{Xu},~\bfnm{L.}\binits{L.}} \AND
  \bauthor{\bsnm{Jordan},~\bfnm{M.~I.}\binits{M.~I.}}
(\byear{1996}).
\btitle{On convergence properties of the {EM} Algorithm for {G}aussian
  mixtures}.
\bjournal{Neural Computation}
\bvolume{8}
\bpages{129-151}.
\end{barticle}
\endbibitem

\bibitem[\protect\citeauthoryear{Yan, Yin and Sarkar}{2017}]{Sarkar_nips2017}
\begin{binproceedings}[author]
\bauthor{\bsnm{Yan},~\bfnm{B.}\binits{B.}},
  \bauthor{\bsnm{Yin},~\bfnm{M.}\binits{M.}} \AND
  \bauthor{\bsnm{Sarkar},~\bfnm{P.}\binits{P.}}
(\byear{2017}).
\btitle{Convergence of gradient {EM} on multi-component mixture of
  {G}aussians}.
In \bbooktitle{Advances in Neural Information Processing Systems 30}.
\end{binproceedings}
\endbibitem

\bibitem[\protect\citeauthoryear{Yi and Caramanis}{2015}]{Caramanis-nips2015}
\begin{binproceedings}[author]
\bauthor{\bsnm{Yi},~\bfnm{X.}\binits{X.}} \AND
  \bauthor{\bsnm{Caramanis},~\bfnm{C.}\binits{C.}}
(\byear{2015}).
\btitle{Regularized {EM} algorithms: a unified framework and statistical
  guarantees}.
In \bbooktitle{Advances in Neural Information Processing Systems 28}.
\end{binproceedings}
\endbibitem

\end{thebibliography}
